\documentclass[12pt]{amsart}
\usepackage{amsfonts,amsthm,amsmath,amssymb,amscd}
\usepackage{graphicx}
\usepackage{color}
\usepackage[backref,bookmarks=false]{hyperref}
\usepackage{pst-node}
\usepackage{tikz-cd} 
\usepackage{enumerate}
\theoremstyle{definition}
\newtheorem{dfn}{Definition}
\newtheorem{rem}{Remark}
\newtheorem*{fixing*}{Fixing the constants}
\theoremstyle{plain}
\newtheorem{lem}[dfn]{Lemma}
\newtheorem{thm}[dfn]{Theorem}
\newtheorem{prop}[dfn]{Proposition}

\newtheorem{cor}[dfn]{Corollary}

\newtheorem*{Condition $W_n$}{Condition $W_n$}
\newtheorem*{thm*}{Theorem}

\newtheorem{obs}[dfn]{Observation}

\def\C{{\mathbb C}}
\def\Z{{\mathbb Z}}
\def\N{{\mathbb N}}
\def\R{{\mathbb R}}
\def\Re{{\rm Re}}
\def\Im{{\rm Im}}
\def\diam{{\rm diam}}
\def\d{{\rm d}}
\def\P{{\mathcal P}}
\def\L{{\mathcal L}}
\def\hL{{\hat{\mathcal L}}}
\def\supp{{\rm supp}}
\def\dist{{\rm dist}}
\def\HD{{\rm HD}}
\def\sbt{\subset}
\def\spt{\supset}
\def\sms{\setminus}
\def\es{\emptyset}
\def\Pr{{\rm P}}

\def\om{\omega} \def\Om{\Omega}
\def\th{\theta}
\def\b{\beta}
\def\alp{\alpha}
\def\lam{\lambda}
\def\La{\Lambda}
\def\ve{\varepsilon}
 \def\Ga{\Gamma}
\def\Sg{\Sigma}
\def\ba{\bf{a}}
\def\({\big(}
\def\){\big)}
\def\lt{\left}
\def\rt{\right}
\def\bu{\bigcup}
\def\bi{\bigcap}

\def\lra{\longrightarrow}

\def\un{\underline}
\def\ov{\overline}
\def\ld{\ldots}

\numberwithin{equation}{section}

\newcommand{\1}{{\sf 1\hspace*{-0.8ex}\sf 1 }}
\title{Random non-hyperbolic exponential maps}
\author{Mariusz Urba\'nski}
\author{Anna Zdunik}
\thanks{M. Urba\'nski was supported in part by the NSF Grant DMS 1361677. A. Zdunik was supported in part by the Grant NCN grant 2014/13/B/ST1/04551.}
\begin{document}
\maketitle

\begin{abstract}
We consider random iteration of exponential entire functions, i.e. of the form $\C\ni z\mapsto f_\lam(z):=\lam e^z\in\C$, $\lam\in\C\sms \{0\}$. Assuming that $\lam$ is in a bounded closed interval $[A,B]$ with $A>1/e$, we deal with random iteration of the maps $f_\lam$ governed by an invertible measurable map $\theta:\Om\to\Om$ preserving a probability ergodic measure $m$ on $\Om$, where $\Om$ is a measurable space. The link from $\Om$ to exponential maps is then given by an arbitrary measurable function $\eta:\Omega\longmapsto [A,B]$. 
We in fact work on the cylinder space $Q:=\C/\sim$, where $\sim$ is the natural equivalence relation: $z\sim w$ if and only if $w-z$ is an integral multiple of $2\pi i$. We prove that then for every $t>1$ there exists a unique random conformal measure 
$\nu^{(t)}$ for the random conformal dynamical system on $Q$. We further prove that this measure is supported on the, appropriately defined, radial Julia set. Next, we show that there exists a unique random probability invariant measure $\mu^{(t)}$ absolutely continuous with respect to $\mu^{(t)}$. In fact $\mu^{(t)}$ is equivalent with $\nu^{(t)}$. Then we turn to geometry. We define an expected topological pressure $\mathcal E \Pr(t)\in\R$ and show that its only zero $h$ coincides with the Hausdorff dimension of $m$--almost every fiber radial Julia set $J_r(\omega)\sbt Q$, $\om\in\Om$. We show that $h\in (1,2)$ and that the omega--limit set of Lebesgue almost every point in $Q$ is contained in the real line $\R$. Finally, we entirely transfer our results to the original random dynamical system on $\C$. As our preliminary result, we show that all fiber Julia sets coincide with the entire complex plane $\C$.
\end{abstract}

\tableofcontents

\section{Introduction}

The study of the dynamics of entire transcendental functions of the complex plane has begun with the foundational research of Pierre Fatou (\cite{Fatou}) in the third decade of 20th century. For some decades since then I. N. Baker (\cite{Baker1}, \cite{Baker2}and \cite{Baker3} for example), was the sole champion of the research in this field. The breakthrough has come in 1981 when Michal Misiurewicz (\cite{mi}) proved that the Julia set of the exponential function $\C\ni z\mapsto e^z\in\C$ is the whole complex plane $\C$. This positively affirmed Fatou's Conjecture from \cite{Fatou} and opened up the gates for new extensive research. Indeed, the early papers such as \cite{Rees}, \cite{Lyubich}, \cite{McMullen}, \cite{DevaneyKrych}, \cite{EremenkoLyubich} have appeared. These concerned topological and measurable (Lebesgue) aspects of the dynamics of entire functions. However, already McMullen's paper \cite{McMullen} also touched on Hausdorff dimension, providing deep and unexpected results. The theme of Hausdorff dimension for entire 
functions was taken up in a series of paper s by G. Stallard (see for ex. \cite{Stallard1}--\cite{Stallard5} and in \cite{UZ1} and \cite{UZ2}. These two latter papers concerned hyperbolic exponential functions, i.e. those of the form 
$$
\C\ni z\longmapsto f_\lambda(z):=\lambda e^z\in\C,
$$
where $\lambda$ is such that the map $f_\lam$ is hyperbolic, i.e. it has an attracting periodic orbit. Although it did not concern entire functions but meromorophic ones (tangent family in fact), we would like to mention here the seminal paper of K. Bara\'nski (\cite{Baranski}) where for the first time the thermodynamic formalism was applied to study transcendental functions. The papers \cite{UZ1} and \cite{UZ2} also used the ideas of thermodynamic formalism and, particularly, of conformal measures. This is in these papers where the concept of a radial (called also conical) Julia set, denoted by $J_r(f)$ occurred. This is the set of points $z$ in the Julia set $J(f)$ for which infinitely many holomorphic pullbacks from $f^n(z)$ to $z$ are defined on balls centered at points $f^n(z)$ and having radii larger than zero independently of $n$. For hyperbolic functions $f_\lam$ this is just the set of points that do not escape to infinity under the action of the map $f_\lam$. What we have discovered in \cite{UZ1} and \cite{UZ2} is that $\HD(J_r(f_\lam))<2$ for hyperbolic exponential functions $f_\lam$ defined above. This is in stark contrast with McMullen's results from \cite{McMullen} asserting that $\HD(J_r(f_\lam))=2$ for all $\lam\in\C\sms \{0\}$. Note that the set $J_r(f_\lam)$ is dynamically significant as for example, because of Poincar\'e's Recurrence Theorem, every finite Borel $f_\lam$--invariant measure on $\C$ is supported on this set. In addition we proved in \cite{UZ1} and \cite{UZ2} that its Hausdorff dimension $\HD(J_r(f_\lam))$ is equal to the unique zero of the pressure function $t\mapsto \Pr(t)$ defined absolutely independently of $J_r(f_\lam)$. 

The study of geometric (Hausdorff dimension) and ergodic  (invariant measures absolutely continuous with respect to conformal ones) properties of transcendental entire functions by means of thermodynamic formalism followed. It is impossible to cite here all of them, we just mention only \cite{KotusU}, \cite{MayerU_1}, \cite{MayerU_2}, \cite{MayerU_3}, \cite{BKZdunik1} and \cite{BKZdunik2}. Related papers include \cite{Rempe}, \cite{Bergweiler}, \cite{BaranskiFagella} and many more. 

We would like to pay particular attention to the paper \cite{uz_nonhyp}, where a fairly full account of ergodic theory and conformal measures was provided for a large class of non-hyperbolic exponential functions $f_\lam$, namely those for which the number $0$ escapes to infinity ``really fast''; it includes all maps for which $\lam$ is real and larger than $1/e$. Our current work stems from this one and provides a systematic account of ergodic theory and conformal measures for randomly iterated functions $f_\lam$, where $\lam>1/e$. The theory of random dynamical systems is a large fast developing subfield of dynamical systems with a specific variety of methods, tools, and goals. We just mention the classical works of Yuri Kifer, \cite{Kifer1}, \cite{Kifer2} and of Ludwig Arnold (\cite{Arnold}), see also \cite{Kifer3}. Our present work in this respect stems from \cite{MSU}, \cite{MayerU4}, and \cite{MayerU5}. 

\medskip Our first result, whose proof occupies Section~\ref{sec:julia}, is that if a sequence $(a_n)_{n=1}^\infty$ of real numbers in $[A,B]$ is taken, where $A>1/e$, then the Julia set of the compositions
$$
f_{a_n}\circ f_{a_{n-1}}\circ\ldots\circ f_{a_2}\circ f_{a_1}:\C\to\C
$$
is equal to the entire complex plane $\C$. Tis is a far going generalization of, already mentioned above, Misiurewicz's result from \cite{mi}. In addition, our proof is a simplification of Misiurewicz's one, even in the autonomous (just one map) case. The above compositions are said to constitute a non--autonomous dynamical system as the ``rule of evolution'' depends on time.

\medskip Our main focus in this paper are random dynamical systems. These are objects lying somewhat in between autonomous and non--autonomous systems, sharing many dynamical and geometrical features with both of them. As in \cite{Arnold}, \cite{crauel}, \cite{MSU}, \cite{MayerU_3}, and \cite{MayerU4} the randomness for us is modeled by a measure preserving invertible 
dynamical system $\theta:\Omega\to\Omega$, where $(\Omega,\mathcal F, m)$ is a complete probability measurable space, and $\theta$ is a measurable invertible map, with $\theta^{-1}$ measurable, preserving the measure $m$.
Fix some real constants $B>A>1/e$ and let 
$$
\eta:\Omega\longmapsto [A,B]
$$
be measurable function. Furthermore, to each $\omega\in\Omega$ associated is the exponential map $f_\omega:=f_{\eta(\omega)}:\C\to\C$; precisely
$$
f_\omega(z):=\eta(\om)e^z.
$$ 
Consequently, for every $z\in \C$, the map 
$$
\Omega\ni\omega\longmapsto f_{\eta(\omega)}(z)\in\C
$$ 
is measurable.

We consider the dynamics of random iterates of exponentials:
$$
f^n_\omega
:=f_{\theta^{n-1}\omega}\circ\cdots\circ f_{\theta\omega}\circ f_\omega:\C\longrightarrow\C.
$$
The sextuple
$$
f:=\(\Omega,\mathcal F,m;\, \theta:\Omega\to \Omega;\, \eta:\Om\to[A,B]; f_\eta:\C\to\C\)
$$
and induced by it random dynamics 
$$
\(f^n_\omega:\C\longrightarrow\C\)_{n=0}^\infty, \  \  \om\in \Om,
$$
will be referred to in the sequel as \emph{random exponential dynamical system}. Following \cite{crauel} we consider random measures (with respect to the measure $m$). 

\medskip We define the equivalence relation $\sim$ on the complex plane $\C$ by saying that $Z\sim W$ if there exists $k\in\mathbb Z$ such that 
$$
Z-W=2\pi i k.
$$
We denote the quotient space $\C/\sim$ by $Q$. So, $Q$ is conformally an infinite cylinder. We denote by $\pi$ the natural projection $\pi:\C\to Q$, i.e., 
$$
\pi(Z)=[Z]
$$ 
is the equivalence class of $z$ with respect to relation $\sim$.
Since both maps $f_\eta:\C\to \C$ and $\pi\circ f_\eta:\C\to Q$, $\eta\in\C^*$, are constant on equivalence classes, they canonically induce conformal maps $f_\eta:Q\to \C$ and 
$$
F_\eta:Q\to Q.
$$
So, $F_\eta$ can be represented as 
$$
F_\eta=\pi\circ f_\eta\circ \pi^{-1}.
$$
Throughout the whole paper in Sections 4--11 we will be exclusively interested in the sextuple
$$
F:=\(\Omega,\mathcal F,m;\, \theta:\Omega\to \Omega;\, \eta:\Om\to[A,B];F_\eta:Q\to Q\)
$$
and induced by it random dynamics 
$$
\(F^n_\omega:=F_{\theta^{n-1}\omega}\circ\cdots\circ F_{\theta\omega}\circ F_\omega:Q\longrightarrow Q\)_{n=0}^\infty, \  \  \om\in \Om.
$$
All our technical work and in Sections  4--11 will concern the sextuple $F$ acting on the cylinder $Q$. The main results for this sextuple, truly interesting on their own, will be obtained in will be obtained in Sections 8--11. In Sections 12 and 13 we fully transfer them for the case of
sextuple
$$
f:=\(\Omega,\mathcal F,m;\, \theta:\Omega\to \Omega;\, \eta:\Om\to[A,B]; f_\eta:\C\to\C\)
$$
and induced by it random dynamics. 

We now describe our results for the sextuple $F$. 
Let $X=\Omega\times\C$ and let
$$
\pi_1:X\to \Omega
$$
be the projection onto the first coordinate, i.e.,
$$
\pi_1(\om,z)=\om.
$$
Let $\mathcal {M}_m\subset \mathcal M(X)$ be the set of all non-negative probability measures on $X$ that project onto $m$ under the map $\pi_1:X\to \Omega$, i.e. 
$$
\mathcal {M}_m=\big\{\mu\in\mathcal M (X): \mu\circ\pi_1^{-1} =m\big\}.
$$
The members of $\mathcal {M}_m$ are called random measures with respect to $m$. Their disintegration measures $\mu_\om$, $\om\in \Om$, with respect to the partition of $X$ into sets $\{\om\}\times \C$, are called fiberwise random measures, and frequently, abusing slightly terminology, these are (also) called just random measures. We are interested in conformal random measures, their existence, uniqueness, and geometrical and dynamical properties. Such measures are characterized by the property that
$$
\nu_{\theta\omega}(F_\om(A))
=\lambda_{t,\omega}\int_A\big|\(F_\om\)'\big|^t\,d\nu_{\omega}
$$
for $m$--a.e. $\om\in\Om$ and for every Borel set $A\sbt Q$ such that $F_\om|_A$ is 1--to--1,where $\lam_t:\Om\to(0,+\infty)$ is some measurable function. Our first main result is about the existence of conformal random measures. Indeed, we proved the following.

\begin{thm}[Existence of conformal measures]\label{ECM_Intro}
For every $t>1$ there exists $\nu^{(t)}$, a random $t$--conformal measure, for the map $F:Q\to Q$. 
\end{thm}

The proof of this theorem is much more involved than its deterministic counterpart of \cite{uz_nonhyp}; the whole Sections~5--8 are entirely devoted to this task. There are many reasons for that. One of them, notorious for random dynamics, is the difficulty to control upper and lower bounds of the measurable function $\lambda_t$. In the deterministic case there is just one number $e^{\Pr(t)}$. Here, we have an apriori uncontrolled function $\lambda_t$. We overcome this difficulty by starting of with good class of random measures: the sets $\mathcal P$ and $\hat{\mathcal P}$ of Sections~5--8. We also must carefully control the trajectories of $0$, the singularity of $f_\eta^{-1}$ for every $\eta\in[A,B]$ and points approaching these trajectories. This is the more difficult in the random case that we now have the trajectory of $0$ for every $\om\in\Om$. There are more subtle and involved issues.

\medskip We then turned our attention to the problem of $F$-invariant random measures absolutely continuous with respect to the random conformal measure of Theorem~\ref{ECM_Intro}. This was done in Section~\ref{inv_meas}. Its full outcome is contained in the following.  

\begin{thm}\label{invariant_Intro}
For every $t>1$ there exists a unique Borel probability $F$--invariant random measure $\mu^{(t)}$ absolutely continuous with respect to $\nu^{(t)}$, the random $t$--conformal measure of Theorem~\ref{ECM_Intro}. In fact, $\nu^{(t)}$ is equivalent with $\nu^{(t)}$ and ergodic.
\end{thm}

\noindent Note that in terms of fiberwise invariant measures, 
$F$--invariance of the measure $\mu^{(t)}$ means that 
$$
\mu^{(t)}_\omega\circ F_\omega^{-1}=\mu^{(t)}_{\theta\omega}
$$
for $m$--a.e. $\omega\in\Omega$.

\noindent Note that we do not claim that the measure $\mu^{(t)}$ is absolutely continuous with respect to any random $t$--conformal measure for the map $F$. We claim this only for the measure $\nu^{(t)}$ resulting from the proof of Theorem~\ref{ECM_Intro}, i.e. Theorem~\ref{t1_2016_10_08}. The proof of Theorem~\ref{invariant_Intro} is done ``globally'' and requires very subtle estimates of fiberwise random conformal measures of various balls and inverse images of measurable sets under all iterates. 

\medskip Turning eventually to geometry, we have defined random counterpart of radial (conical) Julia sets. These are defined as follows.
\begin{equation}
J_r(\om)
:=\big\{z\in Q: \lim_{N\to\infty}\un\rho(N_\om(z,N))=1\big\},
\end{equation}
where $\un\rho(A)$ is the lower (asymptotic) density of $A$, a subset of natural numbers $\N$, and $N_\om(z,N)$ is the set of all integers $n\ge 0$ such that there exists a (unique) holomorphic inverse branch 
$$
F_{\om,z}^{-n}:B(F_\om^n(z),2/N)\to Q
$$
of $F_\om^n:Q\to Q$ sending $F_\om^n(z)$ to $z$ and such that $|F_\om^n(z)|\le N$. 
$J_r(\om)$ is said to be the set of radial (or conical) points of $F$ at $\om$. We further denote:
$$
J_r(F):=\bu_{\om\in\Om}\{\om\}\times J_r(\om),
$$
With HD denoting Hausdorff dimension, we proved in Section~\ref{bowen} the following theorem about the geometric structure of the random radial Julia sets $J_r(\om)$.

\begin{thm}\label{thm:bowen_Intro}
For $t>1$  put 
$$
\mathcal E \Pr(t):=\int_\Om\log \lambda_{t,\omega}d m(\omega).
$$
Then 

\begin{enumerate}
\item $\mathcal E \Pr(t)<+\infty$ for all $t>1$,

\smallskip\item The function $(1,+\infty)\ni t\mapsto \mathcal E \Pr(t)$ is strictly decreasing, convex, and thus continuous, 

\smallskip\item $\lim_{t\to 1}\mathcal E \Pr(t)=+\infty$ and $ \mathcal E \Pr(2))\le 0$. 

\smallskip\item (Bowen's formula) Let $h>1$ be the unique value $t>1$ for which
$ \mathcal E \Pr(t)=0$. Then 
$$
\HD(J_{r,\omega})=h
$$ 
for $m$--a.e.$\omega\in\Omega$.
\end{enumerate}
\end{thm}

A remarkable fact of this theorem is that the Hausdorff dimension of random radial Julia sets $J_{r,\omega}$, $\om\in\Om$, is expressed in terms (zero of the expected pressure $\mathcal E \Pr(t)$) that have nothing to do with these sets. Another remarkable observation about these sets, is their dynamical significance, which follows from the fact, which we proved, that 
$$
\mu(J_r(F))=1
$$
for every $F$--invariant random measure on $Q$.

As a matter of fact, we proved even more about geometry of random radial Julia sets $J_{r,\omega}$ than Theorem~\ref{thm:bowen_Intro}. Namely:

\begin{thm}
The Hausdorff dimension $h=\HD(J_{r,\omega})$ of the random radial Julia set $J_{r,\omega}$, is constant for $m$--a.e. $\om\in\Om$ and satisfies $1<h<2$. In particular, the $2$--dimensional Lebesgue measure of $m$--a.e. $\om\in\Om$ set $J_{r,\omega}$ is equal to zero.
\end{thm}

As its, almost immediate, corollary, we obtain the following result about trajectories of (Lebesgue) typical points.

\begin{thm}[Trajectory of a (Lebesgue) typical point I]
For $m$--almost every $\omega\in\Omega$ there exists a subset $Q_\omega\subset Q$ with full Lebesgue measure such that for all $z\in Q_\omega$  the following holds. 
\begin{equation}\label{bad}
\begin{aligned}
\forall \delta>0 \ \exists n_z(\delta)\in\mathbb N \ &\forall n\ge n_z(\delta)\  \exists k=k_n(z)\ge 0
\\  
&|F^n_\omega(z)-F^k_{\theta^{n-k}\omega}(0)|<\delta\  \  \text{or}\  \  |F^n_\omega(z)|\ge 1/\delta.
\end{aligned}
\end{equation}
In addition, $\limsup_{n\to\infty}k_n(z)=+\infty$.
\end{thm}

\noindent As an immediate consequence of this theorem we get the following. 

\begin{cor}[Trajectory of a (Lebesgue) typical point II]
For $m$--almost every $\omega\in\Omega$ there exists a subset $Q_\omega\subset Q$ with full Lebesgue measure such that for all $z\in Q_\omega$, the set of accumulation points of the sequence
$$
\(F_\om^n(z)\)_{n=0}^\infty
$$
is contained in $[0,+\infty]\cup\{-\infty\}$ and contains $+\infty$. 
\end{cor}

\noindent These last two properties are truly astonishing and were first time observed for the exponential map $\C\ni z\mapsto e^z\in\C$ in \cite{Rees} and \cite{Lyubich} and then extended to many other exponential functions in \cite{uz_nonhyp}. Our approach to establish these two properties is different than those of \cite{Rees} and \cite{Lyubich} and relies on investigation of $h$--dimensional packing measure $Q$. 

\medskip As it is explained in detail in Sections~\ref{apendix1} and Section~\ref{apendix2}, dealing with the sextuple
$$
F:=\(\Omega,\mathcal F,m;\, \theta:\Omega\to \Omega;\, \eta:\Om\to[A,B];F_\eta:Q\to Q\)
$$
and induced by it random dynamics 
$$
\(F^n_\omega:Q\longrightarrow Q\)_{n=0}^\infty, \  \  \om\in \Om.
$$
is entirely equivalent to dealing with the sextuple 
$$
f:=\(\Omega,\mathcal F,m;\, \theta:\Omega\to \Omega;\, \eta:\Om\to[A,B];f_\eta:\C\to \C\)
$$
and induced by it random dynamics 
$$
\(f^n_\omega:\C\longrightarrow C\)_{n=0}^\infty, \  \  \om\in \Om.
$$
if the derivatives of the maps $f_\om^n$ are calculated with respect to the conformal Riemannian metric 
$$
|dz|/|z|.
$$
This metric pops up naturally in Section~\ref{apendix1} and coincides with the metric dealt with in \cite{MayerU_1} and \cite{MayerU_2}.

In Sections 12 and 13 we fully transfer all the main results proven for the sextuple $F$ to the case of the sextuple $f$.

\

\section{Preliminaries}

\subsection{The Quotient Cylinder and the Quotient Maps}

We define the equivalence relation $\sim$ on the complex plane $\C$ by saying that $Z\sim W$ if there exists $k\in\mathbb Z$ such that 
$$
Z-W=2\pi i k.
$$
We denote the quotient space $\C/\sim$ by $Q$. So, $Q$ is conformally an infinite cylinder. We denote by $\pi$ the natural projection $\pi:\C\to Q$, i.e., 
$$
\pi(Z)=[Z]
$$ 
is the equivalence class of $z$ with respect to relation $\sim$.
Since both maps $f_\eta:\C\to \C$ and $\pi\circ f_\eta:\C\to Q$ are constant on equivalence classes, they canonically induce conformal maps $f_\eta:Q\to \C$ and $F_\eta:Q\to Q$. So, $F_\eta$ can be represented as 
$$
F_\eta=\pi\circ f_\eta\circ \pi^{-1},
$$
precisely meaning that for every point in $Q$, its image under $\pi\circ f_\eta\circ \pi^{-1}$ is a singleton and the above equality holds.
Although, formally, $Q$ is the set of equivalence classes $[z]$, we shall often use  the notation $z\in Q$, whenever this does not lead to a confusion.

We will also use occasionally the natural identification
$$
Q\sim\{Z\in \C: 0\le \Im Z <2\pi\},
$$
when this does not lead to a confusion.
For $z\in Q$ we denote 
$$
|z|:=\inf\{|Z|:Z\in \pi^{-1}(z)\}.
$$
Similarly, for $z\in Q$ we denote by $\Re z$ the common value $\Re Z$ for $Z\in\pi^{-1}(z)$.

\medskip We denote by $Y_M$ the set 
$$
Y_M:=\{z\in Q:|\Re(z)|>M\}.
$$ 
This set splits naturally as $Y_M^+\cup Y_M^-$ ,where 
$$
Y_M^+:=\{z\in Q:\Re(z)>M\}\quad\text{and}\quad Y_M^-:=\{z\in Q:\Re(z)<M\}.
$$
We also denote:
$$
Q_M:=\{z\in Q:|\Re z|\le M\}.
$$

For positive variables $A, B$, depending on a collection of parameters, we write $A\preceq B$ if there exists a constant $C$ independent of the parameters such that $$
A\le C\cdot B.
$$
Similarly, we write $A\succeq B$ if $B\preceq A$. 
We write 
$$
A\asymp B \text{{\rm \  \  \  if and only if \ \ }} A\preceq B \  \  {\rm and } \ \ A\succeq B.
$$

\medskip
\subsection{Koebe's Distortion Theorems}
For every $\xi\in \C$ and every $r>0$ let
$$
B(\xi,r):=\{z\in\C:|z-\xi|<r\}
$$
be the open disk (ball) centered at the point $\xi$ with radius $r$. We abbreviate
$$
\ov B(\xi,r):=\ov{B(\xi,r)}.
$$
We record the following classical Koebe's distortion theorems; for proofs see e.g., \cite{Hille_2}.

\begin{thm*}[Koebe's Distortion Theorem]
Let $\xi\in \C$ and let $r>0$.  Let $g: \ov B(\xi,r) \to \C$ be a univalent holomorphic map. Then for every $t\in [0,1)$ and every $z \in \ov{B}(\xi,tr)$ we have 
\begin{gather*}
\frac{1 - t}{(1 + t)^3} \leq \frac{|g'(z)|}{|g'(\xi)|} 
\leq \frac{1 + t}{(1 - t)^3},\\
\frac{t r}{(1 + t)^2}|g'(\xi)| \leq |g(z) - g(\xi)| \leq \frac{t r}{(1 - t)^2}|g'(\xi)|.
\end{gather*}
\end{thm*}
and
\begin{thm*}[Koebe's $1/4$ Theorem]
Let $\xi\in \C$ and let $r>0$. If $g:\mathbb D(\xi,r) \to \C$ is a univalent holomorphic map, then
$$
g(\mathbb D(z_0,r))\spt \mathbb D\lt(g(\xi),\frac14|g'(\xi)|\cdot r\rt).
$$
\end{thm*}
\noindent We shall often refer to these results as to \emph{standard distortion estimates}. From now on throughout the paper, for every $t\in[0,1)$ we set
$$
K_t:=\max\lt\{\frac{1 + t}{(1 - t)^3},\frac{(1 + t)^3}{1 - t}\rt\}\ge 1
$$
and
$$
K:=K_{1/2}.
$$
We often make use of Bloch's theorem, which does not require that the map is univalent:
\begin{thm*}[Bloch Theorem] Let $f$ be a holomorphic map on the unit disc $\mathbb D$; assume that $|f'(0)|=1$. Then there is a region $U\subset \mathbb D$ which is 
mapped by $f$  univalently onto a disc of radius $b\ge 1/72$.
\end{thm*}

\section{Julia Sets of Non--Autonomous Iterations of  exponential Maps}\label{julia_random}\label{sec:julia}
As in the introduction for $\eta\neq 0$ we denote by $f_\eta:\C\to\C$ the entire map defined by
$$
f_\eta(z)=\eta e^z.
$$
Fix two real numbers $A\le B$ with $A>1/e$. 
Put 
$$
{\bf A}:=[A,B]^\mathbb N.
$$
For every infinite sequence of numbers in $[A,B]$, i.e., every element ${\ba}=\{a_1,a_2,\ldots \}$ of the infinite product $[A,B]^\mathbb N$,
define the non--autonomous dynamical system by the following formula:
$$
f^n_{\bf a}
:=f_{a_n}\circ f_{a_{n-1}}\circ\cdots\circ f_{a_2}\circ f_{a_1}:\C\longrightarrow\C.
$$
For every $\ba\in{\bf A}$ the respective Fatou and Julia sets $F_{\ba}$ and  $J_{\ba}$ are then defined analogously as in the deterministic case:
$$
F_{\ba}:=\Big\{z\in \C: f^n_{\ba}\rvert_U\ \text{ is normal for some neighborhood}\ U\ \text{of}\ z\Big\}
$$
and 
$$
J_{\ba}:=\mathbb{C}\setminus F_{\ba}.
$$
Denote by 
$$
\sigma:{\bf A}\longrightarrow{\bf A}
$$
the left shift, i.e., the map 
$$
\sigma(a_1,a_2,a_3\ldots)=(a_2,a_3,a_4\ldots).
$$
Note that both these sets $F_{\ba}$ and $J_{\ba}$ are invariant by the dynamics. More precisely:
$$
f_{\ba}^1(J_{\ba})
=f_{a_1}(J_{\ba})
\subset J_{\sigma(\ba)} 
\  \  \  {\rm and } \  \   \
f_{\ba}^1(F_{\ba})
=f_{a_1}(F_{\ba})
\subset F_{\sigma(\ba)}.
$$ 
Our next theorem extends to the non--autonomous case the celebrated result of Michal Misiurewicz (see \cite{mi}) which was conjectured by Pierre Fatou already in 1926 (see \cite{Fatou}). The proof we provide is simple and it constitutes a substantial simplification also for deterministic maps.
\begin{thm}\label{CC} 
For every $\ba\in {\bf A}$, we have that
$$
J_{\ba}=\mathbb{C}.
$$
\end{thm}
\noindent The proof of Theorem~\ref{CC} will consist of several lemmas.

\begin{lem} 
For every ${\ba}\in {\bf A}$,
$$
J_{\ba}\spt \R.
$$
\end{lem}

\begin{proof}
First, observe that if $x\in\mathbb{R}$, then
$$
\lim_{n\to\infty}f^n_{\ba}(x)\to+\infty.
$$
Now, if $w\in \mathbb{R}\sms J_{\ba}$, then there exists a neighborhood $V\sbt\C$ of the point $w$ in $\C$ such that the family $\(f^n_{\ba}|_V\)_{n=0}^\infty$ is normal.  So, since also $f^n_{\ba}|_{\mathbb{R}\cap V}\to\infty$, as $n\to\infty$, we conclude that $f^n_{\ba}$ converges to infinity uniformly on compact subsets 
of $V$ as $n\to\infty$.  Remember that for this specific family $f_\eta$ we have $f_\eta=f'_\eta$. So, if $\overline{B}(w,r)\subset V$, then 
$$
|(f^n_{\ba})'|\big|_{B(w,r)}\to\infty
$$
uniformly as $n\to\infty$. Thus, by virtue of
Bloch's Theorem, the image $f^n_{\ba}(B(w,r))$ contains a ball of radius $2\pi$ for all $n\ge 0$ sufficiently large. This implies that there 
exists a sequence of points $z_n\in B(w, r)$, $n\ge 0$ large enough, such that
$$
\lim_{n\to\infty}\big|\Re(f^n_{\ba}(z_n))\big|=+\infty,
$$
and
$$
\Im f^n_{\ba}(z_n))\in\pi+2\pi\Z.
$$
Then $f^{n+1}_{\ba}(z_n)\in(-\infty,0)$ and, consequently, $|f^{n+2}_{\ba}(z_n)|<B$. Thus,  $f^n_{\ba}|_{B(w,r)}$ does not tend to infinity as $n\to\infty$. This contradiction finishes the proof.
\end{proof}

\noindent As an immediate consequence of this lemma we get the following.

\begin{cor}\label{cor1}
If $V\subset \C$ is an open set and $V\cap J_{\ba}=\emptyset$, then $V\cap\mathbb{R}=\emptyset$. Furthermore,
\begin{enumerate}
\item 
$$
\mathbb{R}\cap \bigcup_{n=0}^\infty f^n_{\ba}(V)=\emptyset,
$$

and, more generally,

\item
$$
\left(\bigcup _{k\in\mathbb{Z}}\mathbb{R}+k\pi i\right)\cap \bigcup_{n=0}^\infty f^n_{\ba}(V)=\emptyset.
$$
\end{enumerate}
\end{cor}

The next lemma and its proof come as minor modifications from \cite{mi}.

\begin{lem}\label{lem1} 
For every $z\in\C$ and every integer $n\ge 1$,
$$
|(f^n_{\ba})'(z)|\ge|\Im f^n_{\ba}(z)|.
$$
\end{lem}
\begin{proof}
$f_\eta(z)=\eta e^z=\eta e^x\cos y+i\eta e^x\sin y$. Since $|\sin y|\le |y|$, we thus have that $|\Im f_\eta(z)|\le\eta e^x|y|=|f_\eta(z)||\Im (z)|$. So,
\begin{equation}\label{contraction}
\frac{|\Im f_\eta(z)|}{|\Im(z)|}\le |f_\eta(z)|.
\end{equation}
Therefore,
$$
\aligned
|\Im f^n_{\ba}(z)|
&=\frac{|\Im f^n_{\ba}(z)|}{\Im f^{n-1}_{\ba}(z)|}\cdot \frac{|\Im f^{n-1}_{\ba}(z)|}{\Im f^{n-2}_{\ba} (z)|}\cdot \ldots \cdot \frac{|\Im f^2_{\ba}(z)|}{|\Im f_{\ba}(z)|}\cdot |\Im f_{\ba} (z)|\\
&\le |f^n_{\ba}(z)|\cdot |f^{n-1}_{\ba}(z)|\cdot \ldots \cdot |f^2_{\ba}(z)|\cdot |\Im f_{\ba}(z)|\\
&\le |f^n_{\ba}(z)|\cdot |f^{n-1}_{\ba}(z)|\cdot \ldots \cdot |f^2_{\ba}(z)|\cdot |f_{\ba}(z)|\\
&=|(f^n_{\ba})'(z)|.
\endaligned
$$
\end{proof}

\begin{rem}
The above computation, although very simple, reflects the following phenomenon: Denoting by $\mathbb{H}^+$ and $\mathbb{H}^-$, respectively, the upper and lower halfplane, we see that the branches $f^{-1}_\eta$ of the inverse map  are well-defined in $\mathbb{H}^+$ and $\mathbb{H}^-$, and each of them map $\mathbb{H}^\pm$  into $\mathbb{H}^+$ or  $\mathbb{H}^-$. Since the hyperbolic metric in $\mathbb{H}^\pm$ 
is given by $\frac{|dz|}{|\Im (z)|}$, the inequality \eqref{contraction} just expresses the fact that $f^{-1}_\eta$ are contractions in the hyperbolic metric.
\end{rem}

\begin{lem}
If $V\sbt\C$ is an open connected set and $V\subset\overline{V}\subset \C\setminus J_{\ba}$, then there exists an integer $N\ge 0$ such that for all $n\ge N$,
$$
f^n_{\ba}(V)\sbt S:=\{z\in\C: |\Im (z)|<\pi\}.
$$
\end{lem}
\begin{proof}
By Corollary \ref{cor1}, for every $n\in\mathbb{N}$, either the set $f^n(V)$ is contained in $S$, or it is disjoint from $S$.
If $f^n_{\ba}(V)\cap S=\emptyset$ for infinitely many integers $n\ge 1$ then, using Lemma ~\ref{lem1} and the Chain Rule, we conclude that
$$
\limsup_{n\to\infty}|(f^n_{\ba})'|_{|V}=+\infty.
$$
This (using e.g. Bloch's Theorem) implies that for infinitely many integers $n\ge 1$ the set $f^n_{\ba}(V)$ contains a ball of radius $2\pi$. So, for all such $n$, $f^n_{\ba}(V)\cap  \left (\bigcup _{k\in\mathbb{Z}}\mathbb{R}+k\pi i\right )\neq\emptyset$. This however contradicts Corollary~\ref{cor1}, and we are done.
\end{proof}

Write $S$ as 
$$
S=S^+\cup S^-\cup \mathbb{R},
$$ 
where
$$
S^+:=\{z\in\C: 0<\Im(z)<\pi\} \  \  
{\rm and} \  \
S^-:=\{z\in\C: -\pi<\Im (z)<0\}.
$$ 
For $\ba\in{\bf A}$ denote by $g_{\ba}$ the holomorphic branch of $f^{-1}_{\ba}$ defined on $S^+$ and mapping $S^+$ into $S^+$. More generally, for every $\eta\in [A,B]$, the map $g_\eta$ denotes the holomorphic branch of $f_\eta^{-1}$ mapping $S^+$ into $S^+$. Denote by $\rho$ the hyperbolic metric in $S^+$. 

\begin{lem} For every $\eta\in [A,B]$ and for all $z,w\in S^+$, we have that
\begin{equation}\label{two_B}
\rho(g_\eta(z), g_\eta(w))\le \rho(z,w).
\end{equation}
Also, for every compact subset $K\subset S^+$ there exists $\kappa\in(0,1)$ such that for every $\eta\in [A,+\infty)$ and for all $z, w\in K$, we have that
\begin{equation}\label{one_B}
\rho(g_\eta(z),g_\eta(w))\le \kappa \rho(z,w).
\end{equation}
\end{lem}
\begin{proof}
The formula \eqref{two_B} is  a straightforward consequence of Schwarz Lemma. Since the map $g_\eta:S^+\to S^+$ is not bi--holomorphic, it also follows from Schwarz Lemma that 
\begin{equation}\label{2_2017_01_18}
\rho(g_\eta(z), g_\eta(w))< \rho(z,w)
\end{equation}
whenever $z, w\in S^+$ and $z\ne w$, and in addition,
\begin{equation}\label{1_2017_01_18}
\limsup_{z,w\to\xi\atop z\ne w} \frac{\rho(g_\eta(z),g_\eta(w)}{\rho(z,w)}<1
\end{equation}
for every $\xi\in S^+$. 
In order to prove \eqref{one_B}, fix $\eta_2>\eta_1\ge A$. Since
$g_{\eta_2}(z)=g_{\eta_1}(z)-\log\frac{\eta_2}{\eta_1}$ and $g_{\eta_2}(w)=g_{\eta_1}(w)-\log\frac{\eta_2}{\eta_1}$, and since the metric $\rho$ is invariant under the horizontal translation, we have
$$
\rho(g_{\eta_2}(z), g_{\eta_2}(w))=\rho (g_{\eta_1}(z), g_{\eta_2}(w)).
$$
So, it is enough to  check the estimate \eqref{one_B} for $f_A$. But this follows immediately from \eqref{2_2017_01_18}, \eqref{1_2017_01_18}, and compactness of the set $K$.
Indeed, denote by $|f'|_\rho$ the derivative with respect to the metric $\rho$ and  consider the function $G:K\times K\to \mathbb R$ defined by:
$$G(z,w)=
\begin{cases}
\frac{\rho(f_A(z),f_A(w))}{\rho(z,w)}\ &\text{for}\quad z\neq w\\
|f'|_\rho(z)\ &\text{for}\quad z=w
\end{cases}
$$
Then $G$ is continuous in $K\times K$ and  $G(z,w)<1$ for all $(z,w)\in K\times K$, and \eqref{one_B} follows.
\end{proof}





Lemma~\ref{prop1} below will complete the proof of Theorem~\ref{CC}.
\begin{lem}\label{prop1}
The interior of the set 
$$
\Lambda:=\bigcap_{n=0}^\infty f^{-n}_{\ba}(S)
$$
is empty.
\end{lem}

\begin{proof}

Since 
$$
f_{\ba}(S^+)=\{z\in\C: \Im (z)>0\}, \  f_{\ba}(S^-)=\{z\in\C: \Im(z)<0\},
$$
and
$$
f_{\ba}(\mathbb{R})=(0,+\infty),
$$
it follows that
$$\bigcap_{n=0}^\infty f^{-n}_{\ba}(S)=\bigcap_{n=0}^\infty f^{-n}_{\ba}(S^+)\cup\bigcap_{n=0}^\infty f^{-n}_{\ba}(S^-)\cup \mathbb{R}.$$
We shall prove that the set $\bigcap_{n=0}^\infty f^{-n}_{\ba}(S^+)$ has empty interior; the set  $\bigcap_{n=0}^\infty f^{-n}_{\ba}(S^-)$ can be dealt with in the same way.

So, seeking contradiction, suppose that there exists $V\sbt\C$, a nonempty, open, connected, bounded with 
$$
V\subset\overline{V}\subset\bigcap_{n=0}^\infty f^{-n}_{\ba}(S^+).
$$
Then, obviously, the family $\(f^n_{\ba}|_V\)_{n=0}^\infty$ is normal. 
Now, fix a non-empty open connected set $W$ (e.g.: a disk) contained, together with its closure, in $V$. Put  
$$
\delta:=\dist(W,\partial V)>0.
$$
Let $N\ge 1$ be so large integer that
$$
\lt(\frac{\pi}{2}\rt)^N\cdot\frac{\delta}{72}>2 \pi.
$$

Now, seeking a contradiction, assume that there exists $\xi\in W$ such that for at least $N$ integers  $n_1, \dots,n_N\ge 0$ we have that  
$$
f_\om^n(\xi)\in \{z\in\C: \Im z> \pi/2\}.
$$
Then $|(f_\om^{n_N})'(\xi)|>(\pi/2)^N$, and again Bloch's Theorem implies that $f_\om^n(W)$ contains some ball of radius  $2\pi$. Since  $f^{n_N}(W)$ does not intersect the Julia set $J_{\theta^n{\ba}}$, this is a contradiction, as $J_{\theta^n{\ba}}\spt \R+2\pi i\Z$. 

\medskip We therefore conclude that the trajectory  $f^n_{\ba}(z)$ of every point $z\in W$ visits the domain $\{z\in\C: \Im z>\pi/2\}$ at most  $N$ times.
  
\medskip Thus, for all integers $n\ge 0$ large enough, say $n\ge q_1\ge 0$, we have that
\begin{equation}\label{1_2017_02_07}
f^n_{\ba}(W)\sbt \{z\in\C: 0<\Im z<\pi/2\}.
\end{equation}
Consequently, 
$$
f_{\bf a}^n(W)\sbt \{z\in\C:\Re z>0\}
$$
for all $n\ge q_1+1$.
Moreover, observe that there exists a constant $M>0$ such that, if  $\Re z\ge M$, $\Im z\in (0,\pi/2)$, and $f_\eta(z)\in S$, $\eta\in [A,B]$, then
$$
\Re f_\eta(z)>\Re z +1.
$$ 
This, in turn, implies that if $f^j_{\ba}(W)\cap \{z\in\C: \Re z \ge M\}\neq\emptyset$ for some integer $j\ge q_1+1$, then the sequence $\(f_\om^n|_W\)_{n=q_1+1}^\infty$ converges uniformly to $\infty$. But since then the sequence $\((f^n)'|_W\)_{n=q_1+1}^\infty$ also converges uniformly to $\infty$, this possibility is again excluded by the conjunction of Bloch's Theorem and  \eqref{1_2017_02_07}.

So, we have conclude that
$$
f_\om^n(W)\sbt \{z\in\C:0<\Re z<M \ {\rm and } \  0<\Im z<\pi/2\}
$$
for all integers $n\ge q_1$.
Since the family $\(f_\om^n|_W\)_{n=q_1+1}^\infty$ is normal, its every subsequence contains a subsequence converging uniformly in $W$ to some limit holomorphic function. Since all the maps $f_\om^n|_W$, $n\ge q_1+1$, expand the hyperbolic metric $\rho$, there are no constant limits of subsequences $\(f^{n_k}_{\ba}|_W\)_{k=1}^\infty$ with values in $S^+$.

So, let $g$ be a non-constant limit of some subsequence $\(f^{n_k}_{\ba}|_W\)_{k=1}^\infty$ converging uniformly.
Shrinking $W$ if necessary, one can assume that $g(W)$ is contained in some compact subset $K\subset S^+$.
Putting 
$$
\tilde K:=\{z\in S^+:\rho(z,K)\le 1\},
$$ 
we see that there is $q_2\ge q_1+1$ such that for every $k\ge q_2$
$$
f_\om^{n_k}(W)\subset \tilde K.
$$
Note that $\tilde K$ has finite hyperbolic diameter, in fact is compact, and put $D:=\diam_\rho (\tilde K)<\infty$. Record that for all $k>q_2$, we have that
$$
f^{n_k}_{\ba}=f^{n_k-n_{k-1}}_{\theta^{n_{k-1}}{\ba}}\circ\dots \circ f^{n_{q_2+1}-n_{q_2}}_{\theta^{n_{q_2}}{\ba}}\circ f^{n_{q_2}}_{\ba}.
$$
Let $z, w\in W$ with $z\neq w$.
Then, using \eqref{one_B} and \eqref{two_B}, we see that
$\rho(z,w)\le \kappa^{k-q_2}D$ for every $k\ge q_2$, which is a contradiction.

\medskip
So, the sequence $\(f^n_{\ba}|_W\)_{n=0}^\infty$ has no subsequence with a non--constant limit.

Since  all limit functions of subsequences of the sequence $\(f^n_{\ba}|_W\)_{n=0}^\infty$ with values 
in $S^+$ have been also already excluded, we arrive at the following conclusion:

\medskip For every $\theta>0$ there exists an integer $n_\theta\ge 0$ such that 
$$
f^n_{\ba}(W)\sbt \{z\in\C: 0<\Im z<\theta\}\cap \{z\in\C:0<\Re z<M\}. 
$$
for every $n\ge n_\theta$.
\medskip

In order to complete the proof, we now shall show that the above is impossible. This can be deduced immediately  from  the following lemma. 
Its proof is an easy calculation and will be omitted.

\begin{lem}\label{lem6}
Let $\delta>0$ be so small that $(1-\delta)>\frac{1}{A e}$. 
Then for every $\eta\ge A$ and for every $z\in\C$ with $\cos \Im z>1-\delta$, we have that
$$
\Re f_\eta(z)>\Re(z)+A e(1-\delta).
$$
In particular, the map $f_\eta$ moves the region $\{z\in S^+: \cos \Im z>1-\delta)\}$ by $\varepsilon$ to the right.
\end{lem}
\end{proof}

\

\section{Random Exponential Dynamics and Random Measures in $Q$}

As in \cite{Arnold}, \cite{crauel}, \cite{MSU}, \cite{MayerU_3}, and \cite{MayerU4} the randomness is modeled by a measure preserving invertible 
dynamical system 
$$
\theta:\Omega\to\Omega,
$$
where $(\Omega,\mathcal F, m)$ is a complete probability measurable space, and $\theta$ is a measurable invertible map, with $\theta^{-1}$ measurable, preserving the measure $m$. As in the previous section, fix some real constants $A, B$ with $A>1/e$. Let 
$$
\eta:\Omega\longmapsto [A,B]
$$
be a measurable function. Furthermore, to each point $\omega\in\Omega$ associate the exponential map 
$$
f_\omega:=f_{\eta(\omega)}:\C\longrightarrow\C
$$
given by the formula
$$
f_\omega(z)=\eta(\om)e^z.
$$ 
Consequently, for every $z\in \C$, the map 
$$
\Omega\ni\omega\longmapsto f_{\eta(\omega)}(z)\in\C
$$ 
is measurable.

We consider the dynamics of random iterates of exponentials:
$$
f^n_\omega
:=f_{\theta^{n-1}\omega}\circ\cdots\circ f_{\theta\omega}\circ f_\omega:\C\longrightarrow\C.
$$
The quintuple 
$$
\(\Omega,\mathcal F,m;\, \theta:\Omega\to \Omega;\, \eta:\Om\to[A,B]\)
$$
and induced by it random dynamics 
$$
\(f^n_\omega:\C\to\C\)_{n=0}^\infty, \  \  \om\in \Om,
$$
will be referred to in the sequel as \emph{random exponential dynamical system}.
As we have explained it in the introduction, we will in fact do all of our investigations for the maps projected to the cylinder $Q$. More precisely, for every $\om\in \Om$, we consider the map
$$
F_\om=\pi\circ f_\om\circ\pi^{-1},
$$
and the corresponding random dynamical system
$$
F^n_\omega
:=F_{\theta^{n-1}\omega}\circ\cdots\circ F_{\theta\omega}\circ F_\omega:Q\longrightarrow Q.
$$
As it was indicated in the introduction, and explained in detail in Section~\ref{apendix1}, which can be read now with full understanding, and in Section~\ref{apendix2}, this is entirely equivalent to dealing with the random dynamical system $(f_\om^n)$ with derivatives calculated with respect to the conformal Riemannian metric 
$$
\frac{|dz|}{|z|}.
$$
This metric pops up naturally in Section~\ref{apendix1} and coincides with the metric dealt with in \cite{MayerU_1} and \cite{MayerU_2}.

\medskip Recall from \cite{crauel} that a function $g:\Omega\times Q\to\C$, $g(\omega, z)=g_\omega(z)$, is called a random continuous function if, for every $\omega\in\Omega$ the function 
$$
Q\ni z\longmapsto g_\omega(z)\in \C
$$ 
is continuous and bounded, and, in addition, for every $z\in Q$ the function
$$
\Omega\ni\omega\longmapsto g(\om,z)\in\C
$$
is measurable. It then follows ( see, e.g., Lemma~1.1 in\cite{crauel}) that every random continuous function is measurable with respect to the product $\sigma$--algebra $\mathcal F\otimes \mathcal B$, where  $\mathcal B$ is the Borel $\sigma$--algebra in $Q$. Moreover, the map
$$
\Omega\ni\omega\longmapsto \|g_\omega\|_\infty\in\R
$$ 
is measurable and, $m-$ integrable.
The vector space of all real--valued random continuous functions is denoted by $C_b(\Omega\times Q)$. Equipped with the norm 
$$
\|g\|:=\int_\Omega\|g_\omega\|_\infty dm(\om)
$$
it becomes a Banach space.

The simplest example of such a random map is obtained  just by putting 
$
\Omega:=\Pi_{-\infty}^\infty [A,B]$, equipped with some (completed) product measure, and putting, for $\omega=(\dots\eta_{-1},\eta_0,\eta_1\dots)$ $\eta(\omega):=\eta_0$.

Put 
$$
X:=\Omega\times Q.
$$
Denote by $\mathcal M(X)$ the space of all those signed measures $\nu$ defined on the $\sigma$-algebra $\mathcal F\otimes \mathcal B$ for which 
$$
\|\nu\|_\infty:={\rm esssup}\{|\nu_\omega|:\om\in\Om\}<+\infty,
$$
where $\nu_\omega$, $\om\in\Om$, is the corresponding disintegration of $\nu$ and, for each $\om\in\Om$ the number $|\nu_\omega|$ is the total variation norm of $\nu_\om$.

These measures, i.e. the members of $\mathcal M(X)$, can be canonically identified with linear continuous functionals on the Banach space $C_b(\Omega\times Q)$.

Let 
$$
\pi_1:X\to \Omega
$$
be the projection onto the first coordinate, i.e.,
$$
\pi_1(\om,z)=\om.
$$
Let $\mathcal {M}_m\subset \mathcal M(X)$ be the set of all non-negative probability measures on $X$ that project onto $m$ under the map $\pi_1:X\to \Omega$, i.e. 
$$
\mathcal {M}_m=\big\{\mu\in\mathcal M (X): \mu\circ\pi_1^{-1} =m\big\}.
$$

A map $\mu:\Omega\times \mathcal B\to[0,1]$, $(\omega, B)\longmapsto \mu_\omega(B)$, is called a random probability measure on $Q$ if 
\begin{itemize}
\item{} For every set $B\in\mathcal B$ the function $\Om\ni\omega\longmapsto \mu_\omega(B)\in[0,1]$ is measurable,
\item{} For $m$-almost every $\omega\in\Omega$ the map $\mathcal B\ni B\mapsto \mu_\omega(B)\in[0,1]$ is a Borel probability measure.
\end{itemize}

A random measure $\mu$ will be frequently denoted as $\{\mu_\omega\}_{\om\in\Om}$ or $\{\mu_\omega:\om\in\Om\}$.

The set $\mathcal M_m(X)$ can be canonically identified with the collection of all random probability measures on $Q$ as follows.

\begin{prop}[see Propositions 3.3 and 3.6 in\cite{crauel}]
With the above notation, for every measure $\mu\in\mathcal M_m(X)$ there exists a unique random measure $\{\mu_\omega\}_{\om\in\Om}$ on $Q$ such that 
$$
\int_{\Omega\times Q}h(\omega,z)\,d\mu(\omega,z)
=\int_\Omega\left (\int_Q h(\omega,z)\,d\mu_\omega(z)\right)dm(\omega)
$$
for every bounded measurable function $h:\Omega\times Q\to \mathbb R$.

Conversely, if $\{\mu_\omega\}_{\om\in\Om}$ is a random measure on $Q$, then for every bounded measurable function $h:\Omega\times Q\to \mathbb R$ the function $\Om\ni\omega\longmapsto \int_Qh(\omega,z)d\mu_\omega(z)$ is measurable, and the assignment
$$
\mathcal F\otimes \mathcal B\ni A\longmapsto\int_\Omega\int_Q\1_A(\omega,z)d\mu_\omega(z) dm(\omega),
$$
defines a probability measure $\mu\in\mathcal M_m(Q)$.
\end{prop}

\

Both sets $\mathcal M(X)$ are $\mathcal M_m$ are equipped in \cite{crauel} with a topology called therein as a narrow topology. 
This  topology is on $\mathcal M(X)$ generated by the following local bases of neighborhoods of elements $\nu\in\mathcal M(X)$:
$$
U_{g_1,\dots g_k;\delta}(\nu):=\left\{\mu\in \mathcal M: \Big|\int g_j d\mu-\int g_jd\nu\Big|<\delta\right\},
$$ 
where $g_1,\dots g_k$ is an arbitrary collection of random continuous functions and $\delta$ is some positive number. The space $\mathcal M_m$ is then endowed with the subspace topology of the narrow topology on $\mathcal M(X)$. This topology is in general non--metrizable neither on $\mathcal M(X)$ nor on $\mathcal M_m$.

A subset $\mathcal R\subset\mathcal M_m$ is said to be tight if for every $\varepsilon>0$ there exists $M>0$ such that for every $\nu\in \mathcal R$ we have that
$$
\nu(\Omega\times Q_M)\ge1-\varepsilon.
$$
We recall Theorem 4.4 in \cite{crauel}:

\begin{thm}[Crauel's  Prokhorov Compactness Theorem]\label{crauel-prokhorow} A set $\mathcal R\subset \mathcal M_m$ is tight if an only if it is relatively compact with respect to the narrow topology. In this case, $\mathcal R$ is also relatively sequentially compact. 
\end{thm}

\

\section{Random Conformal Measures for Random Exponential Functions; a Preparatory Step}
In this section, after short preparation, we define random $t$--conformal measures, and our main goal in it is to prove their existence for every $t>1$. In order to do this we introduce a subspace  of random  measures for our random dynamics of exponentials. After defining a properly chosen convex and compact subset  $\mathcal P\subset \mathcal M_m$, with respect to the narrow topology, we will check that this set is invariant under an appropriate continuous map. The existence of a random conformal measure will be then deduced from the Schauder--Tichonov Fixed Point Theorem.

\begin{dfn}\label{PF}
We define a family of  operators $\mathcal L_{t,\omega}$, $t>1$, $\om\in\Om$, by 
$$
\L_{t,\omega}(g)(z):=\sum_{w\in F_\omega^{-1}(z)}g(w)\cdot |F'_\omega(w)|^{-t}\in\R,
$$
where $g:Q\to\R$ ranges over bounded continuous functions. Note that the series above converges indeed since $t>1$; this is not difficult to check and can be done in exactly the same way as in \cite{uz_nonhyp}.

Furthermore, we define 
the global transfer operator $\mathcal L_t$ on the space $C_b(\Omega\times Q)$ as follows:  for $(\omega,z)\in X=\Omega\times Q$ and a random continuous function $g$, we put
$$
(\L_t g(z))_\omega:=\L_{t,\theta^{-1}\omega}( g_{\theta^{-1}\omega})(z).
$$
\end{dfn}

Note that $\mathcal L_{t}$ does not act on the space $C_b(\Omega\times Q)$, i.e. its image is not contained in  $C_b(\Omega\times Q)$. The point is that for each $\om\in \Om$ the function  $\L_{t,\omega}(\1)$ is unbounded.
However, we shall check that for each random continuous function $g:X\to\R$ and suitably chosen family of random measures $\nu$, the integral 
$$
\int\L_{t,\omega}(g_{\theta\omega})d\nu_{\theta\omega}
$$
is well defined. This will follow from integrability of the functions 
$$
Q\ni z\longmapsto\L_{t,\omega}(\1)(z)\in\R,
$$
$\om\in\Om$, with respect to the measures $\nu_{\theta\omega}$.   
Verifying this will allow us to define formally the measures $\mathcal L^*_{t,\omega}\nu_{\theta\omega}$, $\om\in\Om$,  as
$$
\mathcal L^*_{t,\omega}\nu_{\theta\omega}(g):=\int \L_{t,\omega}g_{\omega}d\nu_{\theta\omega}.
$$
The random measure $\(\nu_\om\)_{\om\in\Om}$ is then said to be 
$t$--conformal if 
$$
\L_{t,\omega}^*(\nu_{\theta\omega})=\lambda_{t,\omega}\nu_{\omega}
$$
for $m$--a.e. $\om\in\Om$, where $\lam_t:\Om\to(0,+\infty)$ is some measurable function. A straightforward calculation shows that
$t$--conformality is also characterized by the property that
$$
\nu_{\theta\omega}(F_\om(A))
=\lambda_{t,\omega}\int_A\big|\(F_\om\)'\big|^t\,d\nu_{\omega}
$$
for $m$--a.e. $\om\in\Om$ and for every Borel set $A\sbt Q$ such that $F_\om|_A$ is 1--to--1, where $\lam_t:\Om\to(0,+\infty)$ is some measurable function. 

\medskip Our task now, in the upcoming sections, is to prove the existence of  random $t$--conformal measures for every $t>1$. Let $\mathcal P\subset \mathcal M_m$. We want to  define a map
$\Phi:\mathcal P \to\mathcal M_m$ by the following formula/requirement:

\begin{equation}\label{defphi}
(\Phi(\nu))_\omega
:=\frac{\L_{t,\omega}^*(\nu_{\theta\omega})}{\L^*_{t,\omega}(\nu_{\theta\omega})(\1)},
\end{equation}
i.e., the measure $\Phi(\nu)$ is the only measure in $\mathcal M_m$, with disintegration $\Phi(\nu)_\omega$ given by \eqref{defphi}.
We look for a sufficient condition under which the map $\Phi$ is well defined on $\mathcal P$. We first prove a technical lemma and then provide such sufficient condition in Proposition~\ref{prop:phi_well_defined} following it.

\begin{lem}\label{prop_przyblizanie}
Fix $\varepsilon>0$ arbitrary. Let $\mathcal C_\varepsilon \subset C_b(\Omega\times Q)$ be the set of all random continuous functions defined on $\Om\times Q$ that vanish in 
$$
\Om\times\{z\in Q:\Re(z)<\log\varepsilon\}.
$$
Then 
$$
\mathcal L_t g\in C_b(\Omega\times Q)
$$ 
for each $g\in \mathcal C_\varepsilon$.
\end{lem} 

\begin{proof}
In order to prove that $\mathcal L_t g\in C_b(\Omega\times Q)$, we need to  check 

\begin{itemize}
\item Measurability of the function $\Om\ni\omega\longmapsto\mathcal L_{t,\omega} (g_\omega)(z)$, with fixed $z\in Q$, 

\item Continuity of the function $Q\ni z\longmapsto\mathcal L_{t,\omega} (g_\omega)(z)$ with fixed $\omega\in\Omega$, 

and finally, 

\item The bound 
$$
\int _\Omega \|\mathcal L_{t,\omega} (g_\omega)\|_\infty dm(\omega)<\infty. 
$$ 
\end{itemize}

Recall the definition:
$$\L_{t,\omega}(g_\omega)(z)=\sum_{w\in F_\omega^{-1}(z)}g_\omega(w)\cdot |F'_\omega(w)|^{-t}.$$
The preimages $w\in F_\omega^{-1}(z)$ can be easily calculated, using the equation
$\eta(\omega)\exp(w_k)=z+2k\pi i$,
so, $$w_k=w_k(\omega)={\rm Log}\left (\frac{z+2k\pi i}{\eta(\omega)}\right )$$ where we denoted by ${\rm Log}(Z)$ the only $W\in Q$ such that $\exp(W)=Z$. 

With $z$ fixed, the measurability with respect to $\omega$ is now easily seen from the above explicit formula.
Note also, that we can write even more explicitly:
\begin{equation}\label{eq:def_L}
\L_{t,\omega}(g_\omega)(z)=\sum_{w_k(\omega)}g_\omega(w_k)\cdot \left |\frac{\eta(\omega)}{z+2k\pi i}\right |^t.
\end{equation}
Since $t>1$, the above series of continuous functions  converges uniformly in a neighborhood of any point $z\in Q$, $z\neq 0$, thus defining a continuous function. It remains to prove continuity at $0$. But, since we assumed that $g\in \mathcal C_\varepsilon$, it follows that in a sufficiently small neighborhood of $z=0$ the summand corresponding to the integer $k=0$ vanishes, and the sum in \eqref{eq:def_L} is taken only over all $k\neq 0$; then the previous argument, i.e. the one for points $z\neq 0$ applies.

Finally,  the formula \eqref{eq:def_L} also shows that in some neighborhood $U_\varepsilon$ of $0$ we have the following bound:
$$
|\L_{t,\omega}(g_\omega)(z)|\le \sum_{k\in\Z, k\neq 0}\left |\frac{\eta(\omega)}{z+2k\pi i}\right |^t\cdot ||g_\omega||_\infty
$$
while, outside $U_\varepsilon$,
$$|\L_{t,\omega}(g_\omega)(z)|\le \sum_{k\in\Z}\left |\frac{\eta(\omega)}{z+2k\pi i}\right |^t\cdot ||g_\omega||_\infty
$$
Thus, there exists a constant $D_\varepsilon$ such that 
$$
\|\L_{t,\omega}(g_{\omega})\|_\infty\le D_\varepsilon\cdot ||g_\omega||_\infty.
$$ 
We conclude that  $\L_tg=(\L_{t,\omega}(g_{\omega}))_{\omega\in\Omega}$ is a random continuous function.
\end{proof} 

\medskip\begin{prop}\label{prop:phi_well_defined}
Let $\mathcal P\subset \mathcal M_m(X)$. 
Assume that there exist $\rho>0$ and a monotone increasing continuous function $\varphi:(0,\rho)\to [0,+\infty)$ such that $\lim_{\varepsilon\to 0^+}\varphi(\varepsilon)=0$ and
for each $\nu\in\P$,  every $\om\in\Om$ and $\varepsilon\in (0,\rho)$ we have that
\begin{equation}\label{bounds_on_integrals}
\int_{B(0,\varepsilon)}\L_{t,\omega}(\1)(z)\, d\nu_{\theta\omega}(z)\le\varphi(\varepsilon).
\end{equation}
Assume also that there are constants $P\ge p>0$ such that
\begin{equation}\label{5_2017_12_05}
p\le \int \L_{t,\omega}(\1)(z)\, d\nu_{\theta\omega}(z)\le P
\end{equation}
for all $\nu\in\P$ and each $\om\in\Om$. 

Then, the map $\Om\ni\om\longmapsto \mathcal L_{t,\omega}^*\nu_{\theta\omega}$, given by the formula
\begin{equation}\label{6_2017_12_05}
\mathcal L_{t,\omega}^*\nu_{\theta\omega}(g)
:=\nu_{\theta\omega}\(\mathcal L_{t,\omega} g),
\end{equation}
where $g\in C_b(Q)$, is well defined, making $\mathcal L_{t,\omega}^*\nu_{\theta\omega}$ a finite Borel measure on $Q$. Moreover, the global measure 
\begin{equation}\label{7_2017_12_05}
\mathcal L_t^*\nu:=\(\mathcal L_{t,\omega}^*\nu_{\theta\om}\)_{\om\in\Om}
\end{equation}
is well defined and it belongs to $\mathcal M(X)$.

\, Furthermore, the map $\Phi$, given by \eqref{defphi}, is well defined, meaning that

\begin{itemize}
\item The collection $\{\Phi(\nu)_\omega\}_{\om\in\Om}$ forms a random measure on $Q$; equivalently:

\, \item $\Phi(\P)\sbt \mathcal M_m$;

\, \item Furthermore, the map $\Phi:\P\longrightarrow \mathcal M_m$ is continuous with respect to the narrow topology on $\mathcal M_m$. 
\end{itemize}
\end{prop}

\begin{proof}
The first part of the proposition, i.e. the one pertaining to the formula \eqref{6_2017_12_05}, is immediate from \eqref{5_2017_12_05}. Passing to the second part, i.e. the one pertaining to the formula \eqref{7_2017_12_05},
let $\nu\in\mathcal P$. First, we need to check that for every random continuous function $g:\Om\times Q\to\R$, the function
$$
\Om\ni\omega\longmapsto\int_Qg_\omega d\(\mathcal L_{t,\omega}^*\nu_{\theta\omega}\)\in\R
$$ 
is measurable.
Equivalently, we need to check the measurability of the function:
\begin{equation}\label{meas}
\Om\ni\omega\longmapsto\int\L_{t,\omega}(g_\omega)\, d\nu_{\theta\omega}.
\end{equation}
Since $\nu$ is a random measure, for every random continuous function 
$$
h(\omega,z)=h_\omega(z),
$$  
the function $\Om\ni\omega\longmapsto \int h_\omega d\nu_{\theta\omega}\in\R$ is measurable. However,  the 
function $\Om\ni\omega\longmapsto \L_{t,\omega}(g_\omega)$, particularly the function $\Om\ni\omega\longmapsto \L_{t,\theta\omega}(\1)$, is not a random continuous function. This is so because the function $\Om\ni\omega\longmapsto \L_{t,\omega}(g_\omega)$ is unbounded unless 
$g_\omega(z)\to 0$ as $\Re z\to -\infty$.

In order to overcome this difficulty, we invoke Lemma~\ref{prop_przyblizanie}. Indeed, it follows from this  lemma that for every $g\in \mathcal C_\varepsilon$ the function
$$
\Om\ni\omega\longmapsto \int\L_{t,\omega}(g_\omega)\, d\nu_{\theta\omega}
$$ 
is measurable and finite. Since the 
constant function $\1$ is a pointwise limit of a monotone (increasing) sequence of functions in $\mathcal 
C_\varepsilon$ (with $\varepsilon$ converging to $0$), and since  the integrals $\int\L_{t,\omega}(\1)d\nu_{\theta\omega}$ are 
uniformly bounded with respect to  $\nu\in\mathcal P$, we conclude that the function
$$
\Om\ni\omega\longmapsto \int\L_{t,\omega}(\1)d\nu_{\theta\omega}
$$ 
is measurable and bounded, as a pointwise limit of an increasing sequence of measurable and uniformly bounded functions.
In fact the monotonicity property (increasing sequence) and boundedness were inessential in this argument, and the same reasoning shows that the function
$$
\Om\ni\omega\longmapsto\int \mathcal L_{t,\omega}(g_\omega) d\nu_{\theta\omega}
$$ 
is measurable for every $\nu\in \mathcal P$ and any random continuous function $g:X\to\R$. Measurability of the function defined by \eqref{meas} is thus proved, and the part of \eqref{7_2017_12_05} is established.

Then, the assignment
$$
C_b(\Omega\times Q)\ni g\longmapsto \lt(\frac{\int\L_{t,\omega}(g_\omega)d\nu_{\theta\omega}}{\int\L_{t,\omega}(\1)d\nu_{\theta\omega}}\rt)_{\om\in\Om}
$$
defines a linear function from the Banach space $C_b(\Omega\times Q)$ into $\R$, and moreover, by virtue of \eqref{5_2017_12_05}, this function is continuous.

We thus conclude  that $\Phi(\nu)\in\mathcal M(X)$, and, since for each $\om\in\Om$, $(\Phi(\nu))_\omega$, is a probability measure, we get that
$$
\Phi(\nu)\in\mathcal M_m,
$$
i.e., $\Phi(\nu)$ is a random probability measure.

In order to prove continuity of the map $\Phi$, it is enough to show that the map 
$$
\P\ni\nu\longmapsto \L_t^*\nu\in\mathcal M(X)
$$ 
is continuous with respect to the narrow topology.

So, let $V\subset\mathcal M(X)$ be an open set, and assume that
$$\tilde\nu:=\L_t^*\nu\in V$$ where $\nu$ is some measure in $\P$.
We need to show that $\Phi^{-1}(V)$ contains some neighborhood of $\nu$ in the narrow topology on $\mathcal M_m$.

We can assume without loss of generality that $V$ is taken from the the local base of neighborhoods of $\tilde\nu$, i.e. 
$$
V=\{\tilde\mu\in \mathcal M:|\tilde\mu(g_i)-\tilde\nu(g_i)|<\delta,\, i=1,\dots k\}
$$
with some integer $k\ge 1$, some $\d>0$, and $g_i$, $i=1, 2,,\ld,k$ some functions from the space $C_b(\Omega\times Q)$. Now, we can further assume with no loss of generality that $k=1$, so that
$$
V=\{\tilde\mu\in \mathcal M:|\tilde\mu(g)-\tilde\nu(g)|<\delta\}.
$$
where $g$ is some function in $C_b(\Omega\times Q)$.
Thus,
$$
\Phi^{-1}(V)=\big\{\mu\in \mathcal M_m:|\mu(\L_tg)-\nu(\L_tg)|<\delta\big\}.
$$
By the assumptions of Proposition~\ref{prop:phi_well_defined} there exists $\varepsilon>0$ so small that 
$$
\int_{B(0,\varepsilon)}\L_{t,\omega}\1d\nu_{\theta\omega}<\delta/4.
$$
Next, let $h_\varepsilon:Q\to [0,1]$ be a continuous function such that

\begin{itemize}
\item $h_\varepsilon(z)=1$ whenever $z\in Q$ and $|B\exp(z)|<\varepsilon/2$, 

and 

\item $h_\varepsilon(z)=0$ whenever $z\in Q$ and $|B\exp(z)|>\varepsilon$.
\end{itemize}
Then, for every $\omega\in \Om$, the function $\L_{t,\omega}(h_\varepsilon)$ is non-zero only in the ball $B(0,\varepsilon)$.

Define an auxiliary  function  
$$
g_\varepsilon:=(1-h_\varepsilon)g.
$$
Then $g_\varepsilon\in\mathcal C_\varepsilon$. Put
$$
U:=\big\{\mu\in \P: |\mu(\L_tg_\varepsilon)-\nu(\L_t g_\varepsilon)|<\delta/4\big\}.
$$
Then $U$ is an open neighborhood of $\nu$ in $\P$. If $\mu\in U$, then
$$
|\L^*_t\mu(g)-\L_t^*\nu(g)|
\le |\L_t^*\mu(g)-\L_t^*\mu(g_\varepsilon)|+|\L_t^*\mu(g_\varepsilon)-\L_t^*\nu(g_\varepsilon)|+|\L_t^*\nu(g_\varepsilon)-\L_t^*\nu(g)|.
$$
The second summand is just equal to $|\mu(\L_tg_\varepsilon)-\nu(\L_t g_\varepsilon)|$, so it can be estimated above by $\delta/4$.

The third summand is equal to 
$$
\int_\Omega\int_Q\L_{t,\omega}(h_\varepsilon\cdot g)\, d\nu_\omega dm(\omega),
$$
so its absolute value can be estimated above by
$$
\int_\Omega\|g_\omega\|_\infty\int_Q \L_{t,\omega}(h_\varepsilon)\,d\nu_\omega dm(\omega)
\le \int_\Omega\|g_\omega\|_\infty\int_{B(0,\varepsilon)}\L_{t,\omega}(\1)\,d\nu_\omega dm(\omega) 
\le\varphi(\varepsilon)\|g\|.
$$
Since $\mu\in\P$, exactly the same estimate applies to the first summand. Summing up, we conclude that  $U\subset \Phi^{-1}(V)$.
Since $U$ is open in $\P$, the proof is complete.
\end{proof}

\

Our goal is to  apply the general scheme described above, to a properly chosen set $\mathcal P$. This set will be shown in Section~\ref{sec:inv} to be compact, convex,
and invariant under the map $\Phi$.

\

First, we  fix a number
\begin{equation}\label{fixr_0}
r_0\in\lt(0, \frac1{2K}\rt).
\end{equation}

\medskip

Next, we formulate  the following straightforward estimate. Its proof is omitted. 

\begin{lem}\label{lower}
There exist constants $D>d>0$ (depending on $t>1$) such that, for every $z\in Q$, 
$$
\frac{d}{|z|^{t-1}}\le\L_{t,\omega}(\1)(z)\le\frac{D}{|z|^{t-1}}.
$$
\end{lem}

\medskip

Put $c:=d/2$, where $d$ comes from Lemma~\ref{lower}.
For $M_0>0$ define the constants:

\begin{equation}\label{eq:C(M0)}
C(M_0):=\frac{M_0^{t-1}}{c}
\end{equation}
and

\begin{equation}\label{eq:c(M0)}
c(M_0):=2D C(M_0)
\end{equation}
where $D$ comes from Lemma~\ref{lower}.

\begin{dfn}[Definition of the space $\mathcal P$]\label{def:p}
Fix some $t>1$. 
Suppose that $\mathcal P\subset \mathcal M_m$ is such a set for which there exists 
$M_0>0$, with   $c(M_0)>0$, and $C(M_0)>0$ defined as in \eqref{eq:C(M0)}, \eqref{eq:c(M0)}, such that the
the following are satisfied: 

\begin{equation}\label{2.1}
\nu_\omega(Q_{M_0})\ge1/2 \quad\text{for all}\quad \omega\in\Omega,
\end{equation} 
\begin{equation}\label{3.1}
\nu_\omega(Y_M^+)\le c(M_0)e^{\frac{M}{2}(1-t)} \quad\text{for all } \omega\in\Omega \text{ and all } M>0,
\end{equation}
and, for every integer $n\ge 0$ the following \emph{Condition $W_n$} holds: 
\begin{Condition $W_n$}
For every $\om\in \Om$ and every $j\in \mathbb{N}\cup\{0\}$ the following bounds hold:
\begin{equation}\label{1.1}
\nu_{\theta^j(\omega)}
(F^{-n}_{\theta^j\omega,*}(B(F^{n+j}_\omega(0),r_0))\le a_{j,n}(\omega)\cdot b_{n+j}(\omega),
\end{equation}
\end{Condition $W_n$}
where 
\begin{equation}\label{ajn}
a_{j,n}(\omega):=(K\cdot C(M_0))^n |F^{j+1}_\omega(0)|^{-t}\cdot |F^{j+2}_\omega(0)|^{-t}\cdot\dots\cdot |F^{j+n}_\omega(0)|^{-t},
\end{equation}
$$
a_{j,0}(\omega):=1,
$$
and
\begin{equation}\label{b_k}
b_k(\omega):=\left (\frac{Kr_0}{2\pi}\right ) c(M_0)\cdot C(M_0)\cdot |F^{k+1}_\omega(0)|^{1-t}\cdot e^{-\frac{(t-1)}{4}|F^{k+1}_\omega(0)|},
\end{equation}
where $F^{-n}_{\theta^j\omega,*}$ is the holomorphic branch of $F^{-n}_{\theta^j\omega}$, defined in $B(F^{n+j}_\omega(0),r_0)$ and mapping $F^{n+j}_\omega(0)$ back to $F^{j}_\omega(0)$.
\end{dfn}

\

\section{The Map $\Phi$ is Well Defined on $\P$}
\label{sec:inv}
Our goal in this section is to show that if the constant $M_0>0$, together with  $c(M_0)>0$, and $C(M_0)>0$ as in \eqref{eq:c(M0)}, \eqref{eq:C(M0)}, is properly selected, then there exist numbers $\rho>0$, $P>0$, and $p>0$ and a function $\varphi(\varepsilon)$ such that for any set $\P\sbt \mathcal M_m$ fulfilling the requirements of Definition~\ref{def:p}, the hypothesis of Proposition~\ref{prop:phi_well_defined} are satisfied. In particular, the map $\Phi$ is well defined on $\P$. In the next section, we will show that 
\begin{equation}\label{5_2017_12_01}
\Phi(\mathcal P)\subset \mathcal P.
\end{equation}
So, our strategy is to fix a non--empty set $\P\sbt \mathcal M_m$ fulfilling the requirements of Definition~\ref{def:p} with some, undetermined yet, constant $M_0>0$, 
and to work out such sufficient conditions for this constant that the hypothesis of Proposition~\ref{prop:phi_well_defined} will be satisfied,  and later, in the next section, to show that formula \eqref{5_2017_12_01} holds.

Now, given $\om\in\Om$, we define a sequence of radii $\(r_n(\omega)\)_{n=1}^\infty$, converging to $0$ as $n\to\infty$. Put
$$
r_n=r_n(\omega):=\frac14r_0\(|F_\omega(0)|\dots |F^n_\omega(0)|\)^{-1}.
$$
Then, by Koebe's $\frac14$--Theorem,
$$
\tilde B_{n,\omega}:=F^{-n}_\omega\(B(F^n_\omega(0),r_0)\)
\spt B\left(0,\frac14r_0\(|F_\omega(0)|\dots |F^n_\omega(0)|\)^{-1}     \right)=B(0,r_n).
$$

\begin{lem}\label{kula_wokol_zera}
Put $s=3t+7$. Then there exist a constant $C\in (0,+\infty)$, independent of $M_0$, such that if $\nu$ is any random measure in $\mathcal P$, then for every radius $r\in (0,r_0/4)$ we have that
\begin{equation}\label{wokol_zera}
\nu_\omega(B(0,r))\le C\cdot(M_0^{t-1})^{n_\om(r)+2}r^s,
\end{equation}
where $n_\om(r)$ is the unique integer $n\ge 0$ for which
$r_{n+1}(\om)\le r<r_n(\om)$.
\end{lem}
\begin{proof}
Denote $n_\om(r)$ by $n$. Using condition $W_n$ we get
\begin{equation}\label{tilde}
\nu_\omega(F^{-n}_{\omega,*}\(B(F^n_\omega(0),r_0))\)
\le a_{0,n}(\omega)b_{n}(\omega).
\end{equation}
Using this condition again one can easily deduce that
$$
\nu_\om(B(0,r))\le \nu_\om(B(0,r_n))\le a_{0,n}(\omega) b_n(\omega)
\le {\rm Const}(M_0^{t-1})^{n+2}\exp\left (\frac{F^{n+1}_\omega(0)} {8}(1-t)\right ),
$$
where the constants are independent of $\omega, n$ and $M_0$.
Now, there exists an integer $N\ge 1$ such that for all $n\ge N$, all $\omega\in\Omega$, and all $r_{n+1}(\omega)\le r<r_n(\omega)$,
$$
\exp\left (\frac{F^{n+1}_\omega(0)} {8}(1-t)\right )\le\left (
\frac14r_0\(|F_\omega(0)|\dots |F^n_\omega(0)|\cdot|F^{n+1}_\omega(0)\rt)^{-s}    
=r_{n+1}^s\le r^s.
$$
If $n<N$, we still have \eqref{wokol_zera}, by increasing the constant $C$ if needed. The proof is complete.
\end{proof}

\begin{lem}\label{ll}
We have that
$$
\lim_{r\to 0}\frac{n_\om(r)}{\ln\ln\frac{1}{r}}=0
$$
uniformly with respect to $\om\in\Om$.
\end{lem}
\begin{proof}
As in the proof of the previous lemma put $n:=n_\om(r)$. Then we have 
$$\frac{1}{r}> \frac{1}{r_n}=\frac{4}{r_0}\cdot |F_\omega(0)|\cdot \cdots\cdot |F^n_\omega(0)|.$$
So, using  $F^k_\omega(0)=\eta(\theta^{k-1}\omega)\exp\( F^{k-1}_\omega(0)\)>\frac{1}{e}\exp\(F^{k-1}_\omega(0)\)$, true for every $k\ge 1$, we get that
$$
\aligned
\ln \frac{1}{r}&>\ln 4-\ln r_0+\ln |F_\omega(0)|+\cdots +\ln |F^{n-1}_\omega(0)|\\
&>\ln 4-\ln r_0+|F_\omega(0)|+\dots |F^{n-1}_\omega(0)|-n\\
&>|F^{n-1}_\omega(0)|
\endaligned
$$ 
for all $n$ large enough, and so, also for all $n$ large enough:
$
\ln\ln\frac{1}{r}>\ln F^{n-1}_\omega(0)\ge n^2,
$ 
and the lemma follows.
\end{proof}

\begin{lem}\label{lem:est_zero}
There exist $u\ge 2t+7$ and $\rho\in(0,r_0/4)$ ($\rho$ depends on $M_0$) such that, for every measure $\nu\in\mathcal P$, we have that
$$
\nu_\omega(B(0,r))\le r^u
$$ 
for all $r<\rho$ and $m$--a.e. $\omega\in\Omega$.
\end{lem}
\begin{proof}
The estimate \eqref{wokol_zera} says that 
$\nu_\omega(B(0,r))\le C\cdot (M_0^{t-1})^{n_\om(r)+2}r^s$ where $s=3t+7>3t+6>2t+7$. Thus, invoking Lemma~\ref{ll} and the definition of $r_n(\omega)$, we see that the required estimate follows, with $u:=2t+7$.
\end{proof}

For every $\varepsilon\in(0,r_0/4)$ let $k(\varepsilon)\ge 0$ be the least non--negative integer $k$ such that
$$
A\exp\(-M_0(k+1)\)<\varepsilon.
$$
Then, define the function $\tilde \varphi(\varepsilon)$ 
\begin{equation}\label{1_2017_11_29}
\tilde\varphi(\varepsilon):=DB^{t+8}\sum_{k=k(\varepsilon)}^\infty\exp\(-M_0(t+8)k\),
\end{equation}

Conforming to our general strategy, thus aiming to apply Proposition~\ref{prop:phi_well_defined}, we shall prove the following. 

\begin{lem}\label{l1_2017_11_29}
If $\nu\in\mathcal P$, then for $m$--a.e. $\om\in\Om$ and every $\varepsilon\in(0,\rho)$, we have that 
\begin{equation}\label{bounds_on_integrals_A}
\int_{B(0,\varepsilon)}\L_{t,\omega}(\1)(z)\, d\nu_{\theta\omega}(z)\le\tilde\varphi(\varepsilon).
\end{equation}
\end{lem}

\begin{proof}
For every $\om\in\Om$, let
$$
B(\omega)=:\big\{z\in Q:|z|<\eta(\omega) e^{-M_0}\big\}.
$$
Partition the ball $B(\om)$ into annuli 
$$
P_k(\omega):= \big\{z\in Q:\eta(\omega) e^{-(k+1)M_0}<|z|\le\eta(\omega) e^{-kM_0}\big\}.
$$
We define $k_\om(\varepsilon)\ge 0$ to be the only non--negative integer such that $\varepsilon\in P_{k_\om(\varepsilon)}(\om)$. Of course $k(\varepsilon)\le k_\om(\varepsilon)$. Therefore, using also Lemma~\ref{lower} and Lemma~\ref{lem:est_zero}, we can estimate as follows.
$$
\begin{aligned}
\int_{B(0,\varepsilon)}\!\!\! \L_{t,\omega}(\1)(z)d\nu_{\theta\omega}(z)
&\le\sum_{k=k_\om(\varepsilon)}^\infty\int_{P_k(\omega)}\L_{t,\omega}(\1)(z)d\nu_{\theta\omega}(z)
\le D\sum_{k=k_\om(\varepsilon)}^\infty\int_{P_k(\omega)}|z|^{1-t}d\nu_{\theta\omega}(z) \\
&\le DB^{1-t}\sum_{k=k_\om(\varepsilon)}^\infty\exp\(M_0(t-1)(k+1)\) \nu_{\theta\omega}\(P_k(\omega)\) \\
&\le DB^{1-t}\sum_{k=k_\om(\varepsilon)}^\infty\exp\(M_0(t-1)(k+1)\)
      B^{2t+7}\exp\(-M_0(2t+7)k\) \\
&=DB^{t+8}\sum_{k=k_\om(\varepsilon)}^\infty\exp\(-M_0(t+8)\) \\
&\le DB^{t+8}\sum_{k=k(\varepsilon)}^\infty\exp\(-M_0(t+8)\) \\
&= \tilde\varphi(\varepsilon).
\end{aligned}
$$
The proof is complete.
\end{proof}

\medskip\noindent Since the function $(0,\rho)\ni\varepsilon\longmapsto \tilde\varphi(\varepsilon)\in (0,+\infty)$, is monotone increasing and $\lim_{\varepsilon\to 0^+}\tilde\varphi(\varepsilon)=0$, there exists a monotone increasing continuous function $(0,\rho)\ni\varepsilon\longmapsto \varphi(\varepsilon)\in (0,+\infty)$ such that
$$
\tilde\varphi(\varepsilon)\le \varphi(\varepsilon)
$$
for all $\varepsilon\in(0,\rho)$ and 
$$
\lim_{\varepsilon\to 0^+}\varphi(\varepsilon)=0.
$$
Therefore, as an immediate consequence of Lemma~\ref{l1_2017_11_29}, we get the following. 

\begin{lem}\label{l2_2017_11_29}
If $\nu\in\mathcal P$, then for $m$--a.e. $\om\in\Om$ and every $\varepsilon\in(0,\rho)$, we have that 
\begin{equation}\label{bounds_on_integrals_A}
\int_{B(0,\varepsilon)}\L_{t,\omega}(\1)(z)d\nu_{\theta\omega}(z)
\le\varphi(\varepsilon).
\end{equation}
\end{lem}
In this Lemma, both $\rho$, and the function $\varphi(\varepsilon)$ depend on the choice of $M_0$.
As an immediate consequence of this lemma and Lemma~\ref{lower}, we get the following.

\begin{lem}\label{l1_2017_12_05}
If $\nu\in\mathcal P$, then there exists $P\in (0,+\infty)$ such that for $m$--a.e. $\om\in\Om$, we have that 
\begin{equation}\label{bounds_on_integrals_A}
\nu_{\theta\omega}\(\L_{t,\omega}\1)
=\int_Q\L_{t,\omega}(\1)(z)d\nu_{\theta\omega}(z)\le P.
\end{equation}
\end{lem}
and, again, the constant $P$ depends on the choice of $M_0$.

\noindent Using Lemma~\ref{lower} again, we will obtain a common, i.e. good for all $\om\in\Om$, lower bound on 
$\nu_{\theta\omega}\(\L_{t,\omega}\1)$.
\begin{lem}\label{measurelower}
(Recall that we assume that $C(M_0)>0$ is of the form \eqref{eq:C(M0)}. For every measure $\nu\in \mathcal P$ (thus, satisfying in particular \eqref{2.1}) the following holds.
\begin{equation}\label{1.3}
\nu_{\theta\omega}\(\L_{t,\omega}\1)
=\int_Q\L_{t,\omega}(\1)(z)d\nu_{\theta\omega}(z)
\ge \frac{c}{M_0^{t-1}}=\frac{1}{C(M_0)}.
\end{equation}
\end{lem}

\begin{proof}
Using Lemma~\ref{lower}, we obtain
$$
\nu_{\theta\omega}\(\L_{t,\omega}\1)
\ge\nu_{\theta\omega}(Q_{M_0})\cdot \inf_{z\in Q_{M_0}}\L_{t,\omega}(\1)
\ge d\nu_{\theta\omega}(Q_{M_0})\cdot \inf_{z\in Q_{M_0}}\frac{1}{| z|^{t-1}}
\ge \frac{c}{M_0^{t-1}}.
$$
The proof is complete.
\end{proof}

\

\section{Invariance of the Space $\P$ Under the Map $\Phi$:
$\Phi(\mathcal P)\subset \mathcal P$}\label{s9_2017_12_05}

\medskip Having Lemmas~\ref{l1_2017_11_29},
 \ref{l1_2017_12_05}, and \ref{measurelower} proved, we can apply Proposition~\ref{prop:phi_well_defined} and take all its fruits. In particular, the measures $\mathcal L_t^*\nu$ and $\Phi(\nu)$ are well defined for all measures $\nu\in\mathcal P$.  

\begin{lem}\label{right_estimate}
If  $\nu\in\mathcal P$ (thus, in particular, $\nu$ satisfies \eqref{2.1}) then the measure $\Phi(\nu)$ satisfies the estimate  \eqref{3.1}, with the constant 
$$
(D M_0^{t-1}/c+C)=(D\cdot C(M_0)+C),
$$
where $C>0$ is some absolute constant, depending on $t>1$ but independent of $M_0$.
Therefore, if $M_0$ is sufficiently large, then the condition \eqref{3.1} is satisfied.
\end{lem}

\begin{proof}
We have 
$$
\aligned
\L^*_{t,\omega}&\nu_{\theta\omega}(Y_M^+)
=\int\L_{t,\omega}(\1_{Y_M^+})(z)d\nu_{\theta\omega}(z)=\\
&=\int_{|\Re z|<e^{M/2}}\L_{t,\omega}(\1_{Y_M^+})(z)d\nu_{\theta\omega}(z)+\int_{|\Re z|\ge e^{M/2}}\L_{t,\omega}
(\1_{Y_M^+})(z)d\nu_{\theta\omega}(z)\\
&\le \int_{|\Re z|<e^{M/2}}\L_{t,\omega}(\1_{Y_M^+})(z)d\nu_{\theta\omega}(z)+\nu_{\theta\omega}(\{z:|\Re z|
\ge e^{M/2}\})\cdot \sup_{|\Re z|\ge e^{M/2}}\L_{t,\omega}(\1)(z)\\
&=\Sg_1+\Sg_2.
\endaligned
$$
It follows from Lemma~\ref{lower} that for $z$ with $|\Re z|\ge e^{M/2}$, we have that $\L_{t,\omega}\1(z)\le De^{\frac{M}{2}(1-t)}$
and, consequently
\begin{equation}\label{1.4}
\Sg_2\le De^{\frac{M}{2}(1-t)}.
\end{equation}
Now, we estimate $\Sg_1$. 
For  $z=[a+bi]$ with $|\Re z|<e^{M/2}$, Lemma~\ref{lower} yields:
$$
\L_{t,\omega}(\1)(z)\ge de^{\frac{M}{2}(1-t)},
$$
and, writing $z=[a+bi]$, 
\begin{equation}\label{1.5}
\L_{t,\omega}(\1_{Y_M^+})(z)=\sum_k\frac{1}{|a+bi+2 k\pi i|^t}\le Ce^{M(1-t)},
\end{equation}
with another positive constant $C$, where the sum is taken over all such integers $k$ for which $\log|a+bi+2k\pi i|-\log\eta(\omega)>M$.
Therefore, we can write, for  $|z|<e^{M/2}$, the following estimate for the summand $I$,  with possibly modified constant $C$:
\begin{equation}\label{2.5}
\frac{\L_{t,\omega}(\1_{Y_M^+})(z)}{\L_{t,\omega}(\1)(z)}
\le C e^{\frac{M}{2}(1-t)} 
\end{equation}
or, equivalently,
\begin{equation}\label{22.5}
\L_{t,\omega}(\1_{Y_M^+})(z)
\le  C e^{\frac{M}{2}(1-t)} \L_{t,\omega}(\1)(z).
\end{equation}
We are close to the end of the proof of \eqref{3.1}.
Using the lower estimate of Proposition~\ref{measurelower}, together with estimates \eqref{1.4} and  \eqref{22.5}, we can write
$$
\aligned
\L^*_{t,\omega}(\nu_{\theta\omega})(Y_M^+)
&\le \Sg_2+\Sg_1 
 \le De^{\frac{M}{2}(1-t)}+Ce^{\frac{M}{2}(1-t)} \int_Q\L_{t,\omega}(\1)(z)d\nu_{\theta\omega}\\
&=De^{\frac{M}{2}(1-t)}+Ce^{\frac{M}{2}(1-t)}
 \L^*_{t,\omega}\nu_{\theta\omega}(\1).
\endaligned
$$
Now, using the definition of the map $\Phi$, we obtain
$$
\Phi(\nu)_\omega(Y^+_M)
=\frac{\L^*_{t,\omega}(\nu_{\theta\omega})(Y_M^+)}{\L^*_{t,\omega}\nu_{\theta\omega}(\1)}
\le\frac{M_0^{t-1}}{c}De^{\frac{M}{2}(1-t)}+Ce^{\frac{M}{2}(1-t)}
=(DC(M_0)+C)e^{\frac{M}{2}(1-t)},
$$
Since $D$ and $C$ are absolute constants, and $C(M_0)\to\infty$ as $M_0\to\infty$, it is clear that for all $M_0$ sufficiently large $DC(M_0)+C<2DC(M_0)=c(M_0)$.
The proof is complete.
\end{proof}

\medskip At this stage of the paper we have all the constants of Definition~\ref{def:p} except $M_0$. Now, we will determine its value.

\
First of all, we require $M_0$ to be large enough as to satisfy Lemma~\ref{right_estimate}.
Next, let us note the  following direct consequence of Lemma~\ref{right_estimate}.
\begin{cor}\label{cor:large M0}
If $M_0>0$ is large enough, then for every random measure $\nu\in \mathcal P$ (thus, in particular, satisfying condition \eqref{2.1}),  the measures $\Phi(\nu)_\omega$, $\om\in\Om$, satisfy the following:
$$
\Phi(\nu)_\omega(Y^+_{M_0})<1/4.
$$
\end{cor}

From now on, we also assume that  $M_0>0$ is large enough to satisfy Corollary~\ref{cor:large M0}.

\begin{prop}\label{W_0}
If $\nu\in\mathcal P$, then, for every $j\ge 0$, we have that
$$
(\Phi(\nu))_{\theta^j\omega}(B(F^j_\omega(0),r_0))\le b_{j}
(\omega),
$$
where
$$
b_{j}
(\omega):=\frac{Kr_0}{2\pi}
c(M_0)C(M_0)\cdot |F^{j+1}_\omega(0)|^{1-t}\cdot e^{\frac{(1-t)}{4}|F^{j+1}_\omega(0)|}.
$$
In other words,
the measure $\Phi(\nu)$ satisfies condition $W_0$.

\end{prop}
\begin{proof}
Since
$$
F_{\theta^j\omega}\(B(F^j_\omega(0),r_0)\)\sbt B\(F^{j+1}_\omega(0),Kr_0|F^{j+1}_\omega(0)|\).
$$
and since
$$
Kr_0|F^{j+1}_\omega(0)|\le 1/2|F^{j+1}_\omega(0)|,
$$
we conclude that 
$$
F_{\theta^j\omega}(B(F^j_\omega(0),r_0))\subset Y^+_M$$
with 
$
M=\frac{1}{2}|F^{j+1}_\omega(0)|=\frac{1}{2}\Re F^{j+1}_\omega(0).
$

Now, using Lemma~\ref{right_estimate}, i.e. using formula \eqref{3.1} that it yields, with $M=\frac{1}{2}|F^{j+1}_\omega(0)|$, and the fact that the image $F_{\theta^j\omega}(B(F^j_\omega(0),r_0))$ covers $Y_M^+$ at most $\frac{r_0F^{j+1}_\omega(0)\cdot K}{2\pi}$ times, i.e. every point in $F_{\theta^j\omega}(B(F^j_\omega(0),r_0))$ has at most $\frac{r_0F^{j+1}_\omega(0)\cdot K}{2\pi}$ preimages in 
$B(F^j_\omega(0),r_0))$, we can estimate the measure  $\L^*_{t,\theta^j\omega}\nu(B(F^j_\omega(0),r_0))$ as follows:

$$
\aligned
\L^*_{t,\theta^j\omega}\nu(B(F^j_\omega(0),r_0)
&\le c(M_0)\exp\lt(\frac14F^{j+1}_\omega(0)(1-t)\rt)\cdot|F^{j+1}_\omega(0)|^{-t}\cdot \frac{|F^{j+1}_\omega(0)|Kr_0}{2\pi}\\
&=\left (\frac{Kr_0}{2\pi}\right ) c(M_0)\cdot |F^{j+1}_\omega(0)|^{-t}\cdot \exp\left(\frac14(1-t)F^{j+1}_\omega(0)\right)\cdot |F^{j+1}_\omega(0)|
\endaligned
$$
and, using in addition the lower bound provided in Proposition~\ref{measurelower},
$$(\Phi(\nu))_{\theta^j\omega}(B(F^j_\omega(0),r_0)\le \left (\frac{Kr_0}{2\pi}\right ) c(M_0)\cdot C(M_0)\cdot \exp\left (\frac{1-t}{4}F^{j+1}_\omega(0)\right )\cdot |F^{j+1}_\omega(0)|^{1-t}.
$$
\end{proof}

\begin{prop}\label{inductive}
If $\nu$ is a random measure in $\mathcal P$, then  the measure
$\Phi(\nu)$ satisfies all  the conditions  $W_n$, $n\ge 0$. 
\end{prop}

\begin{proof}
It was proved in Proposition~\ref{W_0} that then $\Phi(\nu)$ satisfies the condition $W_0$. So, below, we prove that all the conditions $W_n$, $n\ge 1$, are satisfied.
We estimate as follows:
$$
\aligned
\L^*_{t,\theta^j\omega}&\nu_{\theta^{j+1}\omega}\(F^{-n}_{\theta^j\omega,*}(B(F^{n+j}_\omega(0),r_0))\)=\\
&=\nu_{\theta^{j+1}\omega}\lt(\L_{t,\theta^j\omega}\1_{F^{-n}_{\theta^j\omega,*}\(B(F^{n+j}_\omega(0,r_0))\)}\rt)
=\int_{F^{-(n-1)}_{\theta^{j+1}\omega,*}\(B(F^{n+j}_\omega(0),r_0)\)}\big|(F^{-1}_{\theta^j\omega,*})'(y)\big|^td\nu_{\theta^{j+1}\omega}(y) \\
&=\int_{F^{-(n-j)}_{\theta^{j+1}\omega,*}\(B(F^{n+j}_\omega(0),r_0)\)}|y|^{-t}d\nu_{\theta^{j+1}\omega}(y)\\
&\le K|F^{j+1}_\omega(0)|^{-t} \nu_{\theta^{j+1}\omega}\(F^{-(n-1)}_{\theta^{j+1}\omega}(B(F^{n+j}_\omega(0),r_0)\)).
\endaligned
$$
Thus, using the fact that $\L^*_{t,\omega}(\nu_{\theta\omega})(\1)\ge 1/C(M_0)$, known from Lemma~\ref{measurelower}, together with the estimate $W_n$ applied to the measure $\nu$, we get that
$$
\aligned
\Phi(\nu)_{\theta^j\omega}(F^{-n}_{\theta^j\omega,*}\(B(F^{n+j}_\omega(0),r_0)\)
&\le K C(M_0)|F^{j+1}_\omega(0)|^{-t}\nu_{\theta^{j+1}\omega}\(F^{-(n-1)}_{\theta^{j+1}\omega,*}(B(F^{n+j}_\omega(0),r_0))\)\\
&\le K C(M_0)|F^{j+1}_\omega(0)|^{-t}a_{j+1,n-1}(\omega) b_{j+1+n-1}(\omega)\\
&= K\cdot C(M_0)|F^{j+1}_\omega(0)|^{-t} a_{j+1,n-1}(\omega)b_{j+n}(\omega)\\
&=\kappa\cdot |F^{j+1}_\omega(0)|^{-t}\cdot a_{j+1,n-1}(\omega) b_{j+n}(\omega)\\
&=a_{j,n}(\omega)b_{j+n}(\omega).
\endaligned
$$
Thus, the measure $\Phi(\nu)$ satisfies all conditions $W_n$, $n\ge 1$, and the proof is complete.
\end{proof}

Before stating the next proposition, let us recall that the definition of the space $\mathcal P$ depends on the constant $M_0$, which we assumed to be large enough to for the hypotheses of Corollary~\ref{cor:large M0} to be satisfied.  Proposition~\ref{prop:left} below will impose one more condition on $M_0$.

\begin{prop}\label{prop:left} If $M_0>0$ is large enough then for every $\nu\in \mathcal P$ and $m$--a.e. $\omega\in\Omega$, we have that
$$
\Phi(\nu)_{\omega}(Y_{M_0}^-)
\le (C/c)M^{2t-1}_0(\eta(\omega))^{2t+7} e^{M_0t}\sum_{k=1}^\infty\exp(-(2t+7)kM_0)(M_0^{t-1})^{\log k+\log M_0+2}
<1/4.
$$
\end{prop}

\begin{proof} 
Given $\omega\in\Omega$, we have 
$$
F_{\omega}\(Y_{M_0}^-\)
=B(\omega)=\big\{z\in Q:|z|<\eta(\omega) e^{-M_0}\big\}.
$$
Note also that for every point $z\in B(\om)$ the set
$$
F_{\omega}^{-1}(z)\cap Y_{M_0}^-
$$
is a singleton. Denote it by $w$ and note that
$$
F'_{\omega}(w)=z.
$$
Utilizing the annuli $P_k(\omega)$, introduced in the proof of Lemma~\ref{l1_2017_11_29}, and using Lemma~\ref{ll}, we may assume $M_0>0$ to be so large that, if $z\in P_k(\omega)$, then $n_{\th\om}(|z|)<\log k+\log M_0-2$. So, applying \eqref{wokol_zera} and Lemma~\ref{kula_wokol_zera}, we can thus estimate as follows: 
$$
\aligned
\L^*\nu_{\omega}(Y_{M_0}^-)
&=\nu_{\th\omega}\(\L_{\omega}(\1_{Y_{M_0}^-})\)
=\sum_{k=1}^\infty\int_{P_k(\omega)}\frac{1}{|z|^t}d\nu_{\th\omega}
\le\sum_{k=1}^\infty\sup_{z\in P_k(\omega)}\frac{1}{|z|^t}\cdot \nu_{\th\om}(P_k(\omega))\\
&\le C(\eta(\omega))^{s-t} e^{M_0t}\sum_{k=1}^\infty \exp 
\(\!\!-k(s-t)M_0\) ( M_0^{t-1})^{(n_{\th\om}(\eta(\omega))+2)}\\
&\le C(\eta(\omega))^{s-t} e^{M_0t}\sum_{k=1}^\infty\exp\(\!\!-k(s-t)M_0\)(M_0^{t-1})^{\log k+\log M_0}\\
&=C(\eta(\omega))^{2t+7} e^{M_0t}\sum_{k=1}^\infty\exp\(-(2t+7)kM_0\)(M_0^{t-1}\)^{\log k+\log M_0}.
\endaligned
$$
Therefore, invoking now Lemma~\ref{measurelower}, we get
$$
\Phi(\nu)_{\omega}(Y_{M_0}^-)
\le (C/c)B^{2t+7}M^{t-1}_0 e^{M_0t}\sum_{k=1}^\infty\exp\(-(2t+7)kM_0\)(M_0^{t-1})^{\log k+\log M_0}<1/4,
$$
the last inequality holding provided that $M_0>0$ is large enough.
\end{proof}

\medskip Now, fix  $M_0>0$ so large as to satisfy all the above estimates. 
Let us summarize the above sequence of propositions:

\medskip\begin{itemize} 
\item Let $\nu\in\mathcal P$. 

\medskip\item Then Lemma~\ref{right_estimate} shows that  $\Phi(\nu)$ satisfies the estimate \eqref{3.1}. 

\medskip\item Next, Corollary~\ref{cor:large M0} and Proposition~\ref{prop:left} guarantee that $\Phi(\nu)$ satisfies condition~\eqref{2.1}.

\medskip\item Finally, Propositions~ \ref{W_0} and ~\ref{inductive} guarantee that the conditions $W_0,W_1,\dots W_{n}\dots$ hold for $\Phi(\nu)$.
\end{itemize}

The final conclusion of this section is thus the following.

\begin{prop}\label{p15_2017_12_05}
If $\P$ is the set of all measures in $\mathcal M_m$ satisfying the conditions of Definition~\ref{def:p}, with the appropriate constants $M_0$, $c(M_0)$, and $C(M_0)$, determined in the last two sections, then 
$$
\Phi(\mathcal P)\subset \mathcal P.
$$
\end{prop}

\medskip

\section{Random Conformal Measures for Random Exponential Functions; the Final Step}

\medskip 
Since for every intger $l\ge 1$, we have $lM_0\ge M_0$, Proposition~\ref{prop:left} entails the following.
\begin{prop}\label{prop:left_tight}
If $\nu\in\mathcal P$, where $\mathcal P$ comes from Proposition~\ref{p15_2017_12_05}, then for every $l\in \mathbb N$, we have that
\begin{equation}\label{left_tight}
(\Phi(\nu)_{\omega})(Y_{lM_0}^-)\le S(l)
\end{equation}
where 
$$
S(l):=(C/c)B^{2t+7}(M_0l)^{t-1} e^{M_0tl}
\sum_{k=1}^\infty\exp\(-(2t+7)kM_0l\)(M_0^{t-1})^{\log k+\log M_0+l}
$$
and
\begin{equation}\label{2_2017_12_15}
\lim_{l\to\infty}S(l)=0.
\end{equation}
\end{prop}

If $\P$ is the set produced in Proposition~\ref{p15_2017_12_05}, then we denote
by $\hat\P$ its subset consisting of all those measures $\nu$ for which
\begin{equation}\label{1_2017_12_15}
\nu_\om(Y_{lM_0}^-)\le S(l)
\end{equation}
for $m$--a.e. $\om\in \Om$ and all integers $l\ge 1$.
Because of Proposition~\ref{p15_2017_12_05} and Proposition~\ref{prop:left_tight}, we have the following.

\begin{prop}\label{prop: R_invariant} 
If $\P$ is the set produced in Proposition~\ref{p15_2017_12_05}, then 
$$
\Phi(\P)\sbt \hat\P
$$
and 
$$
\Phi(\hat\P)\subset \hat\P.
$$
\end{prop}
\begin{prop}\label{prop:R tight}
If $\P$ is the set produced in Proposition~\ref{p15_2017_12_05}, then the set $\hat\P$ is nonempty, convex and compact with respect to the narrow topology on $\mathcal M_m$.
\end{prop}

\begin{proof}
First, we shall prove that the set $\P$ produced in Proposition~\ref{p15_2017_12_05} is non-empty. Indeed, define $\nu$  in the following way: for every $\omega\in\Om$ consider the set
$$
Z_\omega:=Q_{M_0}\setminus \bigcup_{j=0}^\infty B\(F_{\th^{-j}\om}^j(0),r_0)\).
$$
Let $\nu_\omega$ be just the normalized Lebesgue measure on $Z_\omega$. Since $\supp(\nu_\omega)\subset Q_{M_0}$, the conditions \eqref{2.1} and \eqref{3.1} are trivially satisfied.
Since, for every  $j\in\mathbb{Z}$ and every $n\ge 0$,
$$
F^{-n}_{\theta^j\omega, *}\(B(F^{n+j}_\omega(0),r_0\)\subset B(F^j_\omega(0),r_0),
$$
all the conditions $W_n$, $n\ge 0$, are also trivially satisfied. So $\nu\in\P$. Then $\hat\P\ne\es$ because of the first part of Proposition~\ref{prop: R_invariant}.

Convexity of $\hat\P$ follows immediately from its definition. The uniform estimates provided by formula \eqref{3.1} and \eqref{1_2017_12_15} along with  \eqref{2_2017_12_15} show that the family $\hat\P$ is tight, thus relatively compact according to Theorem~\ref{crauel-prokhorow}. 

Finally, the set $\hat\P$ is closed with respect to the narrow topology on $\mathcal M_m$ because for every measurable set $A\sbt \Om\times Q$ and all measurable functions $g:\Om\to[0,+\infty)$, both the sets
$$
\big\{\nu\in \mathcal M_m:\nu_\om(A_\om)\le g(\om) 
\ \text{ for all} \ 
\om\in \Om\big\}
$$
and 
$$
\big\{\nu\in \mathcal M_m:\nu_\om(A_\om)\ge g(\om)
\  \text{ for all} \ \om\in \Om\big\}
$$
are closed in  $\mathcal M_m$ with respect to the narrow topology.
\end{proof}

Now, we can prove the main theorem of this section.

\begin{thm}
[Existence of $(\omega, t)$ conformal measures $\nu_{\omega}$]\label{t1_2016_10_08}
For every $t>1$ there exists a random $t$--conformal  measure $\nu^{(t)} \in\hat\P$. Recall that $t$--conformality means that
\begin{equation}\label{1_2017_03_28}
\L_{t,\omega}^*(\nu_{\theta\omega}^{(t)})=\lambda_{t,\omega}\nu_{\omega}^{(t)}
\end{equation}
for every $\om\in \Om$, where $\lambda_{t,\omega}:=\mathcal L^*_{t,\omega}\nu_{\theta\omega}^{(t)}(\1)$.
\end{thm}
\begin{proof}
Because of the second part of Proposition~\ref{prop: R_invariant} and because of \ref{prop:R tight}, the  Schauder--Tichonov Fixed Point Theorem applies to the continuous map $\Phi:\hat\P\to\hat\P$, thus yielding a fixed point of $\Phi$ in $\hat\P$. This just means that formula \eqref{1_2017_03_28} holds.
\end{proof}
Also recall that a, very useful in calculations, property equivalent to \eqref{1_2017_03_28}, which will be frequently used in the sequel, is that
\begin{equation}\label{7_2017_12_19}
\nu_{\theta\omega}^{(t)}(F_\om(A))
=\lambda_{t,\omega}\int_A\big|\(F_\om\)'\big|^t\,d\nu_{\omega}^{(t)}
\end{equation}
for every $\om\in \Om$ and for every Borel set $A\sbt Q$ such that $F_\om|_A$ is 1--to--1. By an immediate induction, we then get for every integer $n\ge 0$ the following.
\begin{equation}\label{6_2017_12_19}
\nu_{\theta\omega}^{(t)}(F_\om^n(A))
=\lambda_{t,\omega}^n\int_A\big|\(F_\om^n\)'\big|^t\,d\nu_{\omega}^{(t)}
\end{equation}
for every $\om\in \Om$ and for every Borel set $A\sbt Q$ such that $F_\om|_A$ is 1--to--1. Lemmas~\ref{l1_2017_12_05}, and \ref{measurelower} can be now reformulate as follows. There are two constants $0<p, P<+\infty$ such that 
\begin{equation}\label{bounds_on_lambda}
1/p\le \lambda_{t,\omega}\le P
\end{equation}
for $m$--a.e. $\om\in\Om$. Let us record the following property of the measure $\nu^{(t)}$.

\begin{prop}\label{prop:supp}
For $m$--a.e $\omega\in\Omega$ we have that 
$$
\supp(\nu_\omega^{(t)})=Q.
$$
Moreover, for all numbers $x>0$, $R>0$, and $\varepsilon\in(0,1)$ there exists a constant  $\xi=\xi(x,R,\varepsilon)>0$ and a measurable set $\Om(x,R,\varepsilon)$ such that
$$
m\(\Om(x,R,\varepsilon)\)>1-\varepsilon
$$
and for every $\omega\in\Omega(x,R,\varepsilon)$ and every $z\in Q_x$, we have that 
$$
\nu^{(t)}_\omega(B(z,R))\ge \xi.
$$
\end{prop}

\begin{proof}
Let $z\in Q$, $r>0$. We need to check that 
$$
\nu^{(t)}_\omega(B(z,r))>0.
$$
Since $J(f_\omega)=\C$, there exists an integer $n=n(\omega,z,r)\ge 0$ such that $f^n_\omega(B(z,r))\cap\R\neq \emptyset$. So, there exists $z'\in B(z,r)$ such that $f^n_\omega(z')\in\R$. Since for every $\om\in\Om$ and every $w\in\R$,
$$
\lim_{k\to\infty}\(f_\om^k\)'(w)=\lim_{k\to\infty}f_\om^k(w)=+\infty,
$$
and since each map $f_\eta$ is 1--to--1 on each open ball with radius $\pi$, we first conclude that for all integers $k\ge 0$ large enough
$$
f_\om^k(B(z,r))\spt B\(f_\om^k(z'),\pi).
$$
Having this, using the above, we then immediately conclude that for given $S>0$, we have that
$$
f_\om^k(B(z,r))\spt B\(f_\om^k(z'),S)
$$
for all integers $k\ge 0$ large enough. Then the sets $f_\om^{k+1}(B(z,r))$ contain annuli centered at the origin with the ratio of the outer and inner radii as large as one wishes. These annuli in turn will contain some set of the form
$$
Q_{M_0}+2l\pi i,
$$ 
where $l\in \Z$. This yields 
\begin{equation}\label{5_2017_12_19}
F_\om^{k+1}(B(z,r))\supset Q_{M_0}
\end{equation}
for all integers $k\ge 0$ large enough. 
On the other hand, if $\nu_\omega^{(t)}(B(z,r))=0$ then, using conformality of the measures $\nu_\gamma$, $\gamma\in\Om$, i.e. using \eqref{6_2017_12_19}, we conclude that
$$
\nu^{(t)}_{\theta^{k+1}\omega}(F_\omega^{k+1}(B(z,r)))=0.
$$
This contradicts \eqref{2.1} and \eqref{5_2017_12_19}, finishing the proof of the first part of Proposition~\ref{prop:supp}.

\medskip In order to prove the second statement first note that in view of its first part, we have that for every radius $r>0$ and $m$--a.e. $\om\in\Om$, 
\begin{equation}\label{7_2017_12_19}
\xi_r(\om):=\inf\big\{\nu_\om(B(z,r)):z\in Q_x\big\}>0.
\end{equation}
Now, fix a countable dense subset $\Gamma$ of $Q_x$. Then the function
$$
\Om\ni\om\longmapsto \xi_R^*(\om):=\inf\big\{\nu_\om(B(z,R/2)):z\in \Gamma\big\}\in [0,1]
$$
is measurable and 
\begin{equation}\label{8_2017_12_19}
\xi_{R/2}(\om)\le \xi_R^*(\om)\le \xi_R(\om).
\end{equation}
In particular $\xi_R^*(\om)>0$ for $m$--a.e. $\om\in\Om$. Therefore, there exists $\xi>0$ so small that 
$$
m\((\xi_R^*)^{-1}((\xi,+\infty)\)>1-\varepsilon.
$$
Hence, taking 
$$
\Om(x,R,\varepsilon):=m\((\xi_R^*)^{-1}((\xi,+\infty)\)
$$
and taking into account the right--hand part of completes the proof. 
\end{proof}

\medskip
Now we shall prove a lemma which is of more restricted scope than Proposition~\ref{prop:supp} but which gives estimates uniform with respect to all $\om\in\Om$. We will the derive some of its consequences and will use them later in the paper.

\begin{lem}\label{l1m9}
For every radius $r>0$ there exists $\Delta(r)\in(0,+\infty)$ such that
$$
\nu_\om(B(0,r))\ge \Delta(r)
$$
for $m$--a.e. $\om\in\Om$.
\end{lem}

\begin{proof}
Proceeding in the same way as at the beginning of the proof of Proposition~\ref{prop:supp}, we conclude that there exists an integer $k\ge 0$ such that
$$
F_\om^k(B(0,r))\sbt Q_{M_0}
$$
for $m$--a.e. $\om\in\Om$. Because of the right hand side of \eqref {bounds_on_lambda} and because of \eqref{2.1}, we get that

$$
\frac12
\le \nu_{\th^k\om}\(F_\om^k(B(0,r))\)
\le \lambda_{t,\omega}^k\int_{B(0,r)}\big|\(F_\om^k\)'\big|^t\,d\nu_\om\le P^k\(f_B^k(0)\)^t\nu_\om(B(0,r)).
$$
Hence,
$$
\nu_\om(B(0,r))\ge \frac12P^{-k}\(f_B^k(0)\)^{-t}>0,
$$
and the proof is complete.
\end{proof}

\begin{cor}\label{c2m9}
For every $r>0$ there exist $r_*>0$ and $\Delta^*(r)>0$ such that
$$
\nu_\om(B(0,r)\sms B(0,r_*))\ge \Delta^*(r)
$$
for $m$--a.e. $\om\in\Om$. 
\end{cor}

\begin{proof}
Fix $u>0$ produced in Lemma~\ref{lem:est_zero}. Take then $r_*\in(0,r)$ so small that $r_*^u<\frac12\Delta(r)$. It then follows from Lemma~\ref{l1m9}
and Lemma~\ref{lem:est_zero} that
$$
\nu_\om(B(0,r)\sms B(0,r_*))
=\nu_\om(B(0,r))- \nu_\om(B(0,r_*))
\ge \Delta(r)-r_*^u
\ge \Delta(r)-\frac12\Delta(r)
=\frac12\Delta(r)>0
$$
So, taking $\Delta^*(r):=\frac12\Delta(r)$ completes the proof.
\end{proof}

\begin{cor}\label{c1m10}
For every $M>0$ there exist $M_+\in(M,+\infty)$ and $\Delta_-(M)>0$ such that
$$
\nu_\om\(Y_M^-\sms Y_{M_+}^-\)\ge \Delta_-(M)
$$
for $m$--a.e. $\om\in\Om$. 
\end{cor}

\begin{proof}
Take $M_+\in(M,+\infty)$ so large that
$$
Be^{-M_+}<(Ae^{-M})_*,
$$
where $(Ae^{-M})_*$ comes from Corollary~\ref{c2m9}. Using this corollary and the right-hand side of \eqref{bounds_on_lambda} again, we obtain
$$
F_\om\(Y_M^-\sms Y_{M_+}^-\)
\spt B\(0,Ae^{-M}\)\sms B\(0,Be^{-M_+}\)
\spt B\(0,Ae^{-M}\)\sms B\(0,(Ae^{-M})_*\)
$$
and
$$
\Delta^*(Ae^{-M})
\le \nu_{\th\om}\(F_\om\(Y_M^-\sms Y_{M_+}^-\)\)
\le \lambda_{t,\omega}\int_{Y_M^-\sms Y_{M_+}^-}\big|\(F_\om\)'\big|^t\, d\nu_\om\le PB^t\nu_\om\(Y_M^-\sms Y_{M_+}^-\).
$$
Hence,
$$
\nu_\om\(Y_M^-\sms Y_{M_+}^-\)
\ge P^{-1}B^{-t}\Delta^*(Ae^{-M}),
$$
and the proof is complete.
\end{proof}

\

\section{Random Invariant Measures Equivalent to Random Conformal Measures}\label{inv_meas}
From now on until explicitly stated otherwise, we fix $t>1$ and the random $t$--conformal measure $\nu$, with disintegrations $\(\nu_\omega\)_{\om\in\Om}$ constructed in the previous section.
Recall that we denote 
\begin{equation}\label{1_2017_03_29}
\lambda_{t,\omega}=\mathcal L^*_{t,\omega}\nu_{\theta\omega}(\1)
\end{equation}
for all $\om\in\Om$. We will also use the notation
$$
\lambda_{t,\omega}^n:=\prod_{j=0}^{n-1}\lambda_{t,\theta^j\omega}.
$$
We introduce normalized operators 
$$
\hat{\mathcal L}_{t,\omega}:=\lambda_{t,\omega}^{-1}\mathcal L_{t,\omega}
\ \  {\rm and} \ \ 
\hat{\mathcal L}^n_{t,\omega}:=(\lambda^n_{t,\omega})^{-1}\mathcal L^n_{t,\omega},
$$
so that
$$
\hat{\mathcal L}_{t,\omega}(\nu_{\theta\omega})=\nu_\omega.
$$


Our purpose in this section is to prove the following.

\begin{thm}\label{invariant2}
There exists a random measure $\mu$, i.e. one belonging to $\mathcal M_m$, such that for all $\om\in\Om$ the fiber measures $\mu_\omega$ and $\nu_\omega$ are 
equivalent, and the random measure $\mu$ is  $F$--invariant. The latter meaning that
$$
\mu\circ F^{-1}=\mu,
$$
or equivalently:
$$
\mu_\omega\circ F_\omega^{-1}=\mu_{\theta\omega}
$$
for $m$--a.e. $\omega\in\Omega$.
\end{thm}

\noindent The proof of Theorem~\ref{invariant2} will follow from Proposition ~\ref{prop:invariant}. We start with some estimates.
Fix some numbers $u>2t+7$ and $\rho>0$ satisfying Lemma ~\ref{lem:est_zero}.
Also, because of Lemma~\ref{lower} there exists $M_1\ge M_0$ large enough so that
\begin{equation}\label{fixing_M1}
\frac{1}{p}\sup\big\{\L_{t,\omega}\1(z):z\in Y_{M_1}\big\}<\frac12.
\end{equation}
The need for such choice of $M_1$ will become clear in the course of the proof of Proposition ~\ref{prop:estimate_in YM}.
Note that there exists an integer $N\ge 1$ large enough  that for all $\omega$ 
$$
Q_{M_1}\cap \bu_{j=N+1}^\infty B\(F_{\th^{-j}\om}^j(0),r_0\)=\es.
$$
Since also $\nu_\omega(Q_{M_1})>1/2$ for $m$--a.e. $\omega\in\Omega$, decreasing $r_0>0$ if necessary, we can assume without loss of generality that $0<r_0<\rho$ and 
$$
\nu_\omega\lt(Q_{M_1}\setminus \bu_{j=0}^\infty B\(F_{\th^{-j}\om}^j(0),r_0\)\rt)>1/2
$$
for $m$--a.e. $\omega\in\Omega$.

\begin{lem}\label{AinQ}
If $n\ge 0$ is an integer and
\begin{equation}
A\subset Q_{M_1}\setminus \bigcup_{j=0}^NB\(F_{\th^{n-j}\om}^j(0),r_0\)
\end{equation}
is a Borel set, then
\begin{equation}\label{2}
\nu_\omega(F^{-n}_\omega(A))\le c(M_1,r_0)\nu_{\theta^n\omega}(A),
\end{equation}
where $c(M_1,r_0)\in(0,+\infty)$ is some constant depending on $M_1$ and $r_0$, but independent of $\omega$.
\end{lem}

\begin{proof}
Notice that by partitioning the set 
$$
Q_{M_1}\setminus \bigcup_{j=0}^NB\(F_{\th^{n-j}\om}^j(0),r_0\)
$$
into a finite disjoint union of Borel sets with diameters smaller than $r_0/4$, we may assume without loss of generality that
$\diam(A)<r_0/4$.
Then we further notice that  holomorphic branches of $F_\omega^{-n}$, labeled as $F^{-n}_{\omega,*}$ are well--defined on $A$, in fact on a ball with radius $r_0/2$ centered at a point of $A$, with  distortion bounded by $K$, meaning that 
$$
\frac{|(F^{-n}_{\omega,*})'(x)|}{|(F^{-n}_{\omega,*})'(y)|}\le K
$$
for all $x,y\in A$. We have

\begin{equation}\label{4}
\nu_\omega(F^{-n}_\omega(A))
=\int_A\hat{\mathcal L}^n_{t,\omega}(\1)(z)d\nu_{\theta^n\omega}(z)
\le\sup_A\(\hat{\mathcal L}^n_{t,\omega}(\1)\)\nu_{\theta^n\omega}(A).
\end{equation}
In order to establish the upper  bound for $\sup_A\(\hat{\mathcal L}^n_{t,\omega}(\1)\)$ notice that 
$$
\aligned
\nu_\omega\Big(F^{-n}_{\omega}\Big(Q_{M_1}\setminus &\bigcup _{j=0}^N
B\(F^j_{\th^{n-j}\omega}(0),r_0\)\Big)\Big)=\\
&=\int_{Q_{M_1}\setminus \bigcup _{j=0}^N
B\(F^j_{\th^{n-j}\omega}(0),r_0)\)}\hat{\mathcal L}^n_{t,\omega}(\1)(z)d\nu_{\theta^n\omega}(z)\\
&\ge\inf_{Q_{M_1}}\(\hat{\mathcal L}^n_{t,\omega}(\1)\) \nu_{\theta^n\omega}\Big(Q_{M_1}\setminus \bigcup _{j=0}^N
B\(F^j_{\th^{n-j}\omega}(0),r_0)\)\Big).
\endaligned
$$
Now, again by distortion estimates, there exists a constant $c(M_1,r_0)>0$ such that
\begin{equation}\label{2017_05_12}
\aligned
\inf\Big(\hat{\mathcal L}^n_{t,\omega}(\1)(z):z\in & Q_{M_1}\setminus \bigcup_{j=0}^N B\(F^j_{\th^{n-j}\omega}(0),r_0\)\Big) \ge \\
&\ge \frac{2}{c(M_1,r_0)}\sup\Big(\hat{\mathcal L}^n_{t,\omega}(\1)(z):z\in Q_{M_1}\setminus \bigcup_{j=0}^N B\(F^j_{\th^{n-j}\omega}(0),r_0\)\Big).
\endaligned
\end{equation}
Thus,
$$
\aligned 
1&\ge
\nu_\omega\Big(F^{-n}_{\omega}\Big(Q_{M_1}\setminus \bigcup _{j=0}^N
B\(F^j_{\th^{n-j}\omega}(0),r_0\)\Big)\Big) \\
&\ge \frac12 \frac{2}{c(M_1,r_0)}\sup\Big(\hat{\mathcal L}^n_{t,\omega}(\1)(z):z\in Q_{M_1}\setminus \bigcup_{j=0}^N B\(F^j_{\th^{n-j}\omega}(0),r_0\)\Big) \\
&=\frac{1}{c(M_1,r_0)}\sup\Big(\hat{\mathcal L}^n_{t,\omega}(\1)(z):z\in Q_{M_1}\setminus \bigcup_{j=0}^N B\(F^j_{\th^{n-j}\omega}(0),r_0\)\Big),
\endaligned
$$
i.e.
\begin{equation}\label{bound_for_L}
\sup\Big(\hat{\mathcal L}^n_{t,\omega}(\1)(z):z\in Q_{M_1}\setminus \bigcup_{j=0}^N B\(F^j_{\th^{n-j}\omega}(0),r_0\)\Big)
\le c(M_1,r_0).
\end{equation}
So, inserting this estimate to \eqref{4}, we obtain 
$\nu_\omega(F^{-n}_\omega(A))\le c(M_1, r_0)\nu_{\theta^n\omega}(A) $, as required. The proof is complete.
\end{proof}

Given $\omega\in\Omega$, $n\in\mathbb N$, and $0\le j\le N$, set 
$$
\beta_{n,j}(\om):=F^{j}_{\theta^{n-j}\omega}(0).
$$
Now let 
$$
A\sbt B\(\beta_{n,j}(\om),r_0\)
$$ 
be an arbitrary Borel set. Consider all connected components $C$ of $F^{-n}_\omega\(B\(\beta_{n,j}(\om),r_0\)\)$. We say that such a  $C$ is good if there exists a holomorphic branch of $F^{-n}_\omega$ defined on $B\(\beta_{n,j}(\om),r_0))$ and mapping $B\(\beta_{n,j}(\om),r_0))$ onto $C$. Otherwise, we say that $C$ is bad. Note that $C$ is bad if and only if $0\in f_{\theta^{k+1}\omega} (F^k_\omega(C))$ for some $0\le k\le n-1$. Equivalently, $C$ is bad if and only if $C$ is unbounded.
Now, the set $F^{-n}_\omega(A)$ splits into the disjoint union
$$
F^{-n}_\omega(A)=F^{-n}_{\omega,B}(A)\cup F^{-n}_{\omega,G}(A),
$$
where $F^{-n}_{\omega,B}(A)$ is the intersection of $F^{-n}_\omega(A)$ with the union of all bad components of $F^{-n}_\omega\(B\(\beta_{n,j}(\om),r_0\)\)$ and 
$F^{-n}_{\omega,G}(A)$ is the intersection of $F^{-n}_\omega(A)$ with the union of all good components of $F^{-n}_\omega\(B\(\beta_{n,j}(\om),r_0\)\)$

The next lemma is proved in an analogous way as Lemma~\ref{AinQ}, with possibly modified constant $c(M_1,r_0)$, still independent of $\omega\in\Omega$.
\begin{lem}\label{prop:est_good}
If $\omega\in\Omega$, $n\in\mathbb N$, $0\le j\le N$, and $A\sbt B\(\beta_{n,j}(\om),r_0\)$ is an arbitrary Borel set, then
$$
\nu_\omega(F^{-n}_{\omega,G}(A))\le c(M_1,r_0)\cdot \nu_{\theta^n\omega}(A).
$$
\end{lem}

\noindent In Lemma~\ref{prop:est_bad}, we will provide  estimates for bad components of $F^{-n}_\omega(A)$. In order to do this, we start with the following.

\begin{lem}\label{prop:around_0}
There exits a constant $\gamma>0$ such that for all radii $0<r\le r_0$, all integers $n\ge 0$, and all $m$--a.e. $\omega\in\Omega$, we have that
$$
\nu_\omega(F_{\omega,B}^{-n}(B(0,r))\preceq r^\gamma.
$$
\end{lem}
\begin{proof}
First note that the only bad component of $F^{-1}_{\theta^{n-1}\omega,B}(B(0,r))$ is of the form $\pi\circ f_{\theta^{n-1}\omega}^{-1}(B(0,r))$ where the latter $B(0,r)$ is considered as a subset of $\C$, and $\pi:\C\to Q$ is the canonical projection. Thus, 
\begin{equation}\label{1_2017_04_11}
F^{-1}_{\theta^{n-1}\omega,B}(B(0,r))=Y_M^-
\end{equation}
for some $M\in [ \ln(1/r)+\ln A, \ln(1/r)+\ln B]$.
Next, using the estimate from Lemma~\ref{lem:est_zero} and \eqref{bounds_on_lambda}, we easily conclude that for $0<r\le r_0$, we have that
\begin{equation}\label{2_2017_04_11}
\nu_{\theta^{n-1}\omega}\(F^{-1}_{\theta^{n-1}\omega,B}(B(0,r))\)
\le C r^{-t}r^u=C r^{u-t}.
\end{equation}
with some constant $C\in (0,+\infty)$.
Now, every component of $F_{\omega,B}^{-n}(B(0,r))$ is of the form $F^{-(n-1)}_*(Y_M^-)$ where $F^{-(n-1)}_*(Y_M^-)$ is some connected component of  $F^{-(n-1)}(Y_M^-)$.
Let us note that the set $f^{-1}_{\theta^{(n-2)}\omega}(\{Z\in\C:\Re Z<-M\})$ is a union of (repeated periodically, with period $2\pi i$) unbounded components, each being bounded by some curve of the form 
$$
f^{-1}_{\theta^{(n-2)}\omega}(\{Z\in\C:\Re Z=-M\}).
$$ 
Since the projection onto $Q$ identifies these components, the set $$
\mathcal C_M:=F^{-1}_{\theta^{(n-2)}\omega}(Y_M^-)\subset Q
$$
is connected, and the map $F$ restricted to $\mathcal C_M$ is infinite--to--one. Similarly, the set 
$$
F^{-1}_{\theta^{(n-2)}\omega}(Y_1^-)\spt\mathcal C_M
$$ 
is connected, and the map $F_{\theta^{(n-2)}\omega}$ restricted to $\mathcal C_1$ is infinite--to--one. 

Now, the holomorphic  branches of $F^{-(n-2)}_\omega$ are all well defined on $\mathcal C_1$ and the restriction of these branches to $\mathcal C_M$ produces all bad connected components of  $F_{\omega}^{-n}(B(0,r))$, i.e., the set $F_{\omega,B}^{-n}(B(0,r))$. Denote 
$$ 
Y(*):=Y_1^-\sms Y_{1_+}^-
=\{z\in Q:\Re z\in [-{1_+},-1]\},
$$
and partition the set $\mathcal C_1$ into subsets $\mathcal C_1^k$ by defining
$$
\mathcal C_1^k:=\big\{z\in \mathcal C_1:\Im f_{\theta^{(n-2)}\omega}(z)\in [2k\pi, 2(k+1)\pi)\big\}. 
$$
Similarly, let 
$$
\mathcal C_M^k
:=C_M\cap \mathcal C_1^k
=\big\{z\in \mathcal C_M:\Im f_{\theta^{(n-2)}\omega}(z)\in [2k\pi, 2(k+1)\pi)\big\}.
$$
Then for each $k\in\Z$  the function $f_{\theta^{n-2}\omega}$ maps $\mathcal C_1^k$ bijectively onto the region 
$$
\big\{Z\in\C: \Re Z<-1\ \text{and}\ \Im Z\in [2k\pi, 2(k+1)\pi)\big\},
$$
which we identify with $Y_1^-$. Denote by $G_k^*$ the corresponding inverse map. Then the holomorphic map 
$$
Z\longmapsto G_k(z):=G_k^*(Z+2 k\pi i)
$$ 
is in fact defined and univalent on $\{Z\in\C: \Re(Z)<-1\}$, and maps the region
$$
\big\{Z\in\C:\Re(Z)<-1\  \text{and}\  \Im Z\in [0,2\pi)\big\},
$$
which we identify with $Y_1^-$, onto $\mathcal C_1^k$, while it maps the region
$$
\big\{Z\in\C:\Re(Z)<-M\  \text{and}\  \Im Z\in [0,2\pi)\big\},
$$
which we identify with $Y_M^-$, onto $\mathcal C_M^k$.

Still keeping the identification $Q\simeq \{Z\in\C:0\le\Im Z< 2\pi\}$, we thus see that the inverse--image \
$$
F_{\omega,B}^{-n}(B(0,r))=F^{-(n-1)}_\omega(Y_M^-)
$$ 
can be expressed as
$$
\bigcup_{k\in\Z}\bigcup_g g\circ  G_k(Y_M^-),
$$
where, the second union is taken over all holomorphic branches $g$ of $F^{-(n-2)}_{\omega}$ defined on $\mathcal C_1$.

Since, as we see, each such branch $g\circ  G_k$ has a univalent holomorphic extension to the whole left half-plane
$\{Z\in\C:\Re(Z)<-1\}$, we can use Koebe's Distortion Theorem to compare  the measure $\nu_\omega (g\circ G_k)\(Y_1^-\sms Y_{1_+}^-)\)$ and  $\nu_\omega (g\circ G_k)(Y_M^-))$. 
Applying this theorem separately for each composition $g\circ  G_k$ and then summing up, with using also \eqref{bounds_on_lambda}, we obtain  that
$$
\frac{\nu_\omega(F^{-(n-1)}_\omega(Y_M^-))}{\nu_\omega\(F^{-(n-1)}_\omega\(Y_1^-\sms Y_{1_+}^-)\)\)}\preceq |M|^3\frac{\nu_{\theta^{n-1}\omega}(Y_M^-)}{{\nu_{\theta^{n-1}\omega}\(Y_1^-\sms Y_{1_+}^-\)}}.
$$
By virtue of \eqref{1_2017_04_11} and \eqref{2_2017_04_11}, this gives
$$
\begin{aligned}
\nu_\omega((F^{-n}_{\omega,B}(B(0,r)))
&=\nu_\omega(F^{-(n-1)}_\omega(Y_M^-))
\preceq\nu_\omega\(F^{-(n-1)}_\omega\(Y_1^-\sms Y_{1_+}^-\)\)\frac{|M|^3r^{u-t}}{{\nu_{\theta^{n-1}\omega}\(Y_1^-\sms Y_{1_+}^-\)}} \\
&\le \frac{|M|^3r^{u-t}}{{\nu_{\theta^{n-1}\omega}\(Y_1^-\sms Y_{1_+}^-)}}.
\end{aligned}
$$ 
The proof is now completed by invoking the bounds $\ln(1/r)+\ln A\le M\le \ln(1/r)+\ln B$ along with Corollary~\ref{c1m10} which gives
$$
\nu_{\omega}\(Y_1^-\sms Y_{1_+}^-\)\ge\Delta_-(1)>0.
$$
\end{proof}

\noindent As an immediate consequence of Lemma~\ref{lem:est_zero}, Lemma~\ref{prop:est_good}, and  Lemma~\ref{prop:around_0}, we get the following.

\begin{lem}\label{l1_2017_04_11}
We have that
$$
\nu_\omega(F^{-n}_\omega\(B(0,r))\)\preceq r^\gamma
$$
for every integer $n\ge 0$, all $\om\in \Om$ and every $r\in (0,r_0]$.
\end{lem}

\noindent We shall prove the following. 
\begin{lem}\label{prop:preimage_0}
There exists $\beta>0$ such that for every Borel set $A\subset B(0,r_0)$ and all integers $n\ge 0$ we have that
$$
\nu_\omega(F^{-n}_\omega(A))\preceq \nu_{\theta^n\omega}^\beta(A).
$$
\end{lem}
\begin{proof}
By Lemma~\ref{lem:est_zero} and Lemma~\ref{l1_2017_04_11} there exist constants $C\in(0,+\infty)$ and $D\in(0,+\infty)$ such that  
$$
\nu_{\omega}(B(0,r))\le Cr^u
$$
and 
$$
\nu_\omega(F^{-n}_\omega(B(0,r))\le D r^\gamma.
$$  
for all $r\in (0,r_0]$, almost  all $\om\in\Om$ and all integers 
$n\ge 0$. So, Since $u>6$, given such $r$, $\om$, and $n$, there exists $r\in(0,r_0]$ such that 
$$
\nu_{\theta^n\omega}(A)=C r^6.
$$
Then, 
$$\nu_\omega(F^{-n}(A\cap B(0,r)))
\le \nu_\omega(F_\omega^{-n}(B(0,r))
\le D r^\gamma
=D\left (\frac{\nu_{\theta^n\omega}(A)}{C}\right )^{\gamma/6}
= DC^{-\gamma/6}\nu_{\theta^n\omega}^{\gamma/6}(A),
$$
while using \eqref{bound_for_L}, we get
$$
\aligned
\nu_\omega(F^{-n}_\omega(A\setminus B(0,r))
&\le\sup\big\{\hL^n_{t,\omega}(\1)(z):z\in A\setminus B(0,r)\big\} \, \nu_{\theta^n\omega}(A\setminus B(0,r))\\
&\preceq  r^{-3}\sup\Big\{\hL^n_{t,\omega}(\1)(z):z\in 
Q_{M_1}\setminus \bigcup_{j=0}^NB(\beta_{n,j}(\om),r_0)\Big\}\, \nu_{\theta^n\omega}(A)\\
&\preceq c(M_1,r_0)\nu_{\theta^n\omega}^{-1/2}(A)\nu_{\theta^n\omega}(A) \\
&= c(M_1,r_0)\nu_{\theta^n\omega}^{1/2}(A).
\endaligned
$$
Thus, the statement holds with $\beta:=\min(\gamma/6, 1/2)$.
\end{proof}

\begin{lem}\label{prop:est_bad} There exists $\beta>0$ such that, if $\om\in\Omega$, $j\le N$, $n\in\mathbb N$, and 
$A\subset B(\beta_{n,j}(\om),r_0)$ is an arbitrary Borel set, then
$$
\nu_\omega(F^{-n}_{\omega,B}(A))\preceq \nu_{\theta^n\omega}^\beta(A).
$$
\end{lem}
\begin{proof}
Recall that the bad components of $F^{-n}_\omega(B(\beta_{n,j}(\om),r_0))$ are all the connected components of the set
$$
F^{-(n-j)}_{\omega}(F^{-j}_{\theta^{n-j}\omega,*}(B(\beta_{n,j}(\om),r_0))),
$$
where $F^{-j}_{\theta^{n-j}\omega,*}$ is the branch of $F^{-j}_{\theta^{n-j}\omega}$ mapping $B(\beta_{n,j}(\om),r_0)$ into $B(0,r_0)$, and $F^{-n}_{\omega,B}(A))$ is the union of all these components intersected with $F^{-n}_{\omega}(A))$.
Since, using  \eqref{bounds_on_lambda}, we obtain
$$
\nu_{\theta^{n-j}\omega}(F^{-j}_{\theta^{n-j}\omega,*}(A))
\le \max_{0\le k\le N}
\big\{K^t|(F_{\theta^{n-j}\omega}^k)'(0)|^{-t}p^k\big\} \nu_{\theta^n\omega}(A),
$$
we thus conclude the proof by applying Lemma~\ref{prop:preimage_0}.
\end{proof}

We summarize the above Lemmas~\ref{AinQ}, \ref{prop:est_good}, \ref{prop:preimage_0}, \ref{prop:est_bad} in the following.

\begin{lem}\label{AinQQ}
There exists $\beta>0$ such that
for every Borel set $A\subset Q_{M_1}$ and for every $n\ge 0$
$$
\nu_\omega(F^{-n}_\omega(A))\preceq \nu_{\theta^n\omega}^\beta(A).
$$
\end{lem}

\noindent The next proposition deals with sets contained in the complement of $Q_{M_1}$.

\begin{lem}\label{prop:estimate_in YM}
There exists $\beta>0$ such that
for every Borel set $A\subset Y_{M_1}$ and for every $n\ge 0$
$$
\nu_\omega(F^{-n}_\omega(A))\preceq \nu_{\theta^n\omega}^\beta(A).
$$
\end{lem}
\begin{proof}
First, let us  notice that using  \eqref{fixing_M1} and the bounds on $\lambda_{t,\omega}$, see \eqref{bounds_on_lambda}, we can estimate as follows.
\begin{equation}\label{onestep}
\begin{aligned}
\nu_{\theta^{n-1}\omega}(F^{-1}_{\theta^{n-1}\omega}(A))
&=\int_A\hL_{t,\theta^{n-1}\omega}(\1)(z)d\nu_{\theta^n\omega}(z)\\
&\le \frac{1}{p}\sup\big\{\L_{t,\theta^{n-1}\omega}(\1)(z)z\in Y_{M_1}\big\} \nu_{\theta^n\omega}(A)\\
&<\frac{1}{2}\nu_{\theta^n\omega}(A).
\end{aligned}
\end{equation}
Write 
$$
F^{-1}_{\theta^{n-1}\omega}(A)=A_1\cup A_2=(F^{-1}_{\theta^{n-1}\omega}(A)\cap Q_{M_1})\cup (F^{-1}_{\theta^{n-1}\omega}(A)\setminus Q_{M_1})
$$ 
and 
$$
F_\omega^{-n}(A)=F^{-(n-1)}_{\omega}(A_1)\cup F^{-(n-1)}_{\omega}(A_2).
$$
Using  \eqref{onestep}  and Lemma~\ref{AinQQ}, we have, with some positive constant $C_1$ guaranteed by Lemma~\ref{AinQQ}:
$$
\nu_{\omega}(F^{-(n-1)}_\omega(A_1))
\le C_1 \nu_{\theta^{n-1}\omega}^\beta(A_1)
\le C_1\nu_{\theta^{n-1}\omega}^\beta (F^{-1}_{\theta^{n-1}\omega}(A))
\le 2^{-\beta}C_1\nu_{\theta^n\omega}^\beta(A),
$$
while, again, 
$$
F^{-1}_{\theta^{n-2}\omega}(A_2)=A_{21}\cup 
A_{22}=(F^{-1}_{\theta^{n-2}\omega}(A_2)\cap Q_{M_1})\cup 
(F^{-1}_{\theta^{n-2}\omega}(A_2)\setminus Q_{M_1}),
$$
and 
$$
F^{-(n-1)}_{\omega}(A_2)=F^{-(n-2)}_{\omega}(A_{21})\cup F^{-(n-2)}_{\omega}(A_{22}), 
$$
where
$$
\nu_{\theta^{n-2}\omega}(A_{21})\le\frac12\nu_{\theta^{n-1}\omega}(A_2)
\le \lt(\frac12\rt)^2\nu_{\theta^n\omega}(A)
$$ 
and, similarly,
$$
\nu_{\theta^{n-2}\omega}(A_{22})\le \lt(\frac12\rt)^2\nu_{\theta^n\omega}(A).
$$
Proceeding inductively we thus obtain the following splitting.
$$
F^{-n}_\omega(A)=F^{-(n-1)}_\omega(A_1)\cup F^{-(n-2)}_\omega(A_{21})\cup F^{-(n-3)}_\omega (A_{221})\dots \cup F^{-1}_\omega(A_{22\dots 1})\cup A_{22\dots 2},
$$
where 
$$
\nu_{\theta^{n-k}\omega}(A_{22\dots 1})\le \lt(\frac12\rt)^k\nu_{\theta^n\omega}(A).
$$
Since for all sets $A_{22\dots 1}$ Lemma~\ref{AinQQ} applies, we conclude that
$$
\nu_\omega(F^{-n}_\omega(A))
\le C_1\nu_{\theta^n \omega}^\beta(A) (1+2^{-\beta}+2^{-2\beta}+\dots+2^{-n\beta}) +\lt(\frac12\rt)^n\nu_{\theta^n\omega}(A).
$$
This ends the proof, with possibly modified constant $c_1$, and the same $\beta$ as in Lemma~\ref{AinQQ}.
\end{proof}

We summarize the above lemmas as follows.

\begin{prop}\label{prop:invariant}
There exist constants $\beta>0$ and $C>0$ such that if $A\subset Q$ is a Borel set then
for every $n\in\mathbb N$ and for $m$--a.e. $\omega\in\Omega$, 
$$
\nu_\omega(F^{-n}_\omega(A))\le C\nu_{\theta^n\omega}^\beta(A).
$$
\end{prop}

Now we are position to prove the following.

\begin{thm}\label{t2im2}
For every $t>1$ there exists a random Borel probability $F$--invariant measure $\mu=\mu^{(t)}$ absolutely continuous with respect to $\nu^{(t)}$, the $t$-conformal random measure for $F:\Omega\times \C\to \Omega\times \C$, produced in Theorem~\ref{t1_2016_10_08}. Furthermore,
$$
\mu^{(t)}(A)=\ell_B\big((\nu^{(t)}\circ F^{-n}(A))_{n=0}^\infty\big),
$$
where $\ell_B:\ell_\infty\to \R$ is a (fixed) Banach limit on $\ell_{\infty}$.
\end{thm}

\begin{proof}
It is well--known in abstract ergodic theory that all assertions of Theorem~\ref{t2im2}, perhaps except that $\mu\in \mathcal M_m$, would follow from uniform absolute continuity of measures $\big(\nu\circ F^{-n}\big)_{n=0}^\infty$ with respect to measure $\nu$. In order to prove this uniform continuity, fix $\varepsilon>0$ and suppose that $A\subset \Omega\times Q$ is a measurable set such that $\nu(A)<\varepsilon^2$. We then get for every integer $n\ge 0$ that
$$
\begin{aligned}
\nu(F^{-n}(A))
&=\int_\Omega\nu_\om(F_\om^{-n}(A_{\th^n(\om)}))\,dm(\om) \\
&=\int_{\Omega_0}\nu_\om(F_\om^{-n}(A_{\th^n(\om)}))\,dm(\om)+
  \int_{\Omega_0^c}\nu_\om(F_\om^{-n}(A_{\th^n(\om)}))\,dm(\om),
\end{aligned}
$$
where
$$
\Om_0:=\big\{\om\in \Om:\nu_{\th^n(\om)}(A_{\th^n(\om)})\big\}\ge \ve.
$$
But then
$$
m(\Om_0)=m\(\{\om\in \Om:\nu_\om(A_\om)\ge \ve\}\)\le \nu(A)/\ve.
$$
So, applying Proposition~\ref{prop:invariant}, we get that
$$
\nu(F^{-n}(A))
\le \frac{\nu(A)}{\ve}+C\ve^\b
\le \ve+\ve^\b,
$$
and the required uniform absolute continuity has been proved. In order to see that $\mu\in\mathcal M_m$, let $\Gamma$ be an arbitrary Borel subset of $\Omega$. Then
$$
\begin{aligned}
\mu(\Gamma\times\C)
&=\ell_B\((\nu\circ F^{-n}(\Gamma\times\C))_{n=0}^\infty\)
=\ell_B\((\nu(\th^{-n}(\Gamma)\times\C))_{n=0}^\infty\) \\
&=\ell_B\((m(\th^{-n}(\Gamma)))_{n=0}^\infty\)
=\ell_B\((m(\Gamma))_{n=0}^\infty\) \\
&=m(\Gamma).
\end{aligned}
$$
This means that $\mu\in \mathcal M_m$, and the proof is complete.
\end{proof}

\noindent We can prove more about the invariant measure $\mu^{(t)}$. Namely:

\begin{thm}\label{t1im3}
Let $t>1$. If $\nu^{(t)}$ is the $t$-conformal random measure for $F:\Omega\times \C\to \Omega\times \C$, produced in Theorem~\ref{t1_2016_10_08}, then the Borel probability $F$--invariant measure $\mu=\mu^{(t)}\in \mathcal M_m$ absolutely continuous with respect to $\nu^{(t)}$, produced in Theorem~\ref{t2im2}, is in fact equivalent with $\nu^{(t)}$.
\end{thm}

\begin{proof}
Since $\lim_{n\to\infty}F_\om^n(0)=+\infty$ uniformly with respect to $\om\in \Om$ and since each measure $\mu_\om$ is a probability one satisfying, by virtue of $F$-invariance, $\mu_{\th(\om)}(F_\om(A))\ge \mu_\om(A)$ for every $\om\in\Om$ and every Borel set $A\subset \C$, we have that
$$
\mu_\om\(\big\{F_{\th^{-n}(\om)}^n(0):n\ge 0\big\}\)=0
$$
for $m$--almost all $\om\in \Om$. Therefore,
$$
\mu\lt(\bu_{\om\in\Om}\{\om\}\times
\big\{F_{\th^{-n}(\om)}^n(0):n\ge 0\big\}\rt)=0
$$
Hence, there exists $R\in(0,r_0/2)$ so small that
$$
\mu\lt(\bu_{\om\in\Om}\{\om\}\times\bu_{n=0}^\infty
B\(F_{\th^{-n}(\om)}^n(0),2R\)\rt)<1/8.
$$
Hence, there exists a measurable set $\Om_0\sbt \Om$ such that
\begin{equation}\label{1im4}
m(\Om_0)\ge 1/2
\end{equation}
and
\begin{equation}\label{2im4}
\mu_\om\lt(\bu_{n=0}^\infty B\(F_{\th^{-n}(\om)}^n(0),2R\)\rt)<1/4
\  \  \  {\rm for~ all }  \  \  \
\om\in\Om_0.
\end{equation}
Now, there exists a constant $M>0$ so large that $\mu(\Om\times Y_M)<1/8$, and therefore there exists a measurable set $\Om_1\sbt \Om_0$ such that 
$$
m(\Om_1)\ge 1/4
$$
and 
$$
\mu_\om(Q_M)\ge 1/2
\  \  \  {\rm for \ all }  \  \  \  \om\in\Om_1.
$$
Combining this along with \eqref{2im4}, we conclude that there exists $\alpha>0$ and for every $\om\in\Om_1$ there exists
$$
\xi_\om\in Q_M\cap\lt(\C\sms \bu_{n=0}^\infty B\(F_{\th^{-n}(\om)}^n(0),2R\)\rt)
$$
such that 
\begin{equation}\label{3im4}
\mu_\om(B(\xi_\om,R))\ge \alpha
\end{equation}
and the choice $\Om_1\ni\om\mapsto\xi_\om$ is measurable. Let
$$
\Ga:=\bu_{\om\in\Om_1}\{\om\}\times B(\xi_\om,R).
$$
Of course
$$
\mu(\Ga)\ge \alpha/8>0.
$$
We shall prove the following

\medskip{\bf Claim~$1^0$:} If $A\sbt\Gamma$ is a measurable set and $\nu(A)>0$, then $\mu(A)>0$.

\begin{proof}
Because of our definition of the set $\Ga$, for every $\om\in\Om_1$, every integer $n\ge 0$, and every $\xi\in F_{\th^{-n}(\om)}^{-n}(\xi_\om)$, we have that
$$
\begin{aligned}
\nu_{\th^{-n}(\om)}\(F_{\th^{-n}(\om),\xi}^{-n}(A_\om)\)
&=\lam_{\th^{-n}(\om)}^{-n}
\int_{A_\om}\big|\(F_{\th^{-n}(\om),\xi}^{-n}\)'\big|^t\,d\nu_\om\\
&\ge K^{-t}\lam_{\th^{-n}(\om)}^{-n}
\big|\(F_{\th^{-n}(\om)}^{-n}\)'(\xi)\big|^{-t}\nu_\om(A_\om),
\end{aligned}
$$
while
$$
\nu_{\th^{-n}(\om)}\(F_{\th^{-n}(\om),\xi}^{-n}(B(\xi_\om,R))\)
\le  K^t\lam_{\th^{-n}(\om)}^{-n}
\big|\(F_{\th^{-n}(\om)}^{-n}\)'(\xi)\big|^{-t}\nu_\om(B(\xi_\om,R)).
$$
Therefore,
$$
\begin{aligned}
\nu_{\th^{-n}(\om)}\(F_{\th^{-n}(\om),\xi}^{-n}(A_\om)\)
&\ge K^{-2t}\nu_{\th^{-n}(\om)}\(F_{\th^{-n}(\om),\xi}^{-n}
 (B(\xi_\om,R))\)\frac{\nu_\om(A_\om)}{\nu_\om(B(\xi_\om,R))}\\
&\ge K^{-2t}\nu_{\th^{-n}(\om)}\(F_{\th^{-n}(\om),\xi}^{-n}
 (B(\xi_\om,R))\)\nu_\om(A_\om)
\end{aligned}
$$
Now notice that if 
$$
\Om_*:=\lt\{\om\in\Om_1:\nu_\om(A_\om)\ge \frac12\nu(A)\rt\},
$$
then
$$
m(\Om_*)>0.
$$
Therefore, for every $\om\in\Om_*$, we get
$$
\begin{aligned}
\nu_{\th^{-n}(\om)}\(F_{\th^{-n}(\om)}^{-n}(A_\om)\)
&=\sum_{\xi\in F_{\th^{-n}(\om)}^{-n}(\xi_\om)}
   \nu_{\th^{-n}(\om)}\(F_{\th^{-n}(\om),\xi}^{-n}(A_\om)\) \\
&\ge \frac12K^{-2t}\nu(A)\sum_{\xi\in F_{\th^{-n}(\om)}^{-n}(\xi_\om)}
   \nu_{\th^{-n}(\om)}\(F_{\th^{-n}(\om),\xi}^{-n}(B(\xi_\om,R))\) \\
&=\frac12K^{-2t}\nu(A)\nu_{\th^{-n}(\om)}
   \(F_{\th^{-n}(\om)}^{-n}(B(\xi_\om,R))\)
\end{aligned}
$$
Hence, we obtain
$$
\begin{aligned}
\nu(F^{-n}(A))
&=\int_\Omega\nu_{\th^{-n}(\om)}(F_{\th^{-n}(\om)}^{-n}
       (A_\om))\,dm(\om) \\
&\ge \int_{\Omega_*}\nu_{\th^{-n}(\om)}(F_{\th^{-n}(\om)}^{-n}
       (A_\om))\,dm(\om) \\
&\ge \frac12K^{-2t}\nu(A)\int_{\Omega_*}\nu(A)\nu_{\th^{-n}(\om)}
   \(F_{\th^{-n}(\om)}^{-n}(B(\xi_\om,R)) \\
&=\frac12K^{-2t}\nu(A)\lt(F^{-n}\lt(\bu_{\om\in\Om_*}\{\om\}\times B(\xi_\om,R)\rt)\rt).
\end{aligned}
$$
Finally, using \eqref{3im4}, we get
$$
\begin{aligned}
\mu(A)
&=\ell_B\big((\nu(F^{-n}(A)))_{n=0}^\infty\big)\\
&\ge\frac12K^{-2t}\nu(A)
   \ell_B\lt(\lt(\nu\lt(F^{-n}\lt(\bu_{\om\in\Om_*}\{\om\}\times 
    B(\xi_\om,R)\rt)\rt)\rt)_{n=0}^\infty\rt)\\
&=\frac12K^{-2t}\mu\lt(\bu_{\om\in\Om_*}\{\om\}\times 
    B(\xi_\om,R)\rt)\nu(A)\\
&\ge \frac12K^{-2t}\alp\nu(A)>0,
\end{aligned}
$$
and the Claim is proved.
\end{proof}

\medskip Now we conclude the proof of Theorem~\ref{t1im3}. So, let $D\subset \Om\times\C$ be an arbitrary Borel set with $\nu(D)>0$.
Then there exist a measurable set $\Om_2\sbt \Om$ and $\eta\in(0,1/2)$ such that $m(\Om_2)>0$ and for every $\om\in\Om_2$ there exists $x_\om\in A(0;2\eta,1/\eta)$, depending measurably on $\om$, such that
\begin{equation}\label{1im6}
\nu_\om(D_\om\cap B(x_\om,\eta))>0.
\end{equation}
Denote the ball $B(x_\om,\eta)$ just by $B_\om$. From our hypotheses  on the functions $f_\om$, $\om\in\Om$, there exists an integer $N\ge 0$ such that
$$
F_\om^n(B(z,R))\spt \bu_{x\in A(0;2\eta,1/\eta)}B(x,\eta)
$$
for all $\om\in\Om$, all $n\ge N$, and all $z\in Q_M$. Since, $m(\Om_1), m(\Om_2)>0$ and since the map $\th:\Om\to\Om$ is ergodic, there exists $n\ge N$ such that
$$
m(\Om_1\cap\th^{-n}(\Om_2))>0.
$$
Then 
\begin{equation}\label{1im7}
m(\th^n(\Om_1)\cap\Om_2)>0,
\end{equation}
and
$$
\begin{aligned}
F^n(\Ga)
&\spt F^n\lt(\bu_{\om\in \Om_1\cap\th^{-n}(\Om_2)}\{\om\}\times B(\xi_\om,R)\rt) \\
&\spt \bu_{\om\in \th^n(\Om_1)\cap\Om_2}\{\om\}\times B_\om \\
&\spt \bu_{\om\in \th^n(\Om_1)\cap\Om_2}\{\om\}\times (D_\om\cap B_\om).
\end{aligned}
$$
Therefore there exists a measurable set $H\sbt \Ga$ such that
\begin{equation}\label{2im7}
F^n(H)
=\bu_{\om\in \th^n(\Om_1)\cap\Om_2}\{\om\}\times (D_\om\cap B_\om)
\sbt D.
\end{equation}
But then, because of \eqref{1im6} and \eqref{1im7}, we have that $\nu(F^n(H))>0$. This in turn, by conformality of $\nu$, yields $\nu(H)>0$. Since $H\sbt\Ga$, it then follows from Claim~$1^0$ that $\mu(H)>0$. Hence, by virtue of \eqref{2im7}, we get that $\mu(D)\ge 
F^n(H)\ge \mu(H)>0$. The proof of Theorem~\ref{t1im3} is thus complete. 
\end{proof}

\

We shall prove more about measures $\mu_t$: their ergodicity and uniqueness. This however requires some preparation. 

\medskip Fix $(\om,z)\in\Om\times Q$. Let $N\in\N$. Define $N_\om(z,N)$ to be the set of all integers $n\ge 0$ such that there exists a (unique) holomorphic inverse branch 
$$
F_{\om,z}^{-n}:B(F_\om^n(z),2/N)\to Q
$$
of $F_\om^n:Q\to Q$ sending $F_\om^n(z)$ to $z$ and such that $|F_\om^n(z)|\le N$. Following a number theory tradition, given a set $A\sbt \N$, we denote by $\un\rho(A)$ and $\ov\rho(A)$ respective lower and upper densities of the set $A$. Precisely,
$$
\un\rho(A):=\varliminf_{n\to\infty}\frac1N\#(A\cap\{1,2,\ld, N\})
$$
and 
$$
\ov\rho(A):=\varlimsup_{n\to\infty}\frac1N\#(A\cap\{1,2,\ld, N\})
$$
We define
\begin{equation}\label{def:J_r}
J_r(\om)
:=\big\{z\in Q: \lim_{N\to\infty}\un\rho(N_\om(z,N))=1\big\}.
\end{equation}
$J_r(\om)$ is said to be the set of radial (or conical) points of $F$ at $\om$. We further denote:
$$
J_r(F):=\bu_{\om\in\Om}\{\om\}\times J_r(\om),
$$
and call it the set of all radial points of $F$. We will need some sufficient conditions for a point $(\om,z)$ to be radial. In order to formulate it, we need an auxiliary subset $\tilde N_\om(z,N)$ of $N_\om(z,N)$. It consists of all integers $n\ge 0$ such that for every integer $0\le k\le n$,
$$
F_{\th^{n-k}(\om)}^k(0)\notin B\(F_\om^n(z),2/N\) 
\ \ {\rm and }  \  \
|F_\om^n(z)|\le N
$$
Of course 
\begin{equation}\label{1im8.1}
\tilde N_\om(z,N)\sbt N_\om(z,N).
\end{equation}
Also, $n\in \tilde N_\om(z,N)$ if and only if 
$$
F_\om^n(z)\in Q_N  
$$
and
$$
F_\om^n(z)
\notin \bu_{k=0}^nB\(F_{\th^{n-k}(\om)}^k(0),2/N\)
=\bu_{k=0}^nB\(F_{\th^{-k}(\th^n(\om))}^k(0),2/N\).
$$
Therefore, if we denote
$$
J_N^*(F)
:=\bu_{\om\in\Om}\{\om\}\times \lt(Q_N\sms \bu_{k=0}^\infty B\lt(F_{\th^{-k}(\om)}^k(0),2/N\rt)\rt)
$$
and 
$$
\tilde N_\om^*(z,N):=\big\{n\ge 0:F_\om^n(z)=F^n(\om,z)\in J_N^*(F)\big\},
$$
then
\begin{equation}\label{1im8.2}
\tilde N_\om^*(z,N)\sbt \tilde N_\om(z,N).
\end{equation}

The first significance of the set of radial points comes from the following.

\begin{prop}\label{p1im8}
If $\mu\in\mathcal M_m$ is $F$--invariant, then $\mu(J_r(F))=1$.
\end{prop}

\begin{proof}

By considering ergodic decomposition, we may assume without loss of generality that measure $\mu$ is ergodic. By virtue of \eqref{1im8.1} and \eqref{1im8.2} it suffices to show that
$$
\lim_{N\to\infty}\mu\(J_N^*(F)\)=1.
$$
And indeed, let
$$
J_N^*(F)^c:=(\Om\times Q)\sms J_N^*(F)
$$
be the complement of $J_N^*(F)$ in $\Om\times Q$. Then
$$
J_N^*(F)^c
=\bu_{\om\in\Om}\{\om\}\times \lt(Y_N\cup \bu_{k=0}^\infty B\lt(F_{\th^{-k}(\om)}^k(0),2/N\rt)\rt)
$$
and $\(J_N^*(F)^c\)_{N=1}^\infty$ is a descending sequence of measurable sets with
\begin{equation}\label{2im8.1}
\bi_{N=1}^\infty J_N^*(F)^c
=\bu_{\om\in\Om}\{\om\}\times\big\{F_{\th^{-k}(\om)}^k(0):k\ge 0\big\}.
\end{equation}
But 
$$
\begin{aligned}
F\lt(\bi_{N=1}^\infty J_N^*(F)^c\rt)
&=\bu_{\om\in\Om}\{\th(\om)\}\times\big\{F_{\th^{-k}(\om)}^{k+1}(0):k\ge 0\big\} \\
&=\bu_{\om\in\Om}\{\th(\om)\}\times\big\{F_{\th^{-(k+1)}(\th(\om))}^{k+1}(0):k\ge 0\big\} \\
&\sbt\bi_{N=1}^\infty J_N^*(F)^c,
\end{aligned}
$$
hence by ergodicity of $\mu$, 
$$
\mu\lt(\bi_{N=1}^\infty J_N^*(F)^c\rt)\in\{0,1\}.
$$
If 
$$
\mu\lt(\bi_{N=1}^\infty J_N^*(F)^c\rt)=0,
$$
we are done. So, suppose that
\begin{equation}\label{3im8.1}
\mu\lt(\bi_{N=1}^\infty J_N^*(F)^c\rt)=1.
\end{equation}
Then for $m$--a.e. $\om\in\Om$, say $\om\in\Om^*$, with $\Om^*$ being $\th$-invariant, 
$$
\mu_\om\lt(\big\{F_{\th^{-k}(\om)}^k(0):k\ge 0\big\}\rt)=1.
$$
But as $\mu_\om\circ F_\om^{-1}=\mu_{\th(\om)}$, we then get that
$$
\begin{aligned}
\mu_{\th(\om)}\lt(\big\{F_{\th^{-(k+1)}(\th(\om))}^{k+1}(0):k\ge 0\big\}\rt)
&=\mu_{\th(\om)}\lt(\big\{F_\om\lt(F_{\th^{-k}(\om)}^k(0):k\ge 0\big\}\rt)\rt) \\
&\ge\mu_\om\lt(\big\{F_{\th^{-k}(\om)}^k(0):k\ge 0\big\}\rt)\\
&=1.
\end{aligned}
$$
Hence,
$$
\mu_\om(F_{\th^{-1}(\om)}(0))=0
$$
for all $\om\in\Om^*$. Proceeding in the same way by induction, we deduce that
$$
\mu_\om\(F_{\th^{-k}(\om)}^k(0)\)=0
$$
for every integer $k\ge 0$ and all $\om\in\Om^*$. Thus
$$
\mu_\om\lt(\big\{F_{\th^{-k}(\om)}^k(0):k\ge 0\big\}\rt)=0
$$
for all $\om\in\Om^*$. By \eqref{2im8.1} this entails
$$
\mu\lt(\bi_{N=1}^\infty J_N^*(F)^c\rt)=0,
$$
contrary to \eqref{3im8.1}. The proof of Proposition~\ref{p1im8} is complete.
\end{proof}

We now pass to consider random conformal measures and we do this with their relations to the set of radial points. Let $t>1$ and suppose we are given two $t$-conformal measures $\nu^{(1)}$ and $\nu^{(2)}$. Denote by $\lambda_\omega^{(1)}$ and  $\lambda_\omega^{(2)}$ the corresponding normalizing factors coming from the definition  of a conformal measure.  For every $l>0$ and $\om\in\Om$ let
\begin{equation}\label{1im10}
L_\om(l)
:=\lt\{n\ge 1:\frac{\lam_\om^{(1)n}}{\lam_\om^{(2)n}}\le l\rt\}\sbt\N.
\end{equation}
Let $\hat\Om_l$ be the set of all points $\om\in \Om$ such that the set $L_\om(l)\sbt\N$ has positive upper density. Finally let 
$$
\hat\Om:=\bu_{l=1}^\infty\hat\Om_l.
$$
We shall prove the following.

\begin{lem}\label{l3im8}
If $t>1$ and two $t$-conformal measures $\nu^{(1)}$ and $\nu^{(2)}$ are given, then for every $m$-a.e. $\om\in \hat\Om$, the fiber measure $\nu_\om^{(2)}|_{J_r(\om)}$ is absolutely continuous with respect to the fiber measure $\nu_\om^{(1)}|_{J_r(\om)}$.
\end{lem}

\begin{proof}
Fix an integer $l\ge 1$ and then an integer $q\ge 1$. By $t$--conformality and quasi topological exactness of the map $F:\Om\times Q\to\Om\times Q$, each measure $\nu_\om^{(i)}$, $i=1,2$, $\om\in\Om$, has full topological support, i.e. it is positive on all non-empty open subsets of $Q$. Therefore, for every $\N\in\N$, we have that
$$
M_N^{(i)}(\om):=\inf\big\{\nu_\om^{(i)}\(B(z,(4KN)^{-1})\):z\in Q_N\big\}>0,
$$
and the function
$$
\Om\ni\om\longmapsto M_N^{(i)}(\om)\in(0,+\infty)
$$
is measurable. Hence, for every integer $k\ge 1$ there exists $\ve_{N,k}^{(i)}>0$ so small that
$$
m\lt(M_N^{(i)-1}\((\ve_{N,k}^{(i)},+\infty)\)\rt)
>1-\frac1{2k}.
$$
By Birkhoff's Ergodic Theorem, for $m$--a.e. $\om\in\Om$, say $\om$ in some $\th$--invariant set $\Om_{N,k}^{(i)}$ with measure $m$ equal to $1$, we have that
\begin{equation}\label{2im11}
\rho\lt(\La_{N,k}^{(i)}(\om)\rt)
=m\lt(M_N^{(i)-1}\((\ve_{N,k}^{(i)},+\infty)\)\rt)
>1-\frac1{2k},
\end{equation}
where
$$
\La_{N,k}^{(i)}(\om)
:=\lt\{n\ge 0:\th^n(\om)\in M_N^{(i)-1}\((\ve_{N,k}^{(i)},+\infty)\)\rt\} 
\sbt \N.
$$
Let $\hat\Om_{l,q}$ be the set of all points $\om\in \Om$ such that the set $\ov\rho(L_\om(l))\ge 1/q$. Of course 
$$
\hat\Om_l=\bu_{q=1}^\infty \hat\Om_{l,q}.
$$
It therefore suffices to prove our lemma with the set $\hat\Om$ replaced by $\hat\Om_{l,q}$. In order to do this we shall estimate from above the limit
$$
\varliminf_{r\to 0}\frac{\nu_\om^{(2)}(B(z,r))}{\nu_\om^{(1)}(B(z,r))}
$$
for all $\om\in \hat\Om_{l,q}$ and all $z\in J_r(\om)$. So, fix $N_q\ge 1$ so large that
\begin{equation}\label{3im10}
\un\rho\(N_\om(z,N_q)\)>1-\frac1{2q}.
\end{equation}
It then follows from $\frac14$--Koebe's Distortion Theorem, Koebe's Distortion Theorem, and $t$-conformality of measure $\nu^{(1)}$ that for every $n\in N_\om(z,N_q)\cap L_\om(l)$, we have that
\begin{equation}\label{4im10}
\begin{aligned}
\nu_\om^{(2)}\bigg(B\bigg(z,\frac14\frac1{N_q}
    \big|\(F_\om^n\)'(z)\big|^{-1}\bigg)\bigg)
&\le \nu_\om^{(2)}\lt(F_{\om,z}^{-n}\(B\(F_\om^n(z),1/N_q\)\)\rt) \\
&\le K^t\lam_\om^{(2)-n}\big|\(F_\om^n\)'(z)\big|^{-t}
\nu_{\th^n(\om)}^{(2)}\(B\(F_\om^n(z),1/N_q\)\) \\
&\le K^t\lam_\om^{(2)-n}\big|\(F_\om^n\)'(z)\big|^{-t}
\end{aligned}
\end{equation}
By the same token,
\begin{equation}\label{5im10}
\begin{aligned}
\nu_\om^{(1)}\bigg(B\bigg(z,\frac14\frac1{N_q}
    \big|\(F_\om^n\)'(z)\big|^{-1}\bigg)\bigg)
&\ge\nu_\om^{(1)}\lt(F_{\om,z}^{-n}\(B\(F_\om^n(z),(4KN_q)^{-1}\)\)\rt) \\
&\ge K^{-t}\lam_\om^{(1)-n}\big|\(F_\om^n\)'(z)\big|^{-t}
\nu_{\th^n(\om)}^{(1)}\(B\(F_\om^n(z),(4KN_q)^{-1}\)\).
\end{aligned}
\end{equation}
Now assume in addition that 
$$
\om\in \Om_{N_q,q}^{(1)}. 
$$
Then, we deduce from \eqref{3im10} and \eqref{2im11} that
$$
\ov\rho\lt(N_\om(z,N_q)\cap L_\om(l)\cap \La_{N_q,q}^{(1)}\rt)>0.
$$
Therefore, for ever $n\in N_\om(z,N_q)\cap L_\om(l)\cap \La_{N_q,q}^{(1)}$, we get that
\begin{equation}\label{1im11}
\begin{aligned}
\frac{\nu_\om^{(2)}\bigg(B\bigg(z,\frac14\frac1{N_q}
    \big|\(F_\om^n\)'(z)\big|^{-1}\bigg)\bigg)}{\nu_\om^{(1)}\bigg(B\bigg(z,\frac14\frac1{N_q}
    \big|\(F_\om^n\)'(z)\big|^{-1}\bigg)\bigg)} 
&\le K^t\lt(\nu_{\th^n(\om)}^{(1)}\(B\(F_\om^n(z),(4KN_q)^{-1}\)\)\rt)^{-1}\frac{\lam_\om^{(1)n}}{\lam_\om^{(2)n}} \\
& \le K^t \(\ve_{N_q,q}^{(1)}\)^{-1}l.
\end{aligned}
\end{equation}
Consequently,
$$
\varliminf_{r\to 0}
\frac{\nu_\om^{(2)}(B(z,r))}{\nu_\om^{(1)}(B(z,r))}
\le \varliminf_{n\to\infty}
\frac{\nu_\om^{(2)}\bigg(B\bigg(z,\frac14\frac1{N_q}
    \big|\(F_\om^n\)'(z)\big|^{-1}\bigg)\bigg)}{\nu_\om^{(1)}\bigg(B\bigg(z,\frac14\frac1{N_q}
    \big|\(F_\om^n\)'(z)\big|^{-1}\bigg)\bigg)} 
\le K^t \(\ve_{N_q,q}^{(1)}\)^{-1}l.
$$
This implies that for each $\om\in\hat\Om_{l,q}\cap\Om_{N_q,q}^{(1)} $, the measure $\nu_\om^{(2)}|_{J_r(\om)}$ is absolutely continuous with respect to $\nu_\om^{(1)}|_{J_r(\om)}$, and the proof of Lemma~\ref{l3im8} is complete.
\end{proof}

Our ultimate theorem about conformal and invariant measures is this.

\begin{thm}\label{t1im3B}
Let $t>1$. If $\nu^{(t)}$ is the $t$-conformal random measure for $F:\Omega\times \C\to \Omega\times \C$, produced in Theorem~\ref{t1_2016_10_08}, then the Borel probability $F$--invariant measure $\mu=\mu^{(t)}\in \mathcal M_m$ absolutely continuous with respect to $\nu^{(t)}$, produced in Theorem~\ref{t2im2}, is in fact 

\medskip\begin{itemize}
\item[(a)] Equivalent with $\nu^{(t)}$,

\medskip\item[(b)] Ergodic,

\medskip\item[(c)] It is the only Borel probability $F$--invariant measure in $\mathcal M_m$ absolutely continuous with respect to $\nu^{(t)}$. 
\end{itemize}
\end{thm}

\begin{proof}
Item (a) is just Theorem~\ref{t1im3}. In order to prove ergodicity of $\mu$, i.e. item (b) of Theorem~\ref{t1im3B}, assume for a contradiction that there are two disjoint totally $F$-invariant measurable sets $A, B\sbt \Om\times\C$ such that 
$$
0<\mu(A),\, \mu(B)<1.
$$
Since $\th:\Om\to\Om$ is ergodic with respect to measure $m$, we have that $0<\mu_\om(A_\om), \mu_\om(B_\om)<1$ for $m$--ae. $\om\in\Om$. Therefore, also 
$$
0<\nu_\om(A_\om),\, \nu_\om(B_\om)<1
$$ 
for $m$--ae. $\om\in\Om$. Define two random measures measures $\hat\nu_A$ and $\hat\nu_B$ by demanding that their fiber measures
$\hat\nu_{A,\om}$ and $\hat\nu_{B,\om}$ are respective conditional measures of the measure $\nu_\om$ on the sets $A_\om$ and $B_\om$. By this very definition both $\hat\nu_A$ and $\hat\nu_B$ belong to $\mathcal M_m$. It is easy to verify that these two measures are also $t$--conformal with respective generalized eigenvalues equal to
$$
\lam_{A,\om}=\lam_\om\frac{\nu_\om(A_\om)}{\nu_\om(A_{\th(\om)})}
$$
and 
$$
\lam_{B,\om}=\lam_\om\frac{\nu_\om(B_\om)}{\nu_\om(B_{\th(\om)})}.
$$
But then
$$
\lam_{A,\om}^n
=\lam_\om^n\frac{\nu_\om(A_\om)}{\nu_\om\(A_{\th^n(\om)}\)}
$$
and 
$$
\lam_{B,\om}^n
=\lam_\om^n\frac{\nu_\om(B_\om)}{\nu_\om\(B_{\th^n(\om)}\)}
$$
for every integer $n\ge 0$. Therefore
$$
\frac{\lam_{A,\om}^n}{\lam_{B,\om}^n}
=\frac{\nu_\om(A_\om)}{\nu_\om(B_\om)}\cdot  \frac{\nu_\om\(B_{\th^n(\om)}\)}{\nu_\om\(A_{\th^n(\om)}\)}
\le \frac{1}{\nu_\om(B_\om)}\cdot \frac{1}{\nu_\om\(A_{\th^n(\om)}\)}.
$$
Now, since $\nu(A)>0$, there exists $\ve>0$ such that
$$
m\(\{\om\in\Om:\nu_\om(A_\om)\ge \ve\}\)>0.
$$
Denote this, just defined, subset of $\Om$ by $\Om^*$. By Birkhoff's Ergodic Theorem and ergodicity of the measure $m$ with respect to the map $\th:\Om\to\Om$, we have for $m$--a.e. $\om\in\Om$, say $\om\in\Om^+$, that 
$$
\rho\(\{n\ge 0:\th^n(\om)\in \Om^*\}\)=m(\Om^*)>0.
$$
For every $k\ge 1$ let 
$$
\Om_k:=\{\om\in\Om:\nu_\om(B_\om)\ge 1/k\}.
$$
Then $\Om_k\cap\Om^+\sbt \hat\Om_{k/\ve}\sbt \hat\Om$. Hence 
$$
\bu_{k=1}^\infty \Om_k\cap\Om^+\sbt \hat\Om.
$$
Since also $m\(\bu_{k=1}^\infty \Om_k\cap\Om^+\)=1$, it thus follows from Lemma~\ref{l3im8} that the fiber measure $\hat\nu_{B,\om}|_{J_r(\om)}$ is absolutely continuous with respect to the fiber measure $\hat\nu_{A,\om}|_{J_r(\om)}$ for $m$--a.e. $\om\in\Om$. But because of Proposition~\ref{p1im8} and Theorem~\ref{t1im3}, $\nu_\om(J_r(\om))=1$ for $m$--a.e. $\om\in\Om$; consequently
$\hat\nu_{B,\om}(J_r(\om))=\hat\nu_{A,\om}(J_r(\om)))=1$ for $m$--a.e. $\om\in\Om$.
We thus obtained that the fiber measure $\hat\nu_{B,\om}$ is absolutely continuous with respect to the fiber measure $\hat\nu_{A,\om}$ for $m$--a.e. $\om\in\Om$. This contradicts the fact that $A_\om\cap B_\om=\es$ for $m$--a.e. $\om\in\Om$, and finishes the proof of item (b), i.e. ergodicity of the measure $\mu$.

\medskip The proof of item (c) is now straightforward. Assume for a contradiction that there exists an $F$-invariant Borel probability measure on $\Om\times Q$ absolutely continuous with respect to $\nu$ and different from $\mu$. Then there also exists an ergodic measure $\eta$ with all such properties. But then by (a), $\eta$ is  absolutely continuous with respect to $\mu$. As both measures $\eta$ and $\mu$ are ergodic, we thus conclude that $\mu=\eta$. This contradiction finishes the proof of item (c) and simultaneously the whole proof of Theorem~\ref{t1im3B}.
\end{proof}

\noindent As an immediate consequence of Proposition~\ref{p1im8} and Theorem~\ref{t1im3B}, we get the following.

\begin{cor}\label{c15_2017_06_19} 
For every $t>1$ we have that $\nu^{(t)}(J_r(F))=1$.
\end{cor} 

As the last important fact in this section, we shall prove the following.
\begin{prop}\label{prop:lyapunov}
For every $t>1$ the global Lyapunov exponent 
$$
\chi_{\mu^{(t)}}
:=\int_{\Omega\times Q} \log|F'_\omega(z)\,d\mu^{(t)}(\om,z) 
=\int_{\Omega\times Q} \log|f'_\omega(z)\,d\mu^{(t)}(\om,z) 
$$
is finite and positive.
\end{prop}

\begin{proof}
We first note that 
$$
\chi_{\mu^{(t)}}
=\int_{\Omega\times Q} \log|f_\omega(z)\,d\mu^{(t)}(\om,z)
=\int_{\Omega\times Q}\(\log\eta(\om)+\Re(z)\)\,d\mu^{(t)}(\om,z).
$$
Since $\log A\le \log\eta(\om)\le \log B$ for all $\om\in\Om$ and since $\mu^{(t)}$ is a probability measure, we are thus to show that
$$
\int_{\Omega\times Q} |\Re(z)|\,d\mu^{(t)}(\om,z)<+\infty.
$$
in order to do this, we will provide sufficiently good upper estimates for $\mu_\om^{(t)}(Y_M^{\pm})$ for all $M\ge 0$ and all $\om\in\Om$. First, using \eqref{3.1} and Proposition~\ref{prop:invariant}, we have 
$$
\nu_\om^{(t)}\(F_\om^{-n}(Y_M^+)\)\le c^\b(M_0)e^{\frac{\b M}{2}(1-t)}
$$
for every integer $n\ge 0$ and every real number $M>0$. Second, by Proposition~\ref{prop:invariant} again and by Proposition~\ref{prop:left_tight} there are two constants $D>0$ and $\gamma>0$ such that
$$
\nu_\om^{(t)}\(F_\om^{-n}(Y_M^-)\)\le De^{-\gamma M}
$$
for every integer $n\ge 0$ and every real number $M>0$. Therefore, 
$$
\nu^{(t)}\(F^{-n}(\Om\times Y_M^+)\)
=\int_\Om\nu_\om^{(t)}\(F_\om^{-n}(Y_M^+)\)\,dm(\om)
\le c^\b(M_0)e^{\frac{\b M}{2}(1-t)}
$$
and likewise,
$$
\nu^{(t)}\(F^{-n}(\Om\times Y_M^-)\)\le De^{-\gamma M}. 
$$
It therefore follows from Theorem~\ref{t2im2} and basic properties of Banach limits that
\begin{equation}\label{1m4}
\mu^{(t)} (\Om\times Y_M^+)
\le c^\b(M_0)e^{\frac{\b M}{2}(1-t)} 
\  \  \  {\rm and} \  \  \
\mu^{(t)} (\Om\times Y_M^-)
\le De^{-\gamma M}.
\end{equation}
Hence, by straightforward calculation:
$$
\int_{\Omega\times Y_1^+} |\Re(z)|\,d\mu^{(t)}(\om,z)
<+\infty.
$$
In the same way, based on the right--hand side of \eqref{1m4}, we get
$$
\int_{\Omega\times Y_1^-} |\Re(z)|\,d\mu^{(t)}(\om,z)<+\infty.
$$
Since obviously,
$$
\int_{\Omega\times Q_1} |\Re(z)|\,d\mu^{(t)}(\om,z)\le 1,
$$
we thus conclude that
$$
\int_{\Omega\times Q} |\Re(z)|\,d\mu^{(t)}(\om,z)<+\infty,
$$
and the proof of finiteness of the global Lyapunov exponent 
$\chi_{\mu^{(t)}}$ is complete.

So, we now pass to the proof that $\chi_{\mu^{(t)}}>0$. The first observation is that for each $\om\in\Om$ the set 
$$
D_\om:=\{z\in Q:|\Im(f_\om(z))|>2\}
$$
is non-empty 
and open. Therefore 
$
\mu^{(t)}(D)>0,
$
where 
$$
D:=\bu_{\om\in\Om}\{\om\}\times D_\om.
$$
It thus follows from ergodicity of the global map $F:\Om\times Q\to \Om\times Q$ with respect to $\mu^{(t)}$ (Theorem~\ref{t1im3B}) and from Birkhoff's Ergodic Theorem that there exists a measurable set $\Ga\sbt\Om\times Q$ such that $\mu^{(t)}(\Ga)=1$ and 
\begin{equation}\label{1m5}
\lim_{n\to\infty}\frac1n\#\big\{0\le j\le n-1:F_\om^j(z)\in D\big\}
=\mu^{(t)}(D)>0
\end{equation}
for every $(\om,z)\in\Ga$. Since $|(F_\om^k)'(z)|=|f_\om^k(z)|\ge |\Im(f_\om^k(z))|$ for each $k\ge 1$, it follows from Lemma~\ref{lem1}, formula \eqref{1m5}, the Chain Rule, and the definition of the set $D$, that
$$
\liminf_{n\to\infty}\frac{1}{n}\log|(F^n_\omega)'(z)|\ge \mu^{(t)}(D)\log 2.
$$
Since, by Birkhoff's Ergodic Theorem again, 
$$
\chi_{\mu^{(t)}}=\lim_{n\to\infty}\frac{1}{n}\log|(F^n_\omega)'(z)|
$$
(in particular the limit exists) for $\mu^{(t)}$--a.e. $(\om,z)\in \Om\times Q$, we thus obtain that $\chi_{\mu^{(t)}}\ge \mu^{(t)}(D)\log 2>0$ and the proof of Proposition~\ref{prop:lyapunov} is complete. 
\end{proof}

\section{Bowen's formula}\label{bowen}

In this section we prove a formula holds that determines the value of the Hausdorff dimension of radial Julia sets. We refer to it as Bowen's formula. Precisely, we prove the following.

\begin{thm}\label{thm:bowen}
For $t>1$  put 
$$
\mathcal E \Pr(t):=\int_\Om\log \lambda_{t,\omega}d m(\omega).
$$
Then 

\begin{enumerate}

\item $\mathcal E \Pr(t)<+\infty$ for all $t>1$,

\smallskip\item The function $(1,+\infty)\ni t\mapsto \mathcal E \Pr(t)$ is strictly decreasing, convex, and thus continuous, 

\smallskip\item $\lim_{t\to 1}\mathcal E \Pr(t)=+\infty$ and $ \mathcal E \Pr(2))\le 0$. 

\smallskip\item Let $h>1$ be the unique value $t>1$ for which
$ \mathcal E \Pr(t)=0$. Then 
$$
\HD(J_{r,\omega})=h
$$ 
for $m$--a.e.$\omega\in\Omega$.
\end{enumerate}
\end{thm}

\noindent The proof of this theorem will be deduced from a series of lemmas.

\begin{lem}
$\mathcal E \Pr(2)\le 0$.
\end{lem}

\begin{proof}
Assume for a contrary that $\mathcal E \Pr(2)> 0$. It then follows from Birkhoff's Ergodic Theorem that
$$
\lim_{n\to\infty}\lambda_{2,\omega}^n=+\infty
$$ 
for $m$--a.e. $\omega\in\Omega$, and in fact, the divergence is exponentially fast. Then 
using Definition~\ref{def:J_r} (of the set  $J_r(\omega)$), conformality of the measure $\nu^{(2)}$  produced in Theorem~\ref{t2im2}, and Koebe's Distortion Theorem, we can write for $m$--almost every $\omega\in\Omega$  and  for  $\nu^{(2)}_\omega$--
almost every $z\in J_r(\omega)$, every integer $N\ge 1$ and all $n\in N_\omega(z,N)$, that 
$$
\nu^{(2)}_\omega(F_{\om,z}^{-n}(B(F_\om^n(z),1/N)))\le C(N)\frac{1}{\lambda^n_{\omega,2}}\diam^2( F_{\om,z}^{-n}(B(F_\om^n(z),1/N))),
$$
with some constant $C(N)\in(0,+\infty)$ depending only on $N$.
Using Koebe's Distortion Theorem again, we thus conclude that 
\begin{equation}\label{eq:dim_small}
\liminf _{r\to 0} \frac{\nu^{(2)}_\omega(B(z,r))}{r^2}=0.
\end{equation} 
But since ${\rm Leb}(B(z,r))=\pi r^2$ for all $r>0$ small enough independently of $z$, where ${\rm Leb}$ denotes the $2$--dimensional Lebesgue measure on $Q$), formula \eqref{eq:dim_small} implies (standard in geometric measure theory, see  e.g., Lemma 2.13 in \cite{mattila} or \cite{PUbook})  that $\nu^{(2)}_\omega(J_r(\omega))=0$. This contradicts Corollary~\ref{c15_2017_06_19}  and finishes the proof.
\end{proof}

\begin{lem}
For every $t>1$ the expected pressure $\mathcal E\Pr(t)$ is finite and the function 
$$
(1,+\infty)\ni t\mapsto \mathcal E\Pr(t)\in\R
$$ 
is convex, thus continuous.
\end{lem}

\begin{proof}
First note that finiteness of the expected pressure follows immediately from  the bounds on $\lambda_{t,\omega}$ provided in \eqref{bounds_on_lambda}. The constants $p,P$ in this estimate depend on $t$ but they are independent of $\omega$.

Obviously, it is enough to prove convexity for every bounded interval $(T_1,T_2)\subset (1,\infty)$. So, from now on let us fix some $1<T_1<T_2$. For every $\om\in\Om$ denote
$$
E_\om:=Q_{M_1}\setminus \bigcup_{j=0}^N B\(F^j_{\th^{n-j}\omega}(0),r_0\).
$$ 
It follows from \eqref{2017_05_12} and  \eqref{bound_for_L} that
\begin{equation}\label{1_2017_05_12}
\frac{1}{\lambda_{t,\omega}^n}\mathcal L^n_{t,\omega}(z)\asymp 1,
\end{equation}
independently of $\om\in\Om$ and $z\in E_\om$, and furthermore, it is easy to see that the comparability constant can be taken the same for all $t\in [T_1,T_2]$. Since, by Birkhoff's Ergodic Theorem the limit $\lim_{n\to\infty}\frac{1}{n}\log \lambda_{t,\omega}^n$ exists for all $t\in[T_1,.T_2]$ and $m$--a.e. $\om\in\Om$, say $\om\in\Om_t$ with $m(\Om_t)=1$, and is equal to $\mathcal E\Pr(t)$, we conclude from \eqref{1_2017_05_12} that also the limit $\lim_{n\to\infty}\frac{1}{n}\log \mathcal L_{t,\omega}^n(z)$ exists for all $t\in[T_1,.T_2]$, all $\om\in\Om_t$, and every $z\in E_\om$, and
$$
\mathcal E\Pr(t)
=\lim_{n\to\infty}\frac{1}{n}\log \lambda_{t,\omega}^n
=\lim_{n\to\infty}\frac{1}{n}\log \mathcal L_{t,\omega}^n(z).
$$
Now, fix $\alpha\in[0,1]$ and $s,t\in [T_1,T_2]$. Fix then $\om\in\Om_s\cap\Om_t\cap \Om_{\alpha s+(1-\alpha)t}$ and $z\in E_\om$. A direct use of H\"older inequality shows that
$$
\begin{aligned}
\mathcal E\Pr(\alpha s+(1-\alpha)t)
&=\lim_{n\to\infty}\frac{1}{n}\log \mathcal L_{\alpha s+(1-\alpha)t,\omega}^n(z)\\
&\le \lim_{n\to\infty}\frac{1}{n}\Big(\alpha\log \mathcal L_{s,\omega}^n(z)
+(1-\alpha)\log\mathcal L_{t,\omega}^n(z)\Big)\\
&= \alpha\lim_{n\to\infty}\frac{1}{n}\log \mathcal L_{s,\omega}^n(z)
+(1-\alpha)\lim_{n\to\infty}\frac{1}{n}\log \mathcal L_{t,\omega}^n(z)\\
&=\alpha\mathcal E\Pr(s)+(1-\alpha)\mathcal E\Pr(t).
\end{aligned}
$$
The proof is complete.
\end{proof}

\begin{lem}\label{prop:decreasing}
The function $(1,\infty)\ni t\longmapsto \mathcal E \Pr(t)\in\R$ is strictly decreasing.
\end{lem}

\begin{proof}
Seeking contradiction suppose that 
\begin{equation}\label{1_2017_05_27}
\mathcal E \Pr(t)\ge \mathcal E \Pr(s)
\end{equation}
for some $1<s<t$. It follows from Corollary~\ref{c15_2017_06_19}, Theorem~\ref{t1im3B}, Proposition~\ref{prop:lyapunov}, and Birkhoff's Ergodic Theorem, that there exist $\chi>0$, an integer $q_0\ge 1$, a measurable set $\Omega_0\sbt \Omega$ with $m(\Omega_0)>1/2$, and for each $\om\in \Omega_0$, a measurable set $J_r^0(\om)\sbt J_r(\om)$ such that 
$$
\nu^{(t)}_\omega(J_r^0(\omega))\ge 1/2,
$$
and
$$
|(F^n_\omega)'(z)|\ge e^{\chi n}
$$
for every $\omega\in \Omega_0$, every $z\in J_r^0(\omega)$, and every integer $n\ge q_0$.
It furthermore follows from Birkhoff's Ergodic Theorem and \eqref{1_2017_05_27} there are an integer $q_1\ge q_0$, a measurable set $\Omega_1\sbt \Omega_0$ with $m(\Omega_1)>1/4$, and 
$$
\frac{\lambda^{-n}_{t,\omega}}{\lambda^{-n}_{s,\omega}}\le  e^{\frac12\chi(t-s) n}.
$$
for every $\omega\in \Omega_1$ and every integer $n\ge q_1$. Fix such $\omega\in\Om_1$ and $z\in J_r^0(\omega)$. By the definition of $J_r(\om)$ there exists an integer $N\ge 1$ such that $\un\rho(N_\om(z,N))> 3/4$.
So, there exist an integer $N\ge 1$ depending on $\om$ and $z$ and an unbounded increasing sequence $(n_k)_{k=1}^\infty$ of integers $\ge q_1$ with lower density $\ge 3/4$ such that for every $k\ge 1$ there exists a holomorphic branch $F^{-n_k}_{\omega,z}:B(F^{n_k}_\omega(z),2/N)\to Q$ of $F^{-n_k}_{\omega,z}$ that maps $F^{n_k}_\omega(z)$ back to $z$ and 
$$
|F^{n_k}_\omega(z)|\le N. 
$$
By Birkhoff's Ergodic Theorem and Proposition ~\ref{prop:supp} there exist a measurable set $\Om_2\sbt \Om_1$ such that $m(\Om_2)>1/8$ and 
$$
\th^n\om\in \Om\(N,(4KN)^{-1},1/4\)
$$
for all $\om\in\Om_2$ and a set of integers $n\ge 0$ of lower density $\ge 3/4$, where the set $\Om\(N,(4KN)^{-1},1/4\)$ comes from Proposition ~\ref{prop:supp}. Passing to a subsequence we may therefore assume that
$$
\th^{n_k}\om\in \Om\(N,(4KN)^{-1},1/4\)
$$
for all $\om\in\Om_2$ and every integer $k\ge 1$. 

Using all the above, Koebe's Distortion Theorems, conformality of the measures $\nu_\omega^{(t)}$ and $\nu_\omega^{(s)}$, and at the end Proposition ~\ref{prop:supp} (the constant $\xi=\xi\(N,(4KN)^{-1},1/4\)>0$ below comes from it), we obtain
$$
\begin{aligned}
&\frac{\nu_\omega^{(t)}\(B(z,(4N)^{-1}|(F^{n_k}_\omega(z)'|^{-1})\)}
{\nu_\omega^{(s)}\(B(z,(4N)^{-1}|(F^{n_k}_\omega(z)'|^{-1})\)} 
\le \frac{\nu_\omega^{(t)}\(F^{-n_k}_{\omega,z}\(B(F^{n_k}_\omega(z),N^{-1}))\)\)}{\nu_\omega^{(s)}\(F^{-n_k}_{\omega,z}\(B(F^{n_k}_\omega(z,(4KN)^{-1})\)\)}\le \\
& \  \  \  \  \  \  \  \  \  \  \le K^{t-s}|(F^{n_k}_\omega)'(z)|^{s-t} 
\frac{\nu_{\theta^{n_k}\omega}^{(t)}\(B(F^n_\omega(z),N^{-1}))\)}{\nu_{\theta^{n_k}\omega}^{(s)}\(B(F^{n_k}_\omega(z,(4KN)^{-1})\)}
  \frac{\lambda^{-n_k}_{t,\omega}}{\lambda^{-n_k}_{s,\omega}} \\
&  \  \  \  \  \  \  \  \  \le K^{t-s}\exp(\chi(s-t)n_k)\(\nu_{\theta^{n_k}\omega}^{(s)}\(B(F^{n_k}_\omega(z),(4KN)^{-1})\)\)^{-1}
     \exp\left(\frac12\chi(t-s)n_k\right) \\
& \  \  \  \  \  \  \  \ =   K^{t-s}\exp\left(\frac12\chi(s-t)n_k\right)\(\nu_{\theta^{n_k}\omega}^{(s)}\(B(F^{n_k}_\omega(z),(4KN)^{-1})\)\)^{-1}\\
&  \  \  \  \  \  \  \  \  \le \xi^{-1}K^{t-s}\exp\left(\frac12\chi(s-t)n_k\right).
\end{aligned}
$$
Therefore
$$
\begin{aligned}
\varliminf_{r\to 0}\frac{\nu_{\omega}^{(t)}(B(z,r))}{\nu_{\omega}^{(s)}(B(z,r))}
&\le\varliminf_{k\to\infty}
\frac{\nu_\omega^{(t)}\(B(z),(4N)^{-1}|(F^{n_k}_\omega(z)'|^{-1})\)}
{\nu_\omega^{(s)}\(B(z),(4N)^{-1}|(F^{n_k}_\omega(z)'|^{-1})\)}\\
&\le \xi^{-1}K^{t-s}\varliminf_{k\to\infty}\exp\left(\frac12\chi(s-t)n_k\right)\\
&=0.
\end{aligned}
$$
This, in a standard way, implies that 
$$
\nu_{\omega}^{(t)}\lt(\bu_{\om\in\Om_1}\{\om\}\times J_r^0(\om)\rt)=0.
$$
But, on the other hand, from the very definition of the sets $\Om_2$ and $J_r^0(\om)$, we have that
$$
\nu_{\omega}^{(t)}\lt(\bu_{\om\in\Om_1}\{\om\}\times J_r^0(\om)\rt)\ge 1/8>0.
$$
This contradiction ends the proof of Lemma~\ref{prop:decreasing}.
\end{proof}

\noindent Before we prove a next part of Theorem~\ref{thm:bowen}, 
note that the following elementary estimate holds:
\begin{equation}\label{eq:calculus}
\sum_{k=0}^\infty\frac{1}{(a+k)^t}\ge \frac{(a)^{1-t}}{t-1}.
\end{equation}



\begin{lem}\label{prop:pressure_at_1}
$\lim_{t\to 1^+}\mathcal E \Pr(t)=+\infty.$
\end{lem}

\begin{proof}
In order to prove this lemma, we have to examine in more detail the lower bound for $ \L^*_{t,
\omega}\nu_{\theta\omega}(\1)$ obtained in Proposition ~\ref{measurelower}, which, in turn, follows from 
Lemma~\ref{lower}. We need to use this estimate for the random conformal measure $\nu_{\omega}^{(t)}$, as  
$$
\lambda_{t,\omega}=\L^*_{t,\omega}\nu_{\theta\omega}^{(t)}(\1).
$$
So, recall the formula \eqref{1.3} from Lemma~\ref{measurelower}:
$$
\L^*_{t,\omega}\nu_{\theta\omega}(\1)\ge \frac12dM_0^{1-t}
$$
where $M_0>0$ is a constant independent of $t$, and  $d>0$, which does depend on $t$, comes from the first estimate in Lemma~\ref{lower}:
\begin{equation}\label{1_2017_05_30}
\L_{t,\omega}(\1)(z)\ge\frac{d}{|z|^{t-1}}.
\end{equation}
Thus, in order to conclude the proof of lemma~\ref{prop:pressure_at_1} it is enough to establish the following.

\medskip 
{\bf Claim~$1^0$}: The constant $d=d(t)>0$ of \eqref{1_2017_05_30} can be chosen so that 
$$
\lim_{t\to 1^+}d(t)=+\infty.
$$
\begin{proof}
Take an arbitrary point $z\in Q$, $z\neq 0$, and its representative in $x+iy\in\C$ with $y\in (0,\pi]$. Then
$$
\mathcal L_{t,\omega}(\1)(z)=\sum_{k\in\Z}\frac{1}{(x^2+(y+2k\pi)^2)^ {t/2}}\ge\sum_{k=0}^\infty\frac{1}{(|x|+y+2k\pi)^t}.
$$
Putting $a:=|x|+y$ and using  \eqref{eq:calculus}, we easily complete the proof of Claim~$1^0$.
\end{proof}

\noindent Lemma~\ref{prop:pressure_at_1} is thus also proved.
\end{proof}

\noindent As an immediate consequence of the above lemmas, we get the following.
\begin{cor}
There exists a unique value $h\in [1,2)$ such that $\mathcal E\Pr(h)=0$.
\end{cor}

Now, in order to conclude the proof of Theorem~\ref{thm:bowen}, we are only left to establish its item (4). Towards this end, we shall prove the following auxiliary result.

\begin{lem}\label{lem:wykl}
For every $\om\in\Om$ and every integer $N\ge 2$ there exists $q(\omega,N)$ such that if $z\in Q_N$ and if $q(\omega,N)\le n\in  N_\omega(z,N)$, then 
$$
|(F_\om^n)'(z)|\ge 2.
$$
\end{lem}
\begin{proof} 
Fix an integer $N\ge 2$. Notice that then there exists an integer $q_N\ge 0$ such that
\begin{equation}\label{7_2017_07_10}
f_\om^n(\R)\sbt [4N,+\infty)
\end{equation}
for all $\om\in\Om$ and all $n\ge q_N$. Now fix also $\om\in\Om$. 
Assume for a contrary that there exist a strictly increasing sequence $(n_l)_{l=1}^\infty$ of integers, all greater than or equal to $q_N$, and a sequence $(z_l)_{l=1}^\infty$ of points in $Q_N$ such that 
$$
n_l\in N_\omega(z_l,N) \  \  {\rm and} \  \  |(F_\om^{n_l})'(z_l)|\le 2
$$
for every $l\ge 1$. Using compactness of $Q_N$ we can replace the sequence $(n_l)_{l=1}^\infty$ by its increasing subsequence for which there exist a point $\xi\in Q_N$ such that 
$$
z_l\in B(\xi,1/(16Nl)).
$$
It then follows from Koebe's $\frac14$--Distortion Theorem that
\begin{equation}\label{8_2017_07_10}
B(\xi,1/(16N))\sbt F_{\om,z_l}^{-n_l}\(B(F_\om^{n_l}(z_l), 1/N)\)
\end{equation}
now seeking contradiction assume that
$$
f_\om^k\(B(\xi,2^{-4}N^{-1})\)\cap \lt(\bu_{j\in\Z}\R+j\pi i\rt)\ne\es
$$
for some integer $k\ge 0$. Then for for every integer $l\ge 1$,
$$
f_\om^k\(F_{\om,z_l}^{-n_l}\(B(F_\om^{n_l}(z_l), 1/N)\)\)
\cap \lt(\bu_{j\in\Z}\R+j\pi i\rt)\ne\es.
$$
Fix $l\ge 1$ so large that $n_l-k\ge q_N+1$. Then invoking \eqref{7_2017_07_10}, we conclude that
$$
B(F_\om^{n_l}(z_l), 1/N)\cap [4N,+\infty)\ne\es.
$$
Hence, $F_\om^{n_l}(z_l)\notin Q_N$ contrary to the fact that $n_l\in N_\om(z_l,N)$. So,
\begin{equation}\label{1m1}
f_\om^k\(F_{\om,z_l}^{-n_l}\(B(F_\om^{n_l}(z_l), 1/N)\)\)
\cap \lt(\bu_{j\in\Z}\R+j\pi i\rt)=\es
\end{equation}
for every integer $k\ge 0$. Now, as at the beginning of the paper, keep $S$ to denote the set
$$
\{z\in\C:|\Im(z)|<\pi\}.
$$
If the set 
$$
A_\om(N):=\big\{k\ge 0:f_\om^k\(B(\xi,2^{-4}N^{-1})\)\cap S=\es\big\}
$$
is infinite, then $\varlimsup_{j\to \infty}|(f_\om^j)'(\xi)|=+\infty$ by lemma~\ref{lem1} and the Chain Rule. This and \eqref{1m1} would however 
contradict Bloch's Theorem, proving that the set $A_\om(N)$ is finite. But if 
$f_\om^k\(B(\xi,2^{-4}N^{-1})\)\cap S\ne\es$, then 
\begin{equation}\label{1m2}
f_\om^k\(B(\xi,2^{-4}N^{-1})\)\sbt S
\end{equation}
by \eqref{1m1} again. Therefore, \eqref{1m2} holds for all but finitely many $k$s. This however contradicts Lemma~\ref{prop1}, finishing the proof of Lemma~\ref{lem:wykl}.
\end{proof}

Now, within the framework of Lemma~\ref{lem:wykl}, let $q_0(\om,N)$ denote the least number $q(\om,N)$ produced by this lemma. We immediately observe the following.

\begin{obs}\label{o1m2}
For every integer $N\ge 2$ the function
$$
\Om\ni\om\longmapsto q_0(\om,N)\in\N
$$
is measurable.
\end{obs}

The last step, the one finishing the proof of Theorem~\ref{thm:bowen} is this.

\begin{lem}\label{l1_2017_05_30}
$$
\HD(J_r(\omega))=h
$$
for $m$--a.e. $\omega\in\Omega$.
\end{lem}

\begin{proof} 
The beginning of this proof is similar to the proof of Lemma~\ref{prop:decreasing}.
It follows from Corollary~\ref{c15_2017_06_19}, Theorem~\ref{t1im3B}, Proposition~\ref{prop:lyapunov}, and  Birkhoff's Ergodic Theorem, that there exist $\chi>0$, a measurable set $\Omega_0\sbt \Omega$ with $m(\Omega_0)=1$, and for each $\om\in \Omega_0$, a measurable set $J_r^0(\om)\sbt J_r(\om)$ such that 
$$
\nu^{(h)}_\omega(J_r^0(\omega))=1
$$
and 
\begin{equation}\label{MUZ_9.13}
\lim_{n\to\infty}\frac1n\log|(F^n_\omega)'(z)|=\chi
\end{equation}
for every $\omega\in \Omega_0$ and every $z\in J_r^0(\omega)$. It furthermore follows from Birkhoff's Ergodic Theorem that there exists a measurable set $\Omega_1\sbt \Omega_0$ with $m(\Omega_1)=1$, and such that
\begin{equation}\label{2_2017_06_19}
\lim_{n\to\infty}\frac1n\log\lambda^n_{h,\omega}=0
\end{equation}
for every $\omega\in \Omega_1$. Fix such $\omega\in\Om_1$ and $z\in J_r^0(\omega)$. Fix $\eta\in(0,1/2)$ arbitrary. By the definition of $J_r(\om)$ there exists an integer $N_\eta\ge 1$ such that 
\begin{equation}\label{5_2017_06_14}
\un\rho(N_\om(z,N_\eta))> 1-\eta.
\end{equation}
For every $r\in(0,1/N_\eta)$ let $k:=k(z,r)$ be the largest integer $n\in N_\om(z,N_\eta)$ such that
\begin{equation}\label{MUZ_9.8}
F_{\omega,z}^{-n}\(B(F_\omega^n(z),1/N_\eta)\)\supset B(z,r).
\end{equation}
Let $s=s_k$ be the largest integer $\ge k+1$ belonging to $N_\om(z,N_\eta)$. It follows from \eqref{5_2017_06_14} that
\begin{equation}\label{MUZ_9.9}
\lim_{r\to 0}\frac{k(z,r)}{s_{k(z,r)}}\ge 1-\eta.
\end{equation}
Applying conformality of the measure $\nu^{(h)}$ and Koebe's Distortion Theorem, we now conclude from \eqref{MUZ_9.8} and the definition of $k$ that
\begin{equation}\label{11_2017_06_14}
\aligned
\nu_{\omega}^{(h)}\(B(z,r)\)
&\le \nu_\omega^{(h)}\(F_{\omega,z}^{-k}\(B(F_\omega^k(z),1/N_\eta)\)\)\\
&\le K^h\lambda^{-k}_{\omega,h}\big|\(F_\omega^k\)'(z)\big|^{-h}   \nu_{\theta^k\omega}\(B(F_\om^k(z),1/N_\eta)\)\\
&\le K^h\lambda^{-k}_{\omega,h}\big|\(F_\omega^k\)'(z)\big|^{-h}.
\endaligned
\end{equation}
On the other hand $B(z,r)\not\subset F_{\omega,z}^{-s}\(B(F_\omega^s(z),1/N_\eta)\)$. But since, by $\frac14$-Koebe's Distortion Theorem,
$$
F_{\omega,z}^{-s}\(B(F_\omega^s(z),1/N_\eta)\)
\supset B\lt(z,\frac14\big|\(F_\omega^s\)'(z)\big|^{-1}N_\eta^{-1}\rt),
$$
we thus get that $r\ge \frac14\big|\(F_\omega^s\)'(z)\big|^{-1}N_\eta^{-1}$. Equivalently,
$$
\big|\(F_\omega^k\)'(z)\big|^{-1}\le 4N_\eta r.
$$
By inserting this into \eqref{11_2017_06_14} and using also the Chain Rule, we obtain 
$$
\nu_{\omega}^{(h)}\(B(z,r)\)
\le (4KN_\eta)^hr^h\lambda^{-k}_{\omega,h}\big|\(F_{\th^k\om}^{s-k}\)'(F_\omega^k(z))\big|^h.
$$
Equivalently,
\begin{equation}\label{1abf2}
\begin{aligned}
\frac{\log\nu_{\omega}^{(h)}(B(z,r))}{\log r}
&\ge h+\frac{\log(4KN_\eta)}{\log r} - \frac{\log\lambda^k_{\omega,h}}{\log r}  +h\frac{\log|\(f_{\th^k\om}^{s-k}\)'(F_\omega^k(z))\big|}{\log r} \\
&=h-\frac{k}{\log(1/r)}\frac{\log(4KN_\eta)}{k} 
+\frac{k}{\log(1/r)}\frac{1}{k}\log\lambda^k_{\omega,h}- \\
&\  \  \  \  \  \  \  \  \  \  \  \ -h\frac{k}{\log(1/r)}\frac{1}{k}\log|\(f_{\th^k\om}^{s-k}\)'(F_\omega^k(z))\big|.
\end{aligned}
\end{equation}
Now, Koebe's Distortion Theorem yields
$$
F_{\omega,z}^{-k}\(B(F_\omega^k(z),1/N_\eta)\)
\sbt B\(z,K\big|\(F_\omega^k\)'(z)\big|^{-1}N_\eta^{-1}\).
$$
Along with \eqref{MUZ_9.8} and the definition of $k$ this gives $r\le K\big|\(F_\omega^k\)'(z)\big|^{-1}N_\eta^{-1}$. Equivalently:
\begin{equation}\label{MUZ_9.12}
-\log r\ge \log(N_\eta/K)+\log\big|\(F_\omega^k\)'(z)\big|.
\end{equation}
Therefore, invoking \eqref{MUZ_9.13}, we get that
\begin{equation}\label{1_2017_06_19}
\limsup_{r\to 0}\frac{k(z,r)}{\log(1/r)}\le 1/\chi.
\end{equation}
Also, formula \eqref{MUZ_9.13} along with \eqref{MUZ_9.9} gives
\begin{equation}\label{MUZ_9.13}
\lim_{r\to 0}\frac1k\log|\(F_{\th^k\om}^{s-k}\)'(F_\omega^k(z))\big|
\le \frac{\chi}{1-\eta}-\chi
=\frac{\eta}{1-\eta}\chi.
\end{equation}
Inserting now \eqref{1_2017_06_19}, \eqref{2_2017_06_19}, and \eqref{MUZ_9.13} to \eqref{1abf2}, we obtain
$$
\liminf_{r\to 0}\frac{\log\nu_{\omega}^{(h)}(B(z,r))}{\log r}
\ge h\lt(1-\frac{\eta}{1-\eta}\rt)
=\frac{1-2\eta}{1-\eta}h.
$$
Since $\eta\in(0,1/2)$ was arbitrary, this yields
\begin{equation}\label{3_2017_06_19}
\liminf_{r\to 0}\frac{\log\nu_{\omega}^{(h)}(B(z,r))}{\log r}
\ge h,
\end{equation}
Therefore 
\begin{equation}\label{11_2017_06_30}
\HD(J_r(\om))\ge \HD(J_r^0(\om))\ge h,
\end{equation}
and one side of the equation from Lemma~\ref{l1_2017_05_30} is thus established.

For the opposite inequality set $\eta:=1/4$ and 
$$
N:=N_{1/4}.
$$
By Lemma~\ref{lem:wykl}, Observation~\ref{o1m2}, and Proposition~\ref{prop:supp}, there exists an integer $q\ge 1$ such that
\begin{equation}\label{1m2}
m\(\big\{\om\in\Om:q_0(\om,N)\le q\big\}\cap \Om(N,1/N,1/8\)>5/8.
\end{equation} 
It therefore follows from Birkhoff's Ergodic Theorem that there exists a measurable set $\hat\Om\sbt\Om$ such that $m(\hat\Om)=1$ and 
$$
\rho\(\big\{n\ge 0:q_0(\th^n\om,N)\le q \ \  {\rm and }  \   \  \th^n\om\in  \Om(N,1/N,1/8)\big\}\)>5/8
$$
for all $\om\in\hat\Om$. Then
\begin{equation}\label{1m3}
\rho\(N_\om(z,N)\cap\big\{n\ge 0:q_0(\th^n\om,N)\le q\ \  {\rm and }  \   \  \th^n\om\in  \Om(N,1/N,1/8)\big\}\)>3/8.
\end{equation}
Fix now an arbitrary element $\om\in\hat\Om$ and $z\in J_r(\om)$. S
There thus exists an integer $l_0\ge 1$ so large that if $l$ is an integer $\ge l_0$ and if $u_l$ is the $l$th element of the set 
$$
N_\om(z,N)\cap\big\{n\ge 0:q_0(\th^n\om,N)\le q \ \  {\rm and }  \   \  \th^n\om\in  \Om(N,1/N,1/8\big\}, 
$$
then
$$
\frac{l}{u_l}\ge 3/8.
$$
So, applying lemma~\ref{lem:wykl}, we thus get that
\begin{equation}\label{5_2017_07_19}
|(F_\om^{u_l})'(z)|\ge 2^{l/q}\ge 2^{3u_l/8q}
\end{equation}
Let $r_l>0$ be the least radius such that
\begin{equation}\label{2abf4}
F_\om^{-u_l}\(B(F_\om^{u_l}(z),1/N)\) \sbt B(z,r_l).
\end{equation}
But, by Koebe's Distortion Theorem, $F_\om^{-u_l}\(B(F_\om^{u_l}(z),1/N)\) \sbt B\(z,KN^{-1}|(F_\om^{u_l})'(z)|^{-1}\)$; hence
\begin{equation}\label{5abf4}
r_l\le KN^{-1}|(F_\om^{u_l})'(z)|^{-1}.
\end{equation}
Formula \eqref{2abf4} along with Koebe's Distortion Theorem and \eqref{5abf4}, and Proposition~\ref{prop:supp} (the constant $\xi=\xi(N,1/N,1/8)>0$ below comes from it), yield
\begin{equation}\label{4abf4}
\begin{aligned}
\nu_{\omega}^{(h)}\(B(z,r_l)\)
&\ge \nu_{\omega}^{(h)}\(F_\om^{-u_l}\(B(F_\om^{u_l}(z),1/N)\)\)\\
&\ge K^{-h}\lambda^{-u_l}_{\omega,h}\big|\(F_\omega^{u_l}\)'(z)\big|^{-h}    
 \nu_{\theta^{u_l}\omega}\(B(F_\om^{u_l}(z),1/N)\)\\
&\ge K^{-h}\xi\lambda^{-u_l}_{\omega,h}\big|\(F_\omega^{u_l}\)'(z)\big|^{-h}\\
&\ge (K^{-2}N)^h\xi\lambda^{-u_l}_{\omega,h}r_l^h.
\end{aligned}
\end{equation}
Therefore, 
\begin{equation}\label{20_2017_06_19}
\begin{aligned}
\frac{\log\nu_{\omega}^{(h)}(B(z,r_l))}{\log r_l}
&\le h+\frac{h\log(N/K^2)}{\log r_l}-\frac{\log\lambda^{u_l}_{\omega,h}}{\log r_l} +\frac{\xi}{\log r_l}.
\end{aligned}
\end{equation}
Formula \eqref{5abf4} equivalently means that
\begin{equation}\label{10_2017_06_30}
-\log r_l\ge \log|(F_\om^{u_l})'(z)|+\log(N/K).
\end{equation}
Hence, invoking  \eqref{5_2017_07_19}, we get that
\begin{equation}\label{10_2017_06_30}
-\log r_l\ge \frac{3\log 2}{8q}u_l+\log(N/K).
\end{equation}
Inserting this to \eqref{20_2017_06_19} and using \eqref{2_2017_06_19}, we get 
$$
\liminf_{r\to 0}\frac{\log\nu_{\omega}^{(h)}(B(z,r))}{\log r}
\le \liminf_{l\to\infty} \frac{\log\nu_{\omega}^{(h)}(B(z,r_l))}{\log r_l}
\le h.
$$
Therefore,
$$
\HD(J_r(\om))\le h,
$$
and long with \eqref{11_2017_06_30} this finishes the proof of Lemma~\ref{l1_2017_05_30}.
\end{proof}

\

\section{Hausdorff Dimension of the radial Julia set is smaller than $2$}
\begin{lem}\label{sums_bounded}
Let $(Y,\frak F, \mu)$ be a probability space and let $T:Y\to Y$ be a measure preserving ergodic transformation. Assume that $\varphi:Y\to \mathbb R$ is an integrable function with $\int\varphi d\mu=0$. Assume further that there exist a set $A\in\frak F$ with $\mu(A)>0$ and a constant $C>0$ such that for all $y\in A$ 
$$
\sup_{k\ge 1}\{S_k\varphi(y)\}<C.
$$
Then for $\mu$-a.e. $y\in A$  the following implication  holds:
\begin{equation}\label{low_bound}
T^n(y)\in A\,\Longrightarrow\, S_n\varphi(y)>-2C.
\end{equation}
\end{lem}

\begin{proof}
Assume that \eqref{low_bound} does not hold.
Then there exists a measurable subset $B\subset A$ with $\mu(B)>0$ and such that for every $y\in B$ there exists an integer $n\ge 1$ for which
\begin{equation}\label{set_B}
T^n(y)\in A \quad \text{and}\quad  S_n\varphi(y)\le -2C.
\end{equation}
Replacing $B$ by its subset, still of positive measure, we can assume that there exists an integer $k\ge 1$ such that \eqref{set_B} holds for integers $n$ being the $k$th returns of $y$ to $A$.
Now, let us consider the map $\hat T_B^{(k)}:B\to B$ being the $k$th return from $B$ to $B$. For $\mu$--almost every $x\in B$ denote by $n_B(x)$ the first return time of $x$ to $B$ and by $n_B^{(k)}(x)$ the $k$th  return time of $x$ to $B$.

Kac's lemma applied for the $k$th return map $\hat T^{(k)}$ and for the function $\varphi$ thus gives 
\begin{equation}\label{sum_zero}
\int_BS_{n_B^{(k)}(x)}\varphi(x)d\mu(x)
=k\int_BS_{n_B(x)}\varphi(x)d\mu(x)
=k\int_X\varphi(x)d\mu(x)=0.
\end{equation}
Still for $\mu$--almost every $x\in B$ denote by $n_A^{(k)}(x)$ the $k$-th entrance time of $x$ to $A$ and notice an obvious inequality $n_B^{(k)}(x)\ge n_A^{(k)}(x)$.
Writing 
$$
S_{n_B^{(k)}(x)}\varphi(x)
=S_{n_A^{(k)}(x)}\varphi(x)+S_{n_B^{(k)}(x)-n_A^{(k)}(x)}\varphi\(T^{n_A^{(k)}(x)}(x)\),
$$
we see that 
$$
S_{n_B^{(k)}(x)}\varphi(x)< -2C+C=-C
$$
for $\mu$--almost all $x\in B$. But this  contradicts \eqref{sum_zero} and finishes the proof of our lemma.
\end{proof}

\

In Section~\ref{bowen} we proved that the dimension of the radial random Julia set $J_r(\omega)$  is almost surely equal to the only value $h$ such that the expected pressure at $h$, i.e.
$$
\mathcal E \Pr(h)=\int\log \lambda_{h,\omega}d m(\omega)=0.
$$
As in Section ~\ref{inv_meas}, we denote
$$
\lambda_{h,\omega}^n:=\lambda_{h,\omega}\cdot\lambda_{h,\theta\omega}\cdot \lambda_{h,\theta^{n-1}\omega}.$$

Our goal now is to prove that $h<2$. The crucial technical ingredient is the following.

\begin{prop}\label{prop:low_lim}
For $m$--a.e. $\omega\in \Omega$ and for $\nu_{\omega,h}$--almost every point $z\in Q$ we have that
\begin{equation}
\liminf_{r\to 0}\frac{\nu_{\omega,h}(B(z,r))}{r^h}=0.
\end{equation}
\end{prop}
\begin{proof}
We consider two separate cases.

{\bf Case$1^0$.} Partial sums 
$$
\log\lambda_{h,\om}^n=\sum_{j=0}^{n-1}\log \lambda_{h,\th^j\omega}
$$ 
are bounded above for a measurable set of points $\omega\in\Omega$ with positive measure $m$. 
This means that there exist a measurable set $A\subset \Omega$ with $m(A)>0$ and a constant $C<+\infty$ such that 
\begin{equation}\label{1_2017_08_11}
\log\lambda_{h,\omega}^n<C
\end{equation}
for all $\om\in A$. By ergodicity of the map $F:\Om\times Q\to \Om\times Q$ with respect to the measure $\mu_h$ (see Theorem~\ref{t1im3B}) and by Birkhoff's Ergodic Theorem there exists a measurable set $\Ga_1\sbt\Om\times Q$ with $\mu_h(\Ga_1)=\nu_h(\Ga_1)=1$ and such that for every point $(\om,z)\in\Ga_1$ there exists an integer $k_1(\om,z)\ge 0$ such that
\begin{equation}\label{1m6}
F^{k_1(\om,z)}(\om,z)\in A\times Q.
\end{equation}
Fix an integer $N\ge 1$ and consider the set 
$$
A\times \(Y_{N+2}^+\sms (\R\times (-2,2))\).
$$
Since $\mu_h\(A\times \(Y_{N+2}^+\sms (\R\times (-2,2))\)\)>0$, again by ergodicity of the map $F:\Om\times Q\to \Om\times Q$ with respect to the measure $\mu_h$ (see Theorem~\ref{t1im3B}) and by Birkhoff's Ergodic Theorem there exists a measurable set $\Ga_2\sbt A\times Q$ with $\mu_h(\Ga_2)=\mu_h(A\times Q)$ and such that for every point $(\tau,\xi)\in\Ga_2$ there exists an integer $k_2(\tau,\xi;N)\ge 0$ such that
\begin{equation}\label{3m6}
F^{k_2(\tau,\xi;N)}(\tau,\xi)\in A\times \(Y_{N+2}^+\sms (\R\times (-2,2))\).
\end{equation}
In conclusion, there exists a measurable set $\Ga_3(N)\sbt \Ga_1$ such that $\mu_h(\Ga_3(N))=1$ and
\begin{equation}\label{4m6}
F^{k_1(\om,z)}(\om,z)\in \Ga_2
\end{equation}
for all points $(\om,z)\in\Ga_3(N)$. In particular $k_2\(F^{k_1(\om,z)}(\om,z);N\)$
is well defined and finite. For every point $(\om,z)\in\Ga_3(N)$ set
\begin{equation}\label{5m6}
\ell_N(\om,z):=k_1(\om,z)+k_2\(F^{k_1(\om,z)}(\om,z);N\).
\end{equation}
Denote
$$
\Ga_3(\infty):=\bi_{N=1}^\infty\Ga_3(N).
$$
Then 
$$
\mu_h(\Ga_3(\infty))=1
$$
and the number $\ell_N(\om,z)$ is well defined for all points $(\om,z)\in \Ga_3(\infty)$ and all integers $N\ge 1$. Fix such $(\om,z)$ and $N$. Then 
\begin{equation}\label{2m7}
\log\lambda_{h,\om}^{\ell_N(\om,z)}
=\log\lambda_{h,\om}^{k_1(\om,z)}+\log\lambda_{h,\th^{k_1(\om,z)}\om}^{k_2\(F^{k_1(\om,z)}(\om,z);N\)}
>\log\lambda_{h,\om}^{k_1(\om,z)}-2C
\end{equation}
by \eqref{1m6}--\eqref{5m6} and Lemma~\ref{sums_bounded}. 

Now, since $F_\om^{\ell_N(\om,z)}(z)\)\in Y_{N+2}^+\sms (\R\times (-2,2))$, the holomorphic inverse branch $F_\om^{-\ell_N(\om,z)}:B\(F_\om^{\ell_N(\om,z)}(z),2\)\to Q$, sending $F_\om^{\ell_N(\om,z)}(z)$ back to $z$, is well defined,
\begin{equation}\label{1m7}
B\(F_\om^{\ell_N(\om,z)}(z),1\)\sbt Y_N^-
\end{equation}
and, with a use of Koebe's Distortion Theorem,
$$
B\lt(z,\frac14\big|\(F_\om^{\ell_N(\om,z)}\)'(z)\big|^{-1}\rt)
\sbt F_\om^{-\ell_N(\om,z)}\(B\(F_\om^{\ell_N(\om,z)}(z),1\)\).
$$
Set
$$
r_N(\om,z):=\frac14\big|\(F_\om^{\ell_N(\om,z)}\)'(z)\big|^{-1}.
$$
Then, using, \eqref{1m7}, \eqref{2m7}, and  \eqref{3.1}, we get 
$$
\begin{aligned}
\frac{\nu_{\om,h}\(B(z,r_N(\om,z)\)}{r_N^h(\om,z)}
&\le \frac{\nu_{\om,h}\(F_\om^{-\ell_N(\om,z)}\(B\(F_\om^{\ell_N(\om,z)}(z),1\)\)\)
}{r_N^h(\om,z)} \\
&\le \frac{\lambda_{h,\om}^{-\ell_N(\om,z)}K^h \big|\(F_\om^{\ell_N(\om,z)}\)'(z)\big|^{-h}
\nu_{\th^{\ell_N(\om,z)}\om,h}\(B\(F_\om^{\ell_N(\om,z)}(z),1\)\)}{r_N^h(\om,z)} \\
&= \frac{(4K)^h\lambda_{h,\om}^{-\ell_N(\om,z)}r_N^h(\om,z)
\nu_{\th^{\ell_N(\om,z)}\om,h}\(B\(F_\om^{\ell_N(\om,z)}(z),1\)\)}
{r_N^h(\om,z)} \\
&= (4K)^h\lambda_{h,\om}^{-\ell_N(\om,z)}
\nu_{\th^{\ell_N(\om,z)}\om,h}\(B\(F_\om^{\ell_N(\om,z)}(z),1\)\)\\
&\le (4K)^h\lambda_{h,\om}^{-\ell_N(\om,z)}
\nu_{\th^{\ell_N(\om,z)}\om,h}(Y_N^+) \\
&\le (4K)^he^{2C}\lambda_{h,\om}^{-k_1(\om,z)}C(M_0)e^{(1-t))\frac{N}2}.
\end{aligned}
$$
Since $1-t<0$, this yields
$$
\varliminf_{r\to 0}\frac{\nu_{\om,h}(B(z,r)}{r^h}
\le \varliminf_{N\to\infty}\frac{\nu_{\om,h}\(B(z,r_N(\om,z)\)}{r_N^h(\om,z)}
\le (4K)^he^{2C}\lambda_{h,\om}^{-k_1(\om,z)}C(M_0)\lim_{N\to\infty} 
e^{(1-t))\frac{N}2}
=0.
$$

{\bf Case 2.} For $m$--a.e $(\omega,z)\in \Omega\times Q$ $$
\limsup_{n\to\infty}\lambda_\omega^n=+\infty.
$$
Let $(\omega, z)$ be a point for which the above upper limit is equal to $+\infty$. There then exists a strictly increasing  sequence $\(n_j(\om,z)\)_{j=1}^\infty$ of positive integers such that 
\begin{equation}\label{3_2017_08_21}
\lim_{n\to\infty} \lambda_\omega^{n_j(\om,z)}=+\infty.
\end{equation}
Fix a radius $0<s<\min\{1,\rho\}/4$, where $\rho$ comes from Lemma~\ref{lem:est_zero}. Fix an integer $j\ge 1$. If 
$$
B(F_\om^{n_j(\om,z)}(z),2s)\cap \big\{F^{n_j(\om,z)-k}_{\theta^k\omega}(0): k=1,\dots n_j-1\big\}=\es,
$$ 
then there exists a holomorphic 
branch $F^{-n_j(\om,z)}_{\omega,z}$ defined on $B\(F^{n_j(\om,z)}(z),2s\)$ and sending $F^{n_j}_\omega(z)$ back to $z$. Analogously as in the previous case, put
$$
r_j(\om,z):=\frac14\big|\(F_\om^{n_j(\om,z)}\)'(z)\big|^{-1}s.
$$
Then by the same token as in the previous case, we get
\begin{equation}\label{1_2017_08_21}
\frac{\nu_{\om,h}\(B(z,r_j(\om,z))}{r_j^h(\om,z)}
\le(4K)^h\lambda_{h,\om}^{-n_j(\om,z)}
\nu_{\th^{n_j(\om,z)}\om}\(B\(F_\om^{n_j(\om,z)}(z),r\)\)
\le (4K)^h\lambda_{h,\om}^{-n_j(\om,z)}.
\end{equation}

Finally, consider the case when the ball $B\(F_\om^{n_j(\om,z)}(z),2s\)$ contains some point from the set $\big\{F^{n_j-k}_{\theta^k\omega}(0): k=1,\dots n_j(\om,z)-1\big\}$. 
Fix such $k\in\{0,\dots n_j-1\}$ with the smallest distance between $F_\om^{n_j(\om,z)}(z)$ and $F^{n_j(\om,z)-k}_{\theta^k\omega}(0)$ in addition.
Denote this distance by $2\hat s$.  Then $2\hat s<2s<\rho/2$ and   
$$
B\(F_\om^{n_j(\om,z)(\om,z)}(z),\hat s\)
\sbt B\(F^{n_j(\om,z)-k}_{\theta^k\omega}(0), 3\hat s\). 
$$
It thus follows from Lemma~\ref{lem:est_zero} that
$$
\nu_{\theta^{n_j(\om,z)}\omega}\(B\(F_\om^{n_j(\om,z)(\om,z)}(z),\hat s\)\)
\le \nu_{\theta^{n_j(\om,z)}\omega}\(B\(F_{\theta^k\omega}^{n_j(\om,z)-k}(0),3\hat s\)\)\preceq \hat r^u \le r^h.
$$
It also follows from the definition of $\hat r$ that there exists a unique holomorphic 
branch $F^{-n_j(\om,z)}_{\omega,z}$ defined on $B\(F^{n_j(\om,z)}(z),2\hat r\)$ and sending $F^{n_j}_\omega(z)$ back to $z$. Analogously as in the previous case, put
$$
\hat r_j(\om,z):=\frac14\big|\(F_\om^{n_j(\om,z)}\)'(z)\big|^{-1}\hat r.
$$
Then, in the same way as \eqref{1_2017_08_21}, we gat 
\begin{equation}\label{2_2017_08_21}
\frac{\nu_\om\(B(z,\hat r_j(\om,z))}{\hat r_j^h(\om,z)}
\le (4K)^h\lambda_{h,\om}^{-n_j(\om,z)}\hat r^{-h}
\nu_{\th^{n_j(\om,z)}\om}\(B\(F_\om^{n_j(\om,z)}(z),\hat r\)\)
\preceq \lambda_{h,\om}^{-n_j(\om,z)}.
\end{equation}
Along with formula \eqref{3_2017_08_21}, formulas \eqref{1_2017_08_21} and \eqref{2_2017_08_21} respectively imply that $\lim_{j\to\infty} r_j(\om,z)=
\lim_{j\to\infty}\hat r_j(\om,z)=0$ and
$$
\liminf_{r\to 0}\frac{\nu_h(B(z,r)}{r^h}=0.
$$
The proof of Proposition~\ref{prop:low_lim} is complete.
\end{proof}

\begin{thm}\label{t1_2017-09_04}
The Hausdorff dimension $h=\HD(J_{r,\omega})$ of the random radial Julia set $J_{r,\omega}$, is constant for $m$--a.e. $\om\in\Om$ and satisfies $1<h<2$. In particular, the $2$--dimensional Lebesgue measure of $m$--a.e. $\om\in\Om$ set $J_{r,\omega}$ is equal to zero.
\end{thm}

\begin{proof}
The fact that the function $\Om\ni\om\mapsto \HD(J_{r,\omega})$ is constant for $m$--a.e. $\om\in\Om$, and the inequality $h>1$ is just item (4) of  Theorem~\ref{thm:bowen}.  

Because of Proposition~\ref{prop:low_lim}, $h$--dimensional packing measure of $Q$ is locally infinite for $m$--a.e. $\omega\in \Om$. Since $2$--dimensional packing measure is just the (properly rescaled) $2$--dimensional Lebesgue measure, it is locally finite. Thus $h<2$.

\end{proof}
As a corollary, we obtain the following result about trajectories of (Lebesgue) typical points.

\begin{thm}[Trajectory of a (Lebesgue) typical point I]\label{t320180410}
For $m$--almost every $\omega\in\Omega$ there exists a subset $Q_\omega\subset Q$ with full Lebesgue measure such that for all $z\in Q_\omega$  the following holds. 
\begin{equation}\label{bad}
\begin{aligned}
\forall \delta>0 \ \exists n_z(\delta)\in\mathbb N \ &\forall n\ge n_z(\delta)\  \exists k=k_n(z)\ge 0
\\  
&|F^n_\omega(z)-F^k_{\theta^{n-k}\omega}(0)|<\delta\  \  \text{or}\  \  |F^n_\omega(z)|\ge 1/\delta.
\end{aligned}
\end{equation}
In addition, $\limsup_{n\to\infty}k_n(z)=+\infty$.
\end{thm}

\begin{proof}
For every $\om\in\Om$, the set of points with trajectories described \eqref{bad} contains the complement of the radial set Julia set $J_r(\omega)$. So, now the first assertion follows immediately from the last assertion of Theorem~\ref{t1_2017-09_04}. The second assertion is obvious.
\end{proof}

\noindent As an immediate consequence of this theorem we get the following.

\begin{cor}[Trajectory of a (Lebesgue) typical point II]\label{c420180410}
For $m$--almost every $\omega\in\Omega$ there exists a subset $Q_\omega\subset Q$ with full Lebesgue measure such that for all $z\in Q_\omega$, the set of accumulation points of the sequence
$$
\(F_\om^n(z)\)_{n=0}^\infty
$$
is contained in $[0,+\infty]\cup\{-\infty\}$ and contains $+\infty$. 
\end{cor}

\section{Random Dynamics on the Complex Plane: \\ the Original Random Dynamical system
$$
f^n_\omega
:=f_{\theta^{n-1}\omega}\circ\cdots\circ f_{\theta\omega}\circ f_\omega:\C\longrightarrow\C.
$$}\label{apendix1}
In this section we will show that both random dynamical systems 
$$
f^n_\omega
:=f_{\theta^{n-1}\omega}\circ\cdots\circ f_{\theta\omega}\circ f_\omega:\C^*\longrightarrow \C^*
$$
and 
$$
F^n_\omega
:=F_{\theta^{n-1}\omega}\circ\cdots\circ F_{\theta\omega}\circ F_\omega:Q\longrightarrow Q,
$$
$\om\in\Om$ are conjugate via conformal (bi-holomorphic) homeomorphisms. As throughout the whole paper, we start with a single exponential map 
$$
f_\eta:\C\lra\C
$$
given by the formula
$$
f_\eta(z)=\eta e^z.
$$
Let 
$$
\C^*:=\C\sms\{0\}.
$$
Since 
$$
\exp:\C\lra\C^*
$$
is a quotient map and $f_\eta$ is constant on each set $\exp^{-1}(z)$, $z\in\C^*$, the map $f_\eta$ induces a unique continuous map
$$
\tilde F_\eta:\C^*\to\C^*
$$
such that the following diagram commutes
\[\begin{tikzcd}
\C \arrow{r}{f_\eta} \arrow[swap]{d}{\exp} & \C\arrow{d}{\exp} \\
{\C^*} \arrow{r}{\tilde F_\eta} & {\C^*}
\end{tikzcd}
\]
i.e.
\begin{equation}\label{120180328}
\tilde F_\eta(\exp(z))=\exp(f_\eta(z)).
\end{equation}
The map $\tilde F_\eta$ can be easily calculated:
$$
\tilde F_\eta(z)=\exp (f_\eta(\exp^{-1}(z)))=\exp(\eta z).
$$
Let 
$$
H_\eta:\C\to\C
$$
be the similarity map given by the formula
$$
H_\eta(z)=z/\eta.
$$
Then 
\begin{equation}\label{220180328}
\tilde F_\eta\circ H_\eta(z)
=\exp(\eta(z/\eta))
=\exp(z)
=\frac1{\eta}f_\eta(z)
=H_\eta\circ f_\eta(z).
\end{equation}

This means that the following diagram commutes
\[\begin{tikzcd}
\C \arrow{r}{f_\eta} \arrow[swap]{d}{H_\eta} & \C \arrow{d}{H_\eta} \\
\C \arrow{r}{{\tilde F}_\eta} & \C
\end{tikzcd}
\]
In other words, the maps $\tilde F_\eta$ and $f_\eta$ are conjugate via the  similarity map $H_\eta$. 

As an immediate consequence of all the above we get the following.

\begin{prop}
For every integer $n\ge 1$,
\begin{equation}\label{320180403}
\tilde F_\eta^n\circ\exp=\exp\circ f_\eta^n,
\end{equation}
i.e. the following diagram commutes
\[\begin{tikzcd}
\C \arrow{r}{f_\eta^n} \arrow[swap]{d}{\exp} & \C\arrow{d}{\exp} \\
{\C^*} \arrow{r}{\tilde F_\eta^n} & {\C^*}
\end{tikzcd}
\]
and 
\begin{equation}\label{420180403}
\tilde F_\eta^n\circ H_\eta=H_\eta\circ f_\eta^n,
\end{equation}
i.e. the following diagram commutes
\[\begin{tikzcd}
\C \arrow{r}{f_\eta^n} \arrow[swap]{d}{H_\eta} & \C \arrow{d}{H_\eta} \\
\C \arrow{r}{{\tilde F}_\eta^n} & \C
\end{tikzcd}
\]
\end{prop}
We now pass to the non--autonomous case. This means that we fix an element ${\bf a}\in [A,B]^{\N}$ and we consider the non--autonomous compositions
$$
f^n_{\bf a}
:=f_{a_{n-1}}\circ f_{a_{n-2}}\circ\cdots\circ f_{a_1}\circ f_{a_0}:\C\longrightarrow\C.
$$
and likewise with $\tilde F_{\bf a}^n$. Iterating (non-autonomously) \eqref{120180328} and doing straightforward calculations based on \eqref{220180328}, we get the following.
\begin{prop}\label{p920180406}
For every integer $n\ge 1$,
\begin{equation}\label{120180403}
\tilde F_{\bf a}^n\circ\exp=\exp\circ f^n_{\bf a}.
\end{equation}
and
\begin{equation}\label{220180403}
\tilde F_{\bf a}^n\circ H_{a_1}=H_{a_n}\circ f_{\sigma({\bf a})}^n
\end{equation}
\end{prop}

We need two more ``little'' results. First recall that the map
$$
\exp: Q\to\C^*
$$
naturally defined from the cylinder $Q$ to $\C^*$ is indeed well defined and is holomorphic. We shall prove the following.

\begin{prop}\label{p620180406}
The map $\exp:Q\lra\C^*$
\begin{enumerate}
\item is a conformal/holomorphic homeomorphism; 

\item transfers the Euclidean metric on $Q$ to the 
conformal metric
$$
|d\rho|:=\frac{|dz|}{|z|}
$$ 
on $\C^*$.

\item conjugates $\tilde F_\eta$ and $F_\eta$, i.e.
$$
\tilde F_\eta\circ \exp=\exp\circ F_\eta
$$
and
\item 
$$
\tilde F_\eta^n\circ \exp=\exp\circ F_\eta^n
$$
for every integer $n\ge 0$. In other words, the following diagram commutes

\[
\begin{tikzcd}
Q \arrow{r}{F_\eta^n} \arrow[swap]{d}{\exp} & Q \arrow{d}{\exp} \\
{\C^*} \arrow{r}{{\tilde F}_\eta^n} & {\C^*}
\end{tikzcd}
\]

\noindent \noindent Furthermore:
\item 
$$
\tilde F_{\bf a}^n\circ \exp=\exp\circ F_{\bf a}^n
$$
for every ${\bf a}\in [A,B]^{\N}$. In other words, the following diagram commutes
\[
\begin{tikzcd}
Q \arrow{r}{F_{\bf a}^n} \arrow[swap]{d}{\exp} & Q \arrow{d}{\exp} \\
{\C^*} \arrow{r}{{\tilde F}_{\bf a}^n} & {\C^*}
\end{tikzcd}
\]

\end{enumerate}
\end{prop}
 
\begin{proof} Item (1) is obvious. In order to prove item (3), we calculate
$$
\exp(F_\eta([w])
=\exp(\pi\circ f_\eta\circ \pi^{-1}([w])
=\exp(\pi\circ f_\eta(w))
=\exp(f_\eta(w)=\tilde F\eta(z)
=\tilde F_\eta(\exp ([w]).
$$
So, 
$$
\exp\circ F_\eta=\tilde F_\eta\circ \exp.
$$
Item (4) is a standard consequence of item (3). Likewise, item (5), follows by a straightforward inductive argument based on (3).

Now, we shall prove item (2). If $[w]\in Q$ and $v$ is a tangent vector at $[w]$ with Euclidean length $1$ then it is mapped by $\exp$ to a tangent vector at the point $z=\exp([w])$, whose Euclidean length is equal to  $|\exp'([w])=|z|$. So, the conformal metric on $\C^*$ which makes the bijection $\exp:Q\to\C^*$ an isometry is exactly the one
$$
d\rho:=\frac{|dz|}{|z|}.
$$
\end{proof}


\section{Random Ergodic Theory and Geometry on the Complex Plane: \\ the Original Random Dynamical system
$$
f^n_\omega
:=f_{\theta^{n-1}\omega}\circ\cdots\circ f_{\theta\omega}\circ f_\omega:\C\longrightarrow\C.
$$}\label{apendix2}

We shall now fully transfer all our results concerning random conformal measures, random invariant measures, Hausdorff dimension of fiber radial Julia sets, and asymptotic behavior of Lebesgue typical points to the case of original random system 
 $$
\(\Omega,\mathcal F,m;\, \theta:\Omega\to \Omega;\, \eta:\Om\to[A,B]\)
$$
and induced by it random dynamics 
$$
\(f^n_\omega:\C\to\C\)_{n=0}^\infty, \  \  \om\in \Om,
$$
given by the formula:
$$
f^n_\omega
:=f_{\theta^{n-1}\omega}\circ\cdots\circ f_{\theta\omega}\circ f_\omega:\C\longrightarrow\C.
$$
We start with the following.

\begin{lem}\label{520180406}
Fix $t>1$. If $\nu=\(\nu_\om\)_{\om\in\Om}$ is a random conformal measure for the random conformal system 
$$
F_\om^n:Q\lra Q, \  \om\in\Om, \ n\ge0,
$$
with a measurable function $\lam:\Om\lra (0,+\infty)$ and the standard Euclidean metric, then the random measure
$$
\tilde\nu:=\(\nu_{\om}\circ\exp^{-1}\)_{\om\in\Om}
$$ 
is a random conformal measure for the random conformal system 
$$
\tilde F_\om^n:\C^*\lra \C^*, \  \om\in\Om, \ n\ge0,
$$
with the same measurable function $\lam:\Om\lra (0,+\infty)$ and the Riemannian metric $\rho$ given by formula (2) of Proposition~\ref{p620180406}. The converse is also true.
\end{lem}
\begin{proof} 
Using formula (5) of Proposition~\ref{p620180406} and, of course, the definition of the random measure $\tilde\nu$, we get for every $\om\in\Om$, every integer $n\ge 1$ and every Borel set $A\sbt\C^*$ such that the restricted map $F_\om^n\vert_A$ is 1--to--1, that
\begin{equation}\label{720180406}
\begin{aligned}
\tilde\nu\(\tilde F_\om^n(A)\)
&=\nu\(\exp^{-1}(\tilde F_\om^n(A))\)
 =\nu\(F_\om^n(\exp^{-1}(A)))\) \\
&=\int_{\exp^{-1}(A)}\lam_\om^n\big|\(F_\om^n\)'\big|^t\,d\nu_\om
 =\int_A\lam_\om^n\big|\(F_\om^n\)'\circ\exp^{-1}\big|^t\,d\tilde\nu_\om \\
&=\lam_\om^n\int_A\big|\tilde F_\om^n(z)\big|^{-t}\big|\(\tilde F_\om^n\)'(z)\big|^t|z|^t\,d\tilde\nu_\om(z) \\
&=\lam_\om^n\int_A\big|\(\tilde F_\om^n\)'\big|_\rho^t\,d\tilde\nu_\om.
\end{aligned}
\end{equation}
An analogous calculation gives the converse.
\end{proof}

\begin{lem}\label{l720180406}
Fix $t>1$. If $\tilde\nu=\(\nu_\om\)_{\om\in\Om}$ is a random conformal measure for the random conformal system 
$$
\tilde F_\om^n:\C\lra\C, \  \om\in\Om, \ n\ge0,
$$
with a measurable function $\lam:\Om\lra (0,+\infty)$ and the Riemannian metric $\rho$ given by formula (2) of Proposition~\ref{p620180406}, then the random measure
$$
\hat\nu:=\(\tilde\nu_{\om}\circ H_{\th^{-1}\om}\)_{\om\in\Om}
$$ 
is a random conformal measure for the random conformal system 
$$
f_\om^n:\C\lra \C, \  \om\in\Om, \ n\ge0,
$$
with the same measurable function $\lam:\Om\lra (0,+\infty)$ and the same Riemannian metric $\rho$. The converse is also true.
\end{lem}
\begin{proof} 
First note that if $s\in\C^*$ and $H_s:\C\lra\C$ is the map given by the formula 
$$
H_s(z)=s^{-1}z,
$$
then 
$$
|(H_s)'(z)|_\rho
=|H_s(z)|^{-1}|\cdot|(H_s)'(z)|\cdot|z|
=\frac{|s|}{|z|}\cdot\frac1{|s|}\cdot|z|. 
=1.
$$
Using this formula, the definition of the random measure $\hat\nu$, and formula \eqref{220180403} of Proposition~\ref{p920180406}, we get for every $\om\in\Om$, every integer $n\ge 1$, and every Borel set $A\sbt\C$ such that the restricted map $f_{\th\om}^n\vert_A$ is 1--to--1, that
\begin{equation}\label{820180406}
\begin{aligned}
\hat\nu_{\th^{n+1}\om}\(f_{\th^n\om}^n(A)\)
&=\tilde\nu_{\th^n\om}\(H_{\th^n\om}\(f_{\th^n\om}^n(A)\)\)
 =\tilde\nu_{\th^n\om}\(\tilde F_\om^n(H_\om(A)\) \\
&=\lam_\om^n\int_{H_\om(A)}\big|\(\tilde F_\om^n\)'\big|_\rho^t|
   \,d\tilde\nu_\om
 =\lam_\om^n\int_A\big|\(\tilde F_\om^n\)'\big|_\rho^t\circ H_\om
   \,d\hat\nu_{\th\om} \\
&=\lam_\om^n\int_A\big|\(f_{\th\om}^n\)'\big|_\rho^t
   \,d\hat\nu_{\th\om}.
\end{aligned}
\end{equation}
\end{proof}

\medskip Now we pass to transferring of invariant random measures. This is even easier. We shall prove the following two lemmas.

\begin{lem}\label{l1020180406}
If $\mu=\(\mu_\om\)_{\om\in\Om}$ is an invariant random measure for the random conformal system 
$$
F_\om^n:Q\lra Q, \  \om\in\Om, \ n\ge0,
$$
then the random measure
$$
\tilde\mu:=\(\mu_{\om}\circ\exp^{-1}\)_{\om\in\Om}
$$ 
is an invariant random measure for the random conformal system 
$$
\tilde F_\om^n:\C^*\lra \C^*, \  \om\in\Om, \ n\ge0,
$$
The converse is also true.
\end{lem}
\begin{proof}
The proof is an immediate consequence of Proposition~\ref{p620180406} (5).
\end{proof}

\begin{lem}\label{l120180410}
If $\tilde\mu=\(\tilde\mu_\om\)_{\om\in\Om}$ is an invariant random measure for the random conformal system 
$$
\tilde F_\om^n:\C\lra \C, \  \om\in\Om, \ n\ge0,
$$
then the random measure
$$
\hat\mu:=\(\tilde\mu_{\om}\circ H_{\th^{-1}\om}\)_{\om\in\Om}
$$ 
is an invariant random measure for the random conformal system 
$$
f_\om^n:\C\lra \C, \  \om\in\Om, \ n\ge0,
$$
The converse is also true.
\end{lem}

\begin{proof}
The proof is carried through by an explicate direct calculation based on formula \eqref{220180403} of Proposition~\ref{p920180406}.
$$
\hat\mu_{\th\om}\circ f_{\th\om}^{-n}
=\tilde\mu_{\om}\circ H_\om\circ f_{\th\om}^{-n}
=\tilde\mu_{\th^n\om}\circ H_{\th^n\om}
=\hat\mu_{\th^{n+1}\om}.
$$
\end{proof}

As an immediate consequence of all the lemmas proven in this section and Theorem~\ref{t1_2016_10_08} along with Theorem~\ref{t1im3B}, we get the following.

\begin{thm}\label{t220180410}
For every $t>1$ there exists a random $t$--conformal measure $\hat\nu^{(t)}$, the one resulting from Theorem~\ref{1_2017_03_28}, Lemma~\ref{520180406} and Lemma~\ref{l720180406}, for the random conformal system 
$$
f_\om^n:\C\lra \C, \  \om\in\Om, \ n\ge0,
$$
with with respect to the Riemannian metric $\rho$ defined in item (2) of Proposition~\ref{p620180406}. This means that formula \eqref{820180406} holds. 

Furthermore, there exists a Borel probability $f$--invariant measure $\hat\mu=\hat\mu^{(t)}$ absolutely continuous with respect to $\hat\nu^{(t)}$. It has the following further properties. 

\medskip\begin{itemize}
\item[(a)] $\hat\mu^{(t)}$ is equivalent to $\hat\nu^{(t)}$,

\medskip\item[(b)] $\hat\mu^{(t)}$ is ergodic.

\medskip\item[(c)] $\hat\mu^{(t)}
$ is is the only Borel probability $f$--invariant measure absolutely continuous with respect to $\hat\nu^{(t)}$. 
\end{itemize}
\end{thm}

\medskip Turning to geometry, we now define random radial (conical) Julia sets on the complex plane $\C$ for the random conformal system 
$$
f_\om^n:\C\lra \C, \  \om\in\Om, \ n\ge0.
$$
These sets are defined analogously as the radial random sets for the random conformal system 
$$
F_\om^n:Q\lra Q, \  \om\in\Om, \ n\ge0.
$$
The definition follows.
\begin{equation}
J_r(f)(\om)
:=\big\{z\in\C: \lim_{N\to\infty}\un\rho(N_\om(z,N))=1\big\},
\end{equation}
where $N_\om(z,N)$ is the set of all integers $n\ge 0$ such that there exists a (unique) holomorphic inverse branch 
$$
f_{\om,z}^{-n}:B(f_\om^n(z),2/N)\lra\C
$$
of $f_\om^n:\C\to\C$ sending $f_\om^n(z)$ to $z$ and such that $|F_\om^n(z)|\le N$. 
The set $J_r(f)(\om)$ is said to be the set of radial (or conical) points of $f$ at $\om$. We further denote:
$$
J_r(f):=\bu_{\om\in\Om}\{\om\}\times J_r(f)(\om).
$$
Based on the propositions proved in this section, it is easy to prove that for every $\om\in\Om$,
\begin{equation}
J_r(f)(\om)
= H_{\th^{-1}\om}^{-1}\circ\exp\(J_r(\th^{-1}\om)\).
\end{equation}
Having this and all the propositions proved in this section, as an immediate consequence of Theorem~\ref{thm:bowen}, Theorem~\ref{t1_2017-09_04}, Theorem~\ref{t320180410}, and Corollary\ref{c420180410}, we get the following.

\begin{thm}\label{thm:bowenB}
For the random conformal system 
$$
f_\om^n:\C\lra \C, \  \om\in\Om, \ n\ge0.
$$
we have that
\begin{enumerate}
\item
$$
\HD\(J_r(f)(\omega)\)=h
$$ 
for $m$--a.e.$\omega\in\Omega$, where $h\in(1,2)$ is the number coming from item (4) of Theorem~\ref{thm:bowen}. In particular:

\item The $2$--dimensional Lebesgue measure of $m$--a.e. $\om\in\Om$ set $J_{r,\omega}$ is equal to zero.

\item For $m$--almost every $\omega\in\Omega$ there exists a subset $\C_\omega\subset\C$ with full Lebesgue measure such that for all $z\in\C_\omega$  the following holds. 
$$
\begin{aligned}
\forall \delta>0 \ \exists n_z(\delta)\in\mathbb N \ &\forall n\ge n_z(\delta)\  \exists k=k_n(z)\ge 0
\\  
&|f^n_\omega(z)-f^k_{\theta^{n-k}\omega}(0)|<\delta\  \  \text{or}\  \  |f^n_\omega(z)|\ge 1/\delta.
\end{aligned}
$$
In addition, $\limsup_{n\to\infty}k_n(z)=+\infty$. In consequence,

\item The set of accumulation points of the sequence
$$
\(f_\om^n(z)\)_{n=0}^\infty
$$
is contained in $[0,+\infty]\cup\{-\infty\}$ and contains $+\infty$. 
\end{enumerate}
\end{thm}

\

\

\










\end{document}